\documentclass{amsart}
\usepackage{amssymb}
\usepackage{hyperref}

\newtheorem{thm}{Theorem}[section]

\newtheorem{lem}[thm]{Lemma}
\newtheorem{prop}[thm]{Proposition}
\theoremstyle{definition}
\newtheorem{defn}[thm]{Definition}
\theoremstyle{remark}

\numberwithin{equation}{section}
\numberwithin{thm}{section}

\newcommand{\Z}{{\mathbb{Z}}}
\newcommand{\C}{{\mathbb{C}}}
\newcommand{\R}{{\mathbb{R}}}

\newcommand{\T}{{\mathbb{T}}}

\DeclareMathOperator*{\dist}{dist}
\newcommand{\eps}{{\varepsilon}}

\newcommand{\lsm}{\lesssim}

\DeclareMathOperator{\supp}{supp}
\DeclareMathOperator*{\wlim}{w-lim}

\newcommand{\qtq}[1]{\quad\text{#1}\quad}
\newcommand{\pd}{\mathcal P_D}
\newcommand{\pds}{P_D\mbox{NLS}}
\newcommand{\nalf}{\mbox{NLS}_\alpha}

\newcommand{\cha}{\chi_n^0}
\newcommand{\chb}{\chi_n^1}
\newcommand{\chc}{\chi_n^2}
\newcommand{\chd}{\chi_n^3}
\newcommand{\che}{\chi_n^4}
\newcommand{\chj}{\chi_n^j}
\newcommand{\pl}{P_{lo}}
\newcommand{\ph}{P_{hi}}
\newcommand{\pmd}{P_{med}}
\newcommand{\pmn}{\mathcal P_{M_n}}
\newcommand{\tu}{\tilde u}
\newcommand{\ul}{u_{lo}}
\newcommand{\uh}{u_{hi}}
\newcommand{\tnj}{\tilde x_n^j}
\newcommand{\pnns}{ P_{M_n}\mbox{NLS}}
\newcommand{\pml}{\mathcal P_{M_n}^{L_n}}
\newcommand{\tun}{\tilde u_n}
\newcommand{\D}{\mathcal D}

\newcommand{\tp}{([-T,T]\times\btn)}
\newcommand{\tr}{([-T,T]\times\R^2)}
\newcommand{\ftr}{L_{t,x}^ 4(I\times\R^2)}
\newcommand{\atr}{L_{t,x}^{\frac 43}(I\times\R^2)}
\newcommand{\ir}{(I\times\R^2)}
\newcommand{\Po}{\mathcal P}
\newcommand{\pmll}{P_{\le 2DM_n}^{L_n}}
\newcommand{\btn}{\mathbb{T}_n}

\let\Re=\undefined\DeclareMathOperator*{\Re}{Re}
\let\Im=\undefined\DeclareMathOperator*{\Im}{Im}

\newcounter{smalllist}

\title[Non-squeezing for the cubic NLS on $\R^2$]{Finite-dimensional approximation and non-squeezing for the cubic nonlinear Schr\"odinger equation on $\R^2$}

\author[R. Killip]{Rowan Killip}
\address{Department of Mathematics, University of California, Los Angeles, CA 90095}%
\email{killip@math.ucla.edu}

\author[M. Visan]{Monica Visan}
\address{Department of Mathematics, University of California, Los Angeles, CA 90095}
\email{visan@math.ucla.edu}

\author[X. Zhang]{Xiaoyi Zhang}
\address{Department of Mathematics, University of Iowa, Iowa City, IA 52242}%
\email{xiaoyi-zhang@uiowa.edu}
{\normalsize }
\begin{document}

\begin{abstract}
We prove symplectic non-squeezing (in the sense of Gromov) for the cubic nonlinear Schr\"odinger equation on $\R^2$.  This is the first symplectic non-squeezing result for a Hamiltonian PDE in infinite volume.  As the underlying symplectic Hilbert space is $L^2(\R^2)$, this requires working with initial data in this space.  This space also happens to be scaling-critical for this equation.  Thus, we also obtain the first unconditional symplectic non-squeezing result in such a critical setting.

More generally, we show that solutions of this PDE can be approximated by a finite-dimensional Hamiltonian system, despite the wealth of non-compact symmetries: scaling, translation, and Galilei boosts.  This approximation result holds uniformly on bounded sets of initial data.  Complementing this approximation result, we show that all solutions of the finite-dimensional Hamiltonian system can be approximated by the full PDE.

A key ingredient in these proofs is the development of a general methodology for obtaining uniform global space-time bounds for suitable Fourier truncations of dispersive PDE models.
\end{abstract}

\maketitle

\section{Introduction}

In this paper, we consider the defocusing cubic nonlinear Schr\"odinger equation in two spatial dimensions:
\begin{align}\label{NLS}
iu_t= - \Delta u + |u|^2 u, \tag{NLS}
\end{align}
which describes the evolution (in time) of a function $u:\R_t^{ }\times\R^2_x\to \C$.

Formally, at least, this is the Hamiltonian evolution associated to the standard symplectic structure on $L^2(\R^2)$ and the Hamiltonian
\begin{align}\label{Hamiltonian}
E(v) := \int_{\R^2} \tfrac12 |\nabla v(x)|^2  +  \tfrac14 |v(x)|^4 \,dx.
\end{align}
As it will be important later in the discussion, let us quickly recall what is meant by the standard symplectic structure:

\begin{defn}
The \emph{standard symplectic structure} on a complex Hilbert space $\mathcal H$ is
$
\omega(z,\zeta) = -\Im \langle z,\zeta\rangle_{\mathcal H}.
$
(Note that in this paper, inner-products are $\C$-linear in the second variable, following Dirac's convention.)  Equivalently, given coordinates $x_k,p_k:\mathcal H\to\R$ adapted to an orthonormal basis $\{e_k\}$ in the sense that
$
x_k(v)+ip_k(v) = \langle e_k, v\rangle,
$
then
\begin{align}\label{E:canoncoord}
\omega = \sum_k dp_k\wedge dx_k.
\end{align}
Coordinate systems in which $\omega$ takes this explicit form are known as \emph{canonical coordinates}.
\end{defn}

The problem of well-posedness of \eqref{NLS} for general data in $L^2(\R^2)$ was a long-standing and conspicuous problem in dispersive PDE until its recent resolution by Dodson \cite{Dodson} (see also \cite{KTV} for the case of radial data):

\begin{thm}[Dodson, \cite{Dodson}]\label{T:Dodson}
The problem \eqref{NLS} is globally well-posed for initial data $u_0\in L^2(\R^2);$ moreover, there is a function $C:[0,\infty)\to[0,\infty)$ so that
$$
\| u \|_{L^4_{t,x}(\R\times\R^2)} \leq C\left( \|u_0\|_{L^2(\R^2)} \right).
$$
\end{thm}

For most purposes, this may be regarded as the definitive result in this direction; by providing uniform space-time bounds, it further guarantees that solutions scatter (see \cite[Theorem~2.1]{CazenaveWeissler}) and are stable under the presence of an additional forcing term in the equation (see \cite[Lemma~3.6]{Matador}).
  
In this paper, we consider two questions about (NLS) for which Theorem~\ref{T:Dodson} does not suffice.  The first question is whether one can approximate the evolution of an ensemble of initial data (specifically, a ball in $L^2$) \emph{simultaneously} by a \emph{single} finite-dimensional system.  The second is the question is whether the (NLS) flow obeys symplectic non-squeezing (in the sense of Gromov, \cite{Gromov}).  We answer both questions in the affirmative; see Theorems~\ref{T:fad} and~\ref{thm:nsqz}, respectively.

In Theorem~\ref{T:dad}, we will also show a converse to our finite-dimensional approximation result, namely, that \emph{all} solutions to the finite-dimensional system can be mimicked by solutions to (NLS).  Again, this is not simply a direct corollary of Theorem~\ref{T:Dodson}.

A key motivation for wishing to connect solutions to PDE to those of a finite-dimensional system (in either direction) is that it allows one to transfer general results for finite-dimensional systems to the PDE setting, with non-squeezing being just one example.  A second incentive for seeking such approximation results is that they guarantee that one can reliably and efficiently simulate the dynamics numerically --- not only the dynamics of individual trajectories, but even statistical questions, that is, those relating to ensembles of solutions.  Moreover, given the increasing interest in the study of dispersive PDE with random initial data, we suspect there will be a growing need for methods adapted to ensembles of data, rather than merely individual solutions.

The particular challenge of proving non-squeezing for Hamiltonian PDE has been a major stimulus for work on finite-dimensional approximations.  For this reason, we first focus our attention on this subject, beginning with an account of the original theorem of Gromov.

To give proper context for Gromov's non-squeezing theorem, we must first review some rudimentary Hamiltonian mechanics.  We speak here about the finite-dimensional setting, where everything is well settled; little is clear cut in the infinite-dimensional setting. Indeed, while nineteenth century mathematics is ample for a full understanding of well-posedness of Hamiltonian ODE, questions of well-posedness of Hamiltonian PDE are, even today, the subject of much intensive investigation.

Let us endow $\C^n$ with the standard symplectic structure laid out above.  A Hamiltonian function $H:\C^n\to \R$ then gives rise to dynamics in the form $\dot z = X_H(z)$, where the vector-field $X_H$ is uniquely determined by the relation
$$
d H(\cdot) = \omega(\cdot,X_H).
$$
In canonical coordinates, this can be expressed equivalently as Hamilton's equations:
$$
\frac{d x_k}{dt} = \frac{\partial H}{\partial p_k} \qtq{and}
\frac{d p_k}{dt} = - \frac{\partial H}{\partial x_k}.
$$

For reasonable $H$, these equations are well-posed and the resulting flow map $\Phi:\R\times\C^n\to\C$ is a one-parameter group of symplectomorphisms.  A symplectomorphism is a diffeomorphism that preserves the symplectic form $\omega$.  In differential form, this means that $\Phi(t)^*\omega = \omega$; however, we find that the integral form expresses this notion more vividly: the integrals of $\omega$ over surfaces $\sigma:[0,1]^2\to \C^n$ are unchanged as the surfaces evolve with the flow:
$$
\int_\sigma \omega = \int_{\Phi(t)\circ\sigma} \omega.
$$
For discussions in infinite dimensions, see, for example, \cite{ChernoffMarsden,KuksinBook}.

Preservation of the symplectic form by Hamiltonian flows guarantees the preservation of its exterior powers, in particular, preservation of the volume form $\omega\wedge\omega\wedge\cdots\wedge\omega$.  Thus Hamiltonian flows preserve phase-space volume, a result known as Liouville's Theorem.  (Incidentally, as there is no translation-invariant measure in infinite dimensions, there can be no Liouville's Theorem, either.)

Liouville's Theorem yields significant information on whether a Hamiltonian flow can carry one region of phase-space into another: the volume of the latter must exceed that of the former.  In fact, it is shown in \cite{DacMoser,Moser} that (barring topological obstructions) this is the only restriction on the possibility of flowing one region into another in a volume preserving manner.

One can also consider the question of whether preservation of volume is the only obstruction to the existence of a symplectomorphism between two sets.  In two dimensions, a symplectic form is a volume form and so this question may be regarded as adequately resolved in view of the results of \cite{DacMoser,Moser}.   In higher dimensions, however, the question remains whether preservation of the symplectic form encodes substantial restrictions beyond mere volume conservation.  The non-squeezing theorem of Gromov (see Corollary \S0.3A of \cite{Gromov}) shows that it does in a dramatic fashion.  For our later purposes, it will be more convenient to phrase this result as follows:

\begin{thm}[Gromov, \cite{Gromov}]\label{T:Gromov}
Fix $0<r<R<\infty$ and $\alpha\in\C$.  Let $B(z_*,R)$ denote the ball of radius $R$ centered at $z_*\in\C^{n}$, let $\ell\in\C^n$ have unit length, and suppose $\phi:B(z_*,R)\to\C^{n}$ is a smooth symplectomorphism (with respect to the standard structure).  Then there exists $z\in B(z_*,R)$ so that
\begin{align*}
\bigl| \langle \ell,\phi(z)\rangle - \alpha \bigr| > r.
\end{align*}
Equivalently, $\phi$ does not map the ball $B(z_*,R)$ wholly inside the cylinder $\{\zeta \in \C^n : |\langle \ell,\zeta\rangle -\alpha | <r\}$ (despite the fact that the volume of the ball is finite and the volume of the cylinder is infinite).
\end{thm}

Note that the real and imaginary parts of $\zeta\mapsto \langle \ell,\zeta\rangle$ provide a pair of canonically conjugate variables --- the position and momentum of a single particle.

From a statistical point of view, the non-squeezing theorem shows that classical mechanics propagates uncertainty in a manner reminiscent of the uncertainty principle in quantum mechanics:  If at the initial time, the positions and momenta of all particles are known only to a resolution $R$, then the position and momentum of even just a single particle cannot be resolved to any scale $r$ finer than $R$ at any later time.

Incidentally, a linear transformation is symplectic if and only if it is invertible and both it and its inverse have the non-squeezing property stated above.  By applying this fact on each tangent space, one then sees that a diffeomorphism, which together with its inverse has the non-squeezing property, is actually a symplectomorphism.  In this way, one can view non-squeezing as a defining property of \emph{all} symplectomorphisms.   Amongst other things, the pursuit of this perspective has lead to a proof that merely \emph{uniform} limits of symplectomorphisms are themselves symplectomorphisms.  For these assertions (and much more), see, for example, \cite{HofZehn,McDuffSalomon}  

In this paper, we will show that the analogue of Gromov's theorem holds for the (infinite-dimensional) dynamics associated to (NLS):

\begin{thm}[Non-squeezing for the cubic NLS]\label{thm:nsqz}
Fix $z_*\in L^2(\R^2)$, $l\in L^2(\R^2)$ with $\|l\|_2=1$, $\alpha\in \C$, $0<r<R<\infty$ and $T>0$.  Then there exists $u_0\in B(z_*, R)$ such that the solution $u$ to \eqref{NLS} with initial data $u(0)=u_0$ satisfies
\begin{align}\label{th:1}
|\langle l, u(T)\rangle-\alpha|>r.
\end{align}
\end{thm}

The initial stimulus to consider this problem came from the authors' attendance of the lecture \cite{Mendelson:talk} by Mendelson.  All non-squeezing results for nonlinear PDE up to that time (and indeed until the work presented here) were for problems posed on tori; see \cite{Bourg:approx,Bourg:aspects,CKSTT:squeeze,HongKwon,Kuksin,Mendelson,Roumegoux}.  At the end of this talk the question was raised whether this represented an intrinsic limitation --- certainly the torus assumption permeates every aspect of these prior works --- or whether one might also expect non-squeezing to hold in infinite-volume settings.  Our result demonstrates the second alternative.

In his paper \cite{Kuksin}, Kuksin argues that non-squeezing gives us insight into what kind of weakly turbulent behavior is possible for Hamiltonian PDE.  Specifically, it precludes the possibility of all the energy evacuating the low or middle frequencies uniformly for a ball of initial data.  We contend that Gromov-type non-squeezing is equally informative about the nature of scattering.  Indeed, the intuitive notion of scattering is precisely that of mass/energy evacuating any compact set and moving off to infinity in physical (as opposed to Fourier) space.  To view the non-squeezing theorem as a constraint on the nature of scattering, one must then prove it for a system that exhibits scattering, which in turn necessitates consideration of a problem in infinite volume.

Several further considerations informed our exact choice of model in pursuing the extension of non-squeezing to an infinite-volume setting.  First, the Schr\"odinger equation does not enjoy finite speed of propagation and correspondingly, even the local dynamics is affected by the global geometry.  Second, in this model solutions scatter, that is, they asymptotically approach linear solutions as $t\to\infty$; more precisely, they asymptotically approach \emph{out-going} waves.  This leads to the suggestion that squeezing may occur --- indeed, all of phase-space is being compressed into a proper subset, namely the out-going waves.  Let us give a simple example that exemplifies this attitude; here we will temporarily adopt the Lax--Phillips paradigm (cf. \cite{LaxPhillips}) for scattering, in which there is a symplectic transformation which converts the nonlinear dynamics into simply that of translation:  The symplectic (indeed unitary) flow
$$
\phi:\R\times L^2(\R) \to L^2(\R) \qtq{with} \phi(t,f)(x) = f(x-t)
$$
preserves the unit ball, as well as the set
$$
A:=\{ f\in L^2(\R) : \textstyle\int_0^\infty |f(x)|^2\,dx \geq \int^{-\infty}_0 |f(x)|^2\,dx \}; 
$$
moreover, for all $f\in L^2(\R)$ there exists $T(f)$ so that $\phi(t,f)\in A$ for all $t\geq T(f)$.  Thus, informally speaking, this flow compresses the unit ball into half of itself (as $t\to\infty)$; this is impossible in finite dimensions due to volume conservation.   For another example of how infinite dimensions allows for behaviors commensurate with the colloquial meaning of `squeezing', that are not possible in finite dimensions, see \cite[\S4]{AbbMajer}.

A third reason for considering the cubic NLS in two dimensions is criticality; specifically, the symplectic phase-space on which this equation is Hamiltonian is $L^2$, which is also a scale-invariant space for the equation.  In particular, among all $\dot H^s(\R^2)$ spaces, $s\in\R$, this is the largest in which the equation is even locally well-posed (see \cite{ChCoTao}).

We are not the first to consider non-squeezing at the scaling-critical regularity.  This major advance was made in Mendelson's work \cite{Mendelson} on the cubic nonlinear Klein--Gordon equation on the three-torus.  Currently, this equation is only known to be \emph{locally} well-posed in the critical space and ultimately, Mendelson's results are conditional, unless one restricts to very short times $T$.  Note that she assumes not only global well-posedness of her equation (in a form directly analogous to Theorem~\ref{T:Dodson} above), but also well-posedness of frequency-truncated forms of the equation.

By elementary arguments using scaling, finite speed of propagation, and partitions of unity, one can show that in Medelson's setting, the assumptions made on the frequency-localized versions of the equation actually imply well-posedness of the full equation.  Naturally, one would like to proceed in the opposite manner and obtain a result that is contingent only on well-posedness of the full equation.   This is difficult.  In our setting, for example, Dodson's proof of Theorem~\ref{T:Dodson} does not apply to frequency-localized versions of the equation --- the localization ruins the Morawetz monotonicity formula at its core.  Moreover, the localized versions cannot be handled as perturbations of the full equation.  Nonetheless, the methods introduced in this paper can be used to verify Theorem~1.2 in \cite{Mendelson} assuming only her hypothesis on the original equation; however, our argument does \emph{not} proceed via this path.

A proper explanation of how the criticality of the problem makes it more difficult to prove non-squeezing will have to wait until we delve deeper into the details.  The simplest aspect of this is that scaling --- indeed, any non-compact symmetry --- ruins compactness of the nonlinear effects, which, in turn, makes it more difficult to approximate the dynamics by a finite-dimensional system.

Our fourth and final reason for selecting the cubic NLS in two dimensions is its popularity both mathematically and as a model in physics (optics in particular). Our arguments apply also to the mass-critical problem in other dimensions; however, we felt that presenting such generality would muddy the exposition.  In dimensions three and higher, this leads to non-squeezing for some low-regularity symplectomorphisms; indeed, the flow map for the mass-critical NLS is merely $C^{1,4/d}(L^2)$ for $d\geq 5$, $C^{1,1}(L^2)$ for $d=4$, and $C^{2,1/3}(L^2)$ for $d=3$. 

\subsection{Prior Work}
Let us now describe prior work on non-squeezing for Hamiltonian PDE in more detail, which will also permit us to introduce some of the fundamental ideas associated with this problem.  The subject begins with the work \cite{Kuksin} of Kuksin.  His approach was to develop a variant of Gromov's theorem in Hilbert space and then verify the hypotheses of his theorem for several PDE examples.  Ultimately, his method is based on finite-dimensional approximation; specifically, he extends the symplectic capacity of Hofer--Zehnder to the Hilbert space setting via finite-dimensional approximation and then shows that this notion of capacity is preserved, provided the nonlinear component of the dynamics is sufficiently weak.

The specific assumptions on the dynamics imposed by Kuksin to implement this strategy can be summarized as follows:  The underlying linear dynamics should be unitary on the symplectic Hilbert space $Z$ and have discrete spectrum.  The nonlinearity should be compact in the following strong sense: it can be extended continuously to a mapping $Z_-\to Z_+$ where $Z_\pm$ are Hilbert spaces adapted to an eigenbasis of the linear part and admitting compact embeddings $Z_+\hookrightarrow Z\hookrightarrow Z_-$.

Kuskin's work inspired several other authors who extended considerably the list of models to which his abstract result could be applied; see \cite{Bourg:aspects,Roumegoux} for results of this type.

In Kuksin's paper (and in all the subsequent work) the linear functional $\ell$ appearing in Theorem~\ref{thm:nsqz} is restricted to be an eigenvector of the underlying linear dynamics.  Upgrading these prior works to the generality presented in our theorem is relatively easy.  On the other hand, since the Laplacian has purely continuous spectrum on $\R^2$, there is no analogue of their reduced generality in our setting.  In fact, this spectral property means that the unitary group $e^{it\Delta}$ has no invariant subspaces of finite dimension (excepting $\{0\}$); this raises issues even for approximating the \emph{linear} flow by a finite-dimensional system and underlines another way in which our proof must diverge from its predecessors. 

It is perhaps tempting to imagine that one may define the Hofer--Zehnder capacity directly on subsets of the Hilbert space.  This seems na\"ive.  Recall that this capacity is defined by the maximal total variation of a Hamiltonian function that does not produce an orbit with period $\leq 1$.  The key analytical basis underlying this construction is the fact that in finite dimensions, non-trivial Hamiltonian flows necessarily have closed orbits.  No such result holds in infinite dimensions; indeed, a quadratic Hamiltonian will admit periodic orbits if and only if it is the quadratic form associated to an operator with discrete spectrum.

Promisingly, Abbondandolo and Majer \cite{AbbMajer} have recently constructed a capacity directly on \emph{convex} sets in Hilbert space (without finite-dimensional approximation).  They employ the dual action principle, which is peculiar to the convex case and yields the existence of periodic orbits in the finite-dimensional setting --- they correspond to critical points of the dual action functional.  As noted, periodic orbits need not exist in infinite dimensions; instead, these authors define the capacity directly from the variational problem and prove symplectic invariance by studying Palais--Smale sequences associated to this problem.  There is no reason to believe that nonlinear PDE carry balls of initial data into convex sets, except on very short time intervals.  In this sense, the result of \cite{AbbMajer} is complementary to the PDE development outlined below.

Bourgain \cite{Bourg:approx} was the first to prove non-squeezing for a model that does not fall under Kuksin's framework, considering the cubic NLS
\begin{align}\label{B1dNLS}
i\partial_t u = -\Delta u + |u|^2 u \quad\text{posed on $\R/\Z$,}
\end{align}
that is, on the one-dimensional torus.  Note that the complete integrability of this equation plays no role in his arguments, which apply also in the presence of a smooth coupling coefficient $b(t,x)$ in front of the nonlinearity.

The presence of the non-compact Galilei symmetry suffices to see that Kuksin's compactness hypothesis fails for \eqref{B1dNLS}.  Bourgain proves non-squeezing for this model by showing directly that it can be approximated by a finite-dimensional system, specifically, the cubic NLS where a sharp Fourier cutoff $\Pi_{\leq N}$ is placed on the nonlinearity (and the initial data):
\begin{align}\label{B1dNLS'}
i\partial_t u = -\Delta u + \Pi_{\leq N}\bigl( |u|^2 u \bigr).
\end{align}
The key quantitative estimate used by Bourgain is the following:  For fixed $M,N_0,t$, and $\eps$, all positive, there exists $N\gg N_0$ so that the following holds:  If $u$ and $v$ denote the solutions to \eqref{B1dNLS} and \eqref{B1dNLS'}, respectively, with initial data $\phi=\Pi_{\leq N}\phi$ obeying $\|\phi\|_{L^2}\leq M$, then
\begin{align}\label{Bapprox}
\| \hat u(t,\xi) - \hat v(t,\xi) \|_{\ell^2(|\xi|\leq N_0)} \leq \eps.
\end{align}

Observe that this estimate immediately yields non-squeezing: If squeezing occurred for the full dynamics \eqref{B1dNLS}, then \eqref{Bapprox} guarantees that it also occurs for the finite-dimensional dynamics \eqref{B1dNLS'}, which is forbidden by Theorem~\ref{T:Gromov}.

As $N_0$ in \eqref{Bapprox} is fixed but arbitrary, Bourgain's argument is inherently one of showing that the full flow \eqref{B1dNLS} can be approximated in the weak topology by a sequence of finite-dimensional models of the form \eqref{B1dNLS'}.  Indeed, on the torus, a bounded sequence in $L^2$ converges weakly if and only if their  Fourier transforms converge pointwise.  

The proof of the key estimate \eqref{Bapprox} can be reduced (by induction) to the case where $t$ is small enough that existence on the interval $[0,t]$ follows from a single application of contraction mapping.  Note that the restriction on $t$ depends only on $M$, which is a hall-mark of the subcriticality of the problem.  In this reduced regime, the problem is transformed into a careful analysis of the contraction mapping step; in particular, Bourgain shows that low-frequencies are little affected by much higher-frequencies because such interactions are non-resonant.

Note that the method just outlined does not carry over to the scaling-critical case, because scale invariance forbids the possibility of a local existence time \emph{universal} across an entire ball of initial data (except in the small-data regime).

Broadly speaking, the paper \cite{CKSTT:squeeze} implemented Bourgain's strategy for \eqref{B1dNLS} in the setting of the Korteweg--de Vries equation on the circle; however, in order to do this, the authors had to surmount several major difficulties.

The symplectic Hilbert space on which KdV is Hamiltonian is the space of mean-zero functions subject to the $H^{-1/2}(\R/\Z)$ norm.  While the KdV equation is subcritical in this space, it does represent the end-point regularity for strong notions of well-posedness.  Specifically, the data to solution map is real-analytic on $H^s(\R/\Z)$ for $s\geq-\frac12$ and not even uniformly continuous on bounded sets for $s<-\frac12$; see \cite{ChCoTao,CKSTT:JAMS,CKSTT:JFA04} and the references therein.  On the other hand, well-posedness (with mere continuous dependence) holds down to $H^{-1}$; see \cite{KappTop}.

It is shown in \cite{CKSTT:squeeze} that the introduction of a sharp Fourier cutoff (as in Bourgain's work) does not work; the requisite approximation result fails.  This problem is remedied by using a standard (smooth) Littlewood--Paley projector $P_{\leq N}$ instead; specifically, the Hamiltonian is modified from its usual form
$$
\int \tfrac12 u_x^2 + u^3 \,dx \qtq{to} \int \tfrac12 u_x^2 + (P_{\leq N} u)^3 \,dx.
$$
The latter clearly yields a finite-dimensional Hamiltonian system on the space of functions in the range of $P_{\leq N}$.  The problem is to verify that solutions to the full system can be adequately approximated by this system; the precise formulation is a close analogue of \eqref{Bapprox}.  The authors of \cite{CKSTT:squeeze} do this by revisiting and further strengthening their proof of well-posedness in $H^{-1/2}$; this in turn involves passing to mKdV via the Miura map and upgrading previous trilinear estimates.

More recently, Hong and Kwon, \cite{HongKwon}, have simplified the proof of non-squeezing for KdV on the torus and further shown that non-squeezing holds for a certain system of coupled KdV equations.  The principal innovation here is the discovery of a simple normal form transformation that can be used (in place of the Miura map) to verify finite-dimensional approximation of the same type as before.

Our earlier discussion of Mendelson's work emphasized her conditional result \cite[Theorem~1.2]{Mendelson}.  This paper contains a second principal result concerning non-squeezing (see \cite[Theorem~1.1]{Mendelson}) that is unconditional, but only applies to very short time intervals.  The proof of this unconditional result is rather more complicated and introduces the novel idea of tracking randomized initial data, for which one has more leverage in the well-posedness problem.  On the other hand, we contend that the precise formulation of \cite[Theorem~1.1]{Mendelson} permits a simpler proof, via the expedient of approximating the nonlinear flow by the linear flow, rather than a frequency-localized nonlinear problem.  (Non-squeezing for the linear Klein--Gordon flow is elementary.)  The basic Klein--Gordon estimate for the approach we are advocating is the following:
$$
\bigl\| \bigl(u(t)-u^{\text{lin}}(t),u^{ }_t(t) - u^{\text{lin}}_t(t)\bigr) \bigr\|_{L^\infty_t H^1_x\times L^2_x([0,T]\times \T^3)}
	\lesssim  T \bigl\| \bigl(u(0),u_t(0)\bigr) \bigr\|_{H^1_x\times L^2_x}^3,
$$
which holds for $T$ sufficiently small depending on the $H^1_x\times L^2_x$ norm of the initial data.  Note that the energy space $H^1_x\times L^2_x$ does not coincide with the symplectic space in the Klein--Gordon setting; however, \cite[Theorem~1.1]{Mendelson} considers only frequency-localized data. 

This completes our discussion of prior work.  Given the wealth of work on the torus, it is natural to ask if non-squeezing holds for the cubic NLS on the two-torus $\R^2/\Z^2$, as opposed to $\R^2$, the problem studied here.  In actuality, this question is premature --- it is not known if the cubic NLS is well-posed on $L^2(\R^2/\Z^2)$.  Indeed, even the small data problem remains open.  In \cite{Kishimoto}, Kishimoto shows that if the data to solution map does exist, it cannot be $C^3$ even at the origin.

\subsection{Finite-dimensional approximation}\label{SS:approx}
Our finite-dimensional approximating system will be a frequency-truncated version of \eqref{NLS} posed on a large torus $\R^2/ L\Z^2$.  However, we must delay the formal statement  of this result a little longer, because the precise formulation is dictated by the need to circumvent a key difficulty (stemming from the criticality of our problem), whose nature and resolution we elaborate first.

As in the analysis of numerical schemes, there are two key components to an approximation result: stability and convergence. The former is the assertion that the approximate problem (or numerical method) is stable, namely, that the (putative) solutions it produces obey bounds.  These bounds must be \emph{uniform} as the supposed quality of the scheme improves.  The second step is to show that these well behaved approximate solutions do actually converge to a true solution.

For us, stability is a major issue.  This stems from the criticality of the problem and did not appear in previous work (excepting \cite{Mendelson} where it was merely assumed).  The difficulty appears already if one merely introduces a Fourier cutoff in the cubic nonlinearity.  In fact, for the proof it is convenient to first overcome this particular hurdle in the infinite-volume setting, before even starting to consider the frequency-truncated problem on the torus.  In Theorem~\ref{thm:scapds}, we therefore consider the problem
\begin{align}\label{E:PdNLS}
i\partial_t u = -\Delta u + \mathcal P \bigl( |\mathcal P u|^2 \mathcal P u\bigr) 
\end{align}
for a suitable compactly supported Fourier multiplier with symbol $m(\xi)$ and prove \emph{uniform} space-time bounds that depend only on the mass of the initial data.  Note that this equation is Hamiltonian with energy functional
$$
H(u) = \int \tfrac12 |\nabla u|^2 + \tfrac14 |\mathcal P u|^4\,dx.
$$
Ultimately, we will move the frequency cutoff to higher and higher frequencies by using multipliers $\mathcal P_{M_n}$ with rescaled symbols $m(\xi/M_n)$ and sending $M_n\to\infty$.  However, as our problem is scale-invariant, the stability problem for these multipliers is equivalent to the one for \eqref{E:PdNLS}.  There is one vestige of this rescaling though, namely, we must study \eqref{E:PdNLS} \emph{globally} in time.  Note that global existence is trivial for \eqref{E:PdNLS}.  What we need, however, are uniform global space-time bounds.

Because we have an adequate perturbation theory for \eqref{E:PdNLS}, see Lemma~\ref{lm:stab}, a proof of global space-time bounds for \eqref{E:PdNLS} would imply Theorem~\ref{T:Dodson}.  Indeed, \eqref{E:PdNLS} becomes (NLS) in the low-frequency limit.   The converse implication is untrue --- the multiplier has a meaningful effect on solutions that have their Fourier support concentrated in the region where $\mathcal P$ is transitioning from the identity to zero.  Remember that to prove non-squeezing, we must approximate \emph{all} solutions (in some ball) \emph{simultaneously}.  We do not know how to prove such bounds for \eqref{E:PdNLS} when $\mathcal P$ is a normal Littlewood--Paley projector, even in the radial case; this choice breaks the Morawetz monotonicity formula at the heart of the proof of Theorem~\ref{T:Dodson}.

In Section~\ref{S:4} we present a robust scheme (based on an inductive argument) to show that \emph{suitably} frequency-localized versions of an equation inherit space-time bounds from the original equation.  The key such stability result in the case of \eqref{E:PdNLS} is Theorem~\ref{thm:scapds}.

The first component of Theorem~\ref{thm:scapds} is the choice of the symbol $m(\xi)$ of $\mathcal P$.  It will transition very very slowly from its value one near the origin to vanishing near infinity.  In fact, the rate required is dictated by the maximal mass of the initial data under consideration via the constant $C(M)$ appearing in Dodson's Theorem.  The rationale for this is that solutions which remain concentrated (on the Fourier side) will experience the multiplier as a single number.  This is helpful since uniform global bounds for (NLS) with modified coupling constant $\alpha^4\in[0,1]$ do follow directly from Theorem~\ref{T:Dodson}.  Ultimately, we will convert this heuristic into a proof of space-time bounds for initial data that is well localized on the Fourier side; see Proposition~\ref{prop:onebubble}.

The second component of Theorem~\ref{thm:scapds} is an induction argument, reminiscent of that introduced in \cite{Bourg:JAMS}, to treat the case of Fourier delocalized initial data; see Proposition~\ref{prop:twobubble}.  For well-posedness problems, the original argument of Bourgain has mostly been displaced by the methodology of Kenig and Merle, \cite{KenigMerle:H1}.  Recall that the latter proceeds as follows: one first shows the existence of a minimal blowup solution (which necessarily will have good compactness properties) via the concentration compactness techniques of \cite{BG99,Keraani} and subsequently, one shows that such almost periodic solutions are incompatible with known monotonicity formulas and/or conservation laws.  This is the style of argument used by Dodson in \cite{Dodson}.  As we explain in Section~\ref{S:4}, concentration compactness arguments do not aid in the proof of Theorem~\ref{thm:scapds}; moreover, we will make no use of any monotonicity formulas for \eqref{E:PdNLS} --- none are known --- nor any of the conservation laws associated with it.  Even conservation of mass plays no role, despite being the quantity on which we induct.

Let us now make a few definitions as a last prerequisite to the formulation of our approximation result.  The approximating sequence of systems are of the following form:  Given a fixed parameter $D$ and sequences of parameters $L_n$ and $M_n$ we consider
\begin{align}\label{approx system}
i\partial_t u = -\Delta u + \mathcal P_{M_n}^{L_n} \bigl( |\mathcal P_{M_n}^{L_n} u|^2 \mathcal P_{M_n}^{L_n} u\bigr) 
\end{align}
posed on the finite-dimensional Hilbert space
$$
\mathcal H_n:=\{f\in L^2(\T_n):\, P_{>2DM_n}^{L_n} f =0 \} \qtq{with} \T_n := \R^2/L_n \Z^2.
$$
Here $\mathcal P_{M_n}^{L_n}$ denotes the Fourier multiplier on $\T_n$ with symbol $m_D(\xi/M_n)$ defined explicitly in \eqref{m_D definition} and $P_{>2DM_n}^{L_n}$ denotes the usual Littlewood--Paley projection operator on the torus $\T_n$.

To general $\ell\in L^2(\R^2)$, which will be regarded as a linear functional on $L^2(\R^2)$, we will need to associate a linear functional on $\mathcal H_n \subseteq L^2(\T_n)$.  To do this we assume  that $\ell$ has compact support and then define $p_*\ell \in L^2(\T_n)$ via the covering map $p:\R^2\to\T_n$; see \eqref{covering shit}.  As $L_n\to\infty$ and compactly supported functions are dense in $L^2(\R^2)$, the reader should not take this formal necessity too seriously.

Our two approximation results are as follows:

\begin{thm}[Finite-dimensional approximation]\label{T:fad}  Given $M>0$, $T>0$, and $M_n\to\infty$, there exists $D>0$ and $L_n\to\infty$ so that the following holds:  Let $u_n$ be any sequence of solutions to \emph{(NLS)} obeying $\|u_n(0)\|_{L^2} \leq M$; then there exist solutions $v_n$ to  \eqref{approx system} with $v_n(0)\in\mathcal H_n$ satisfying $\|v_n(0)\|_{L^2} \leq M$ and
\begin{equation}\label{E:T:1.5}
\bigl| \langle p_*\ell, v_n(t) \rangle_{L^2(\T_n)} - \langle \ell, u_n(t) \rangle_{L^2(\R^2)}\bigr| \longrightarrow 0 \quad\text{as $n\to\infty$}
\end{equation}
for all $-T\leq t\leq T$ and all $\ell\in L^2(\R^2)$ of compact support.
\end{thm}  

We stated this theorem for sequences of solutions $u_n$.  A completely equivalent assertion is that the approximation of solutions to (NLS) by sequences of solutions to \eqref{approx system} holds uniformly across the $M$-ball of initial data, which is to say, convergence in \eqref{E:T:1.5} is uniform on this ball. Without the assertion of uniformity in the choice of initial data obeying $\|u(0)\|_{L^2}\leq M$, this result would be relatively easy to prove.  As we have emphasized previously, the substance here is to be able to simultaneously simulate a non-compact ensemble of trajectories.

Theorem~\ref{T:fad} does not immediately yield Theorem~\ref{thm:nsqz}, in part, because it does not guarantee that all solutions to \eqref{approx system} appear as approximations to solutions to (NLS).  Our proof of the non-squeezing theorem relies precisely on this latter type of approximation; specifically, we show that the dynamics of all solutions to \eqref{approx system} can be approximated (uniformly on balls) by that of solutions to (NLS):

\begin{thm}[Finite-dimensional embedding]\label{T:dad}  Given $M>0$ and $M_n\to\infty$, there exists $D>0$ and $L_n\to\infty$ so that the following holds: Given any sequence of solutions $v_n$ to \eqref{approx system} with initial data $v_n(0)\in\mathcal H_n$ satisfying $\|v_n(0)\|_{L^2}\leq M$, there exist solutions $u_n$ to \textup{(NLS)} with initial data satisfying $\|u_n(0)\|_{L^2}\leq M$ that also obey
$$
\bigl| \langle p_*\ell, v_n(t) \rangle - \langle \ell, u_n(t) \rangle\bigr| \longrightarrow 0 \quad\text{as $n\to\infty$}
$$
for all $t\in\R$ and all $\ell\in L^2(\R^2)$ of compact support.
\end{thm}

Note that no coherence is assumed for $v_n$ as $n$ varies; this encapsulates the underlying uniformity of this embedding of the finite-dimensional systems.  This is essential for our proof of Theorem~\ref{thm:nsqz}, since in that case, the $v_n$ will simply be witnesses to the non-squeezing phenomenon in the finite-dimensional setting.  Since all known proofs of Theorem~\ref{T:Gromov} are obstructive in nature, we cannot say anything at all about these witnesses --- they are merely inside the ball at time $t=0$ and outside the cylinder at time $t=T$.

This embedding theorem is not inherently weaker or stronger than Theorem~\ref{T:fad}.  For purposes of exposition, however, we have chosen to prove Theorem~\ref{T:dad} first and then use it (together with some ingredients of the proof) to verify Theorem~\ref{T:fad}.

\subsection{Outline of Proof}

The proofs of Theorems \ref{thm:nsqz}, \ref{T:fad}, and \ref{T:dad} all rest on the same three pillars:

\noindent
$\;\bullet$ Theorem~\ref{thm:scapds}, showing well-posedness and space-time bounds for the frequency-localized NLS on $\R^2$.  This is a key ingredient in our proof that the approximation by finite-dimensional systems is stable.

\noindent
$\;\bullet$ Theorem~\ref{thm:wc}, which shows the following:  Given a sequence $\tilde u_n$ of solutions to the frequency-localized NLS on $\R^2$ with frequency truncation at height $M_n\to\infty$, whose initial data $\tilde u_n(0)$ converge weakly in $L^2(\R^2)$, then the sequence $\tilde u_n(t)$ converges weakly in $L^2(\R^2)$ at all times $t\in\R$; moreover, the pointwise-in-time weak limit is a solution to (NLS).

\noindent
$\;\bullet$ Theorem~\ref{thm:app}, which shows that solutions to the frequency-localized NLS on the torus can be approximated \emph{in norm} by solutions to the frequency-localized NLS on the whole plane $\R^2$.

Before discussing the character and proof of each of these theorems --- and the sizable corps of preliminary results that this requires --- let us first explain (at least in caricature) how these results lead to a proof of Theorem~\ref{T:dad}; details are given in Section~\ref{S:9}.

Let the sequence $v_n(t)$ of solutions to \eqref{approx system} posed on the torus with $M_n\to\infty$ be given.  By Theorem~\ref{thm:app} there is a sequence of solutions $\tilde u_n(t)$ to \eqref{approx system} posed on the plane that approximates $v_n(t)$ in norm (as $n\to\infty$) at any time $t$.  Suppose now that we pass to a subsequence along which $\tilde u_n(0)$ converges.  In this reduced setting, Theorem~\ref{thm:wc} guarantees that $\tilde u_n(t)$ can be approximated by a single solution to (NLS), at least in the weak topology.  This is then sufficient (by Theorem~\ref{BG-like}) to guarantee that the solutions $u_n(t)$ to (NLS) with initial data $u_n(0)=\tilde u_n(0)$ approximate the solutions $v_n$ in the weak topology.

The above argument did not make direct use of Theorem~\ref{thm:scapds}.  As we discussed in the early part of subsection~\ref{SS:approx}, this is however an essential and non-trivial ingredient in the proof of Theorem~\ref{thm:wc}.  Indeed, one cannot hope to describe the asymptotic behavior of the sequence $\tilde u_n$ without uniform bounds.   Theorem~\ref{thm:scapds} is also essential in the proof of Theorem~\ref{thm:app}.

The proof of Theorem~\ref{thm:scapds} proceeds by the inductive argument outlined earlier.  This occupies the whole Section~\ref{S:4} and uses a number of tools developed in Section~\ref{S:3}.

Theorem~\ref{BG-like} is the analogue of Theorem~\ref{thm:wc} where we consider sequences of solutions $u_n$ to (NLS), rather than to the frequency-localized problem.  It says that if $u_n(0)\rightharpoonup u_{\infty,0}$ and $u_\infty(t)$ is the solution to (NLS) with initial data $u_{\infty,0}$, then $u_n(t)\rightharpoonup u_\infty(t)$ for all $t\in\R$.  (Here all limits are in the weak topology on $L^2$.)  Note that this is precisely the statement that (NLS) is well-posed in the weak topology on $L^2(\R^2)$.

Well-posedness in the weak topology does not follow from well-posedness in norm.  Although strong-to-strong continuity implies weak-to-weak continuity for linear maps, this is no longer true for nonlinear maps: the map $f\mapsto |f|$ acting on $L^2(\R/\Z)$ is a particularly simple example.

Well-posedness in the weak topology was first proved at critical regularity in \cite{BG99}, which treated the energy-critical wave equation.  (In the subcritical setting, such an assertion is much easier to prove due to local compactness.)  Due to criticality, well-posedness in the weak topology actually requires one to develop a structure theorem for sequences of solutions up to a vanishing error in \emph{norm}.  This is precisely what a nonlinear profile decomposition does; indeed, the development of such decompositions is one of the fundamental advances in \cite{BG99}.  A nonlinear profile decomposition for (NLS) was constructed in \cite{Keraani} building on the linear profile decomposition of \cite{MerleVega}.  These tools permit a direct proof of Theorem~\ref{BG-like} along the lines laid out in~\cite{BG99}.

The proof of Theorem~\ref{thm:wc}, in which we are considering a sequence of solutions $\tilde u_n$ that each obey a different equation, is more difficult.  The presence of the Fourier multiplier in the nonlinearity means that the nonlinear profiles are not simply related by symmetries of (NLS); the equation that each profile solves depends on the relation of its intrinsic length scale to the frequency cutoff.  This obstacle is overcome in Section~\ref{S:5}, where the requisite nonlinear profile decomposition is obtained; see Theorem~\ref{T:npd}.  In particular, the reader will see that the broken scale invariance leads to \emph{three} distinct types of nonlinear profiles.

The proof of Theorem~\ref{thm:wc} is completed in Section~\ref{S:6}.  Note that the presence of profiles not built from solutions of (NLS) shows that it is essential that Theorem~\ref{thm:wc} relates to the weak topology; no such approximation is possible in norm. 

This leaves us to discuss the third pillar, namely, Theorem~\ref{thm:app}.  The change in geometry from the torus to the plane necessitates understanding solutions in the norm topology, even if one ultimately only seeks conclusions in the weak topology, just as in the case of the Fourier truncation discussed above.

The first step in connecting the two flows is to connect the initial data.  The obvious procedure here is to simply cut the torus and unwrap; equivalently, to view functions on the torus as functions on the plane supported in a single fundamental domain of the covering map $p:\R^2\to\R^2/L\Z^2$.  There is a problem with this idea: if one cuts the torus in a place where the initial data has a bubble of concentration, this will defeat norm approximation of the corresponding nonlinear solutions.  Indeed, the nonlinear evolution of two half-bubbles does not resemble the nonlinear evolution of the full bubble.  Instead, we choose $L$ very large and use a pigeonhole argument to find a location for cutting that is well-separated from any region of concentration of the initial data.

This argument also shows that it is not merely sufficient that the initial data is minimally affected by the cutting of the torus; this also needs to hold for the nonlinear solution up to time $T$.  This is effected by a careful choice of cutoffs (see subsection~\ref{SS:cutoffs}) and by controlling the motion of mass (cf. Lemma~\ref{lm:sm}).  This is one place where we see why it was essential to pass through the frequency-localized problem in the planar setting \emph{before} attempting to connect the dynamics of \eqref{NLS} to that of the finite-dimensional system.

Having transferred the initial data to the plane, one can then build solutions to the frequency-localized problem there.  One must then lift them back to the torus and verify that the mimic the dynamics there.  This requires overcoming several further difficulties; let us briefly describe some of these:

The frequency-localization operators on the plane and torus are not identical.  Their symbols may be the same, but their kernels are not: one is built from a trigonometric sum, the other by a Fourier integral.  One must link the two; see Lemma~\ref{lm:cls}.  Subsections~\ref{SS:8.2} and~\ref{SS:8.3} also contain a number of lemmas controlling the way in which the two frequency-localization operators interact with cutoff functions. 

By Theorem~\ref{thm:scapds} we know that the solution to the frequency-localized problem on the plane obeys space-time bounds.  By resolving the preceding problems, we are able to show that the lift of the planar solution is almost a solution on the torus, in the sense that it satisfies the equation up to a small error.  One must then show that small errors in the equation lead to small modifications to the solution; this is what is known as perturbation theory.  Traditionally, perturbation theory is developed via Strichartz estimates.  However, the standard mass-critical Strichartz estimates fail on the torus in any dimension (cf. \cite{Bourg:lattice}).  

The role of Section~\ref{S:7} in this paper is to develop Strichartz estimates on $\R^2/\Z^2$ that are scale-invariant and do not lose derivatives, by exploiting the frequency cutoff appearing in our equation.  (Once again, it is essential to pass to the frequency-localized problem before moving to the torus.)  These estimates are then used to obtain a scale-invariant perturbation theory on the torus; see Proposition~\ref{P:stab}.

This brings to a close our outline of the proofs of the three main theorems.  While the arguments of this paper are presented in the concrete setting of the mass-critical NLS in two space dimensions, no peculiar feature of this equation has been exploited in the proofs.  Indeed, in many ways, this particular equation was selected because it seemed most antagonistic to the problem at hand.  We contend that the scheme presented here provides a road map of general applicability for proving finite-dimensional approximation and non-squeezing results in both finite and infinite volume.

\subsection*{Acknowledgements} This material is based on work supported by the National Science Foundation under Grant No. 0932078000 while the authors were in residence at the Mathematical Sciences Research Institute in Berkeley, California, during the Fall 2015 semester.

The authors are grateful to MSRI and to the organizers of the program ``New Challenges in PDE: Deterministic Dynamics and Randomness in High and Infinite Dimensional Systems'', which provided not only the stimulus for this work, but also the perfect environment for its gestation. 

This work was partially supported by a grant from the Simons Foundation (\#342360 to Rowan Killip).
R.~K. was further supported by NSF grants DMS-1265868 and DMS-1600942.  M.~V. was supported by NSF grant DMS-1500707. X.~Z. was supported by a Simons Collaboration grant.

\section{Preliminaries}

Throughout this paper, we will will use
\begin{align*}
F(u):=|u|^2 u
\end{align*}
to denote the cubic nonlinearity.

The development of many estimates on the torus is often much simplified by exploiting its product structure; however, the induction on mass argument in Section~\ref{S:4} benefits even more substantially from choosing spherically symmetric Littlewood--Paley multipliers (rather than a product version).  For this reason we adopt the latter.  To this end, let $\varphi:\R^2\to[0,1]$ be smooth, spherically symmetric, and obey
$$
\varphi(x) = 1 \text{ for $|x|<1.41$} \qtq{and} \varphi(x) = 0 \text{ for $|x|>1.42$.}
$$
We then define Littlewood--Paley projections onto low frequencies according to
\begin{align}\label{E:LP defn}
\widehat{ P_{\leq N} f }(\xi) := \varphi(\xi/N) \hat f(\xi)
\end{align}
and then projections onto individual frequency bands via
\begin{align}\label{E:LP defn'}
f_N := P_N f := [ P_{\leq N} - P_{\leq N/2} ] f .
\end{align}

\begin{defn}[Strichartz spaces] We define the Strichartz norm of a space-time function via
$$
\| u \|_{S(I)} = \| u\|_{C^0_t L^2_x (I\times\R^2)} + \| u\|_{L_t^3 L^6_x (I\times\R^2)}
$$
and the dual norm via
$$
	\| F \|_{N(I)} = \inf_{F=F_1+F_2} \| F_1\|_{L^1_t L^2_x (I\times\R^2)} + \| F_2 \|_{L^{\frac32}_t L^{\frac65}_x (I\times\R^2)}
$$
\end{defn}

\begin{lem}[Bilinear Restriction \cite{Tao:bi}]\label{lm:bl1} Let $I$ be a compact interval.  Assume that $u_1,u_2:I\times\R^2\to \C$ have frequency supports in $\{ |\xi|\le N\}$ that are separated by at least $cN$. Then for $q>\frac53$ we have
\begin{align*}
\|u_1u_2\|_{L_{t,x}^q(I\times\R^2)}\lsm_{c,q} N^{2-\frac4q}\|u_1\|_{S_*(I\times\R^2)}\|u_2\|_{S_*(I\times \R^2)},
\end{align*}
where $S_*$ denotes the strong Strichartz norm which is defined by
\begin{align*}
\|u\|_{S_*(I\times\R^2)}:=\inf_{t_0\in I} \|u(t_0)\|_2+\|(i\partial_t+\Delta)u\|_{N(I)} .
\end{align*}
\end{lem}

For the following, see \cite[Lemma 2.5]{Visan:Duke}, which builds on earlier versions in \cite{borg:book, CKSTT:gwp}.

\begin{lem}[Bilinear Strichartz] \label{lm:bl2}For $u,v:I\times\R^2\to \C$ we have
\begin{align*}
\|u_N v_M\|_{L_{t,x}^2\ir}\lsm\bigl(\tfrac MN\bigr)^{\frac 12}\|u_N\|_{S^0_*\ir}\|v_M\|_{S^0_*\ir}
\end{align*}
whenever $M\leq N$.
\end{lem}

Before we record the linear profile decomposition for the Schr\"odinger propagator for bounded sequences in $L^2$, we first introduce some relevant notations and concepts.   Given parameters $(N_n,  \xi_n, x_n, t_n)\in \R^+\times\R^2\times\R^2\times \R$, we define unitary operators
\begin{align}
[g_n f](x):=[g_{N_n, \xi_n, x_n} f](x):= N_n e^{ix\cdot \xi_n}f(N_n(x-x_n))
\end{align}
and
\begin{align}\label{G defn}
G_n f:=G_{N_n, \xi_n, x_n, t_n} f:= g_n e^{it_n\Delta} f.
\end{align}
These operators are well suited to a discussion of the linear problem.  When we turn later to the nonlinear problem, we will need to employ analogous transformations of space-time functions, namely,
\begin{align}\label{T defn}
[T_n v](t,x) := N_n e^{ix\cdot \xi_n} e^{-it|\xi_n|^2} v(t_n+tN_n^2,N_n(x-x_n-2\xi_n t)).
\end{align}
Observe that $e^{it\Delta} G_n f = T_n [e^{it\Delta} f]$.  Incidentally, none of the collections of operators given above form groups (under composition); however, this could be remedied by augmenting by phase rotations.

\begin{defn} [Othogonality] Two quadruples of parameters $(N_n^j, \xi_n^j, x_n^j, t_n^j)$ and $(N_n^k, \xi_n^k, x_n^k, t_n^k)$ are said to be \emph{orthogonal} if for $n\to \infty$,
\begin{align*}
\tfrac{N_n^j}{N_n^k}+\tfrac{N_n^k}{N_n^j}+\tfrac{|\xi_n^j-\xi_n^k|^2}{N_n^j N_n^k}+&N_n^j N_n^k\bigl |\tfrac{t_n^j}{(N_n^j)^2}-\tfrac{t_n^k}{(N_n^k)^2}\bigr|\\
&+N_n^j N_n^k\bigl| x_n^j-x_n^k-2\tfrac{t_n^j}{(N_n^j)^2}(\xi_n^j-\xi_n^k)^2\bigr|\to 0.
\end{align*}
In this case we write $(N_n^j, \xi_n^j, x_n^j, t_n^j)\perp (N_n^k, \xi_n^k, x_n^k, t_n^k)$ or $j\perp k$ for short.
\end{defn}

\begin{thm}[Linear profile decomposition \cite{MerleVega}]\label{thm:lpd}
Let $u_n$ be a bounded sequence in $L^2(\R^2)$. Then, passing to a subsequence if necessary, there exists $J^*\in \{0, 1, \cdots\}\cup\{\infty\}$ and for each finite $1\leq j \leq J^*$ there exist a non-trivial function $\phi^j\in L^2(\R^2)$ and parameters $(N_n^j, \xi_n^j, x_n^j, t_n^j)\subset \R^+\times\R^2\times\R^2\times \R$, conforming to one of the three cases below, so that defining $G_n^j:=G_{N_n^j, \xi_n^j, x_n^j, t_n^j}$ and $r_n^J$ via
\begin{align*}
u_n=\sum_{j=1}^ J G_n^j\phi^j +r_n^J,
\end{align*}
we have the following properties:
\begin{gather*}
j\perp k \qtq{for any}j\neq k, \\
\lim_{J\to \infty}\limsup_{n\to \infty} \|e^{it\Delta}r_n^J\|_{L_{t,x}^4}=0,\\
\phi^1=\wlim (G_n^1)^{-1} u_n \qtq{and} \phi^j=\wlim(G_n^j)^{-1} r_n^{j-1} \qtq{for} j\ge 2, \\
\sup_J \lim_{n\to \infty}\biggl[\|u_n\|_2^2 -\sum_{j=1}^J \|\phi^j\|_2^2 -\|r_n^J\|_2^2\biggr]=0.
\end{gather*}
The three cases are:

Case I: $N_n^j\to \infty$; or $N_n^j\to 0$ and  $|\xi_n^j|\to \infty$; or $N_n^j \equiv 1$ and $|\xi_n^j|\to \infty$.

Case II: $N_n^j\to 0$ and  $\xi_n^j\to \xi^j\in \R^2$.

Case III: $N_n^j\equiv 1$ and $\xi_n^j\equiv 0$.

\noindent
In each of these cases we may assume that either $t_n^j\equiv 0$ or $t_n^j\to \pm \infty$.
\end{thm}

\begin{proof}
The result quoted here differs from that of \cite{MerleVega} only in claiming that it is possible to choose parameters that conform to one of the three cases given.   Therefore, it suffices to show that by passing to a further subsequence and altering the profiles $\phi^j$, if necessary, one may reduce general tuples of parameters to conformant tuples.

Passing to a subsequence and applying a diagonal procedure, we can assume that either $t_n^j\to t^j\in \R$ or $t_n^j\to \pm\infty$ for every $j$.  When $t_n^j\to t^j\in \R$, by the dominated convergence theorem,
\begin{align*}
\lim_{n\to\infty}\|(e^{it_n^j\Delta}-e^{it^j\Delta})\phi^j\|_2=0.
\end{align*}
Therefore, in this case we can assume $t_n^j\equiv0$ by redefining $\phi^j$ as $e^{it^j\Delta}\phi^j$ and throwing the error into $r_n^J$.

Passing to a further subsequence and applying a diagonal procedure again we can also assume that for each $j$,
\begin{align*}
\lim_{n\to \infty} N_n^j =0 \qtq{or} \lim_{n\to \infty} N_n^j=N^j\in \R^+ \qtq{or} \lim_{n\to \infty}N_n^j=\infty.
\end{align*}
If $N_n^j\to N^j\in \R^+$, by redefining $\phi^j$ as $N^j\phi^j(N^j\cdot)$ and rescaling the other parameters accordingly (which does not affect the asymptotic orthogonality of profiles), we can assume $N_n^j\to N^j=1$.   Moreover, as in this case
$$
\lim_{n\to\infty}\|N_n^j\phi^j(N_n^jx)- \phi^j(x)\|_{L_x^2}=0,
$$
we may assume that $N_n^j\equiv 1$, by throwing the error into $r_n^J$.  Hence, for each $j$ we can assume
\begin{align*}
\lim_{n\to \infty} N_n^j =0 \qtq{or} N_n^j\equiv 1 \qtq{or} \lim_{n\to \infty}N_n^j=\infty.
\end{align*}

When $N_n^j\to 0$ or $N_n^j\equiv 1$, by passing to a further subsequence we can assume that
\begin{align*}
\text{either} \quad |\xi_n^j|\to \infty \qtq{or} \xi_n^j\to \xi^j\in \R^2.
\end{align*}

Finally, we make the requisite further simplifications in the case $N_n^j\equiv 1$ and $\xi_n^j\to \xi^j\in \R^2$.  Observe that if we set $\tilde x_n^j = x_n^j - 2t_n^j \xi_n^j$, then
\begin{align*}
\phi_n^j=G_n^j \phi^j&=e^{it_n^j\Delta}\bigl[e^{ix\cdot \xi_n^j}e^{it_n^j|\xi_n^j|^2}\phi^j(x-x_n^j+2t_n^j\xi_n^j)\bigr]\\
&=e^{it_n^j\Delta}\bigl[e^{i(x-\tnj)\cdot \xi_n^j}\phi^j(x-\tnj)e^{i(\tnj\cdot\xi_n^j+t_n^j|\xi_n^j|^2)}\bigr].
\end{align*}
Passing to a subsequence, we can assume
\begin{align*}
\lim_{n\to \infty}e^{i(\tnj\cdot \xi_n^j+t_n^j|\xi_n^j|^2)}=e^{i\theta}.
\end{align*}
On the other hand, the dominated convergence theorem yields
\begin{align*}
\lim_{n\to \infty}\bigl\| e^{i\xi_n^j \cdot x}\phi^j -e^{i\xi^j\cdot x}\phi^j\bigr\|_2 =0.
\end{align*}
Thus, if we replace $\phi^j$ by $e^{i\theta} e^{i\xi^j\cdot x}\phi^j$, set $\xi_n^j\equiv 0$ and replace the original translation parameters $x_n^j$ by $\tilde x_n^j$, then the resulting errors can be safely absorbed into $r^J_n$.  It is elementary to verify that the changes just described do not spoil the orthogonality between the $j^\text{th}$ profile and the others.
\end{proof}

\section{Well-posedness theory for several NLS equations}\label{S:3}
Throughout the paper, we will use several versions of the cubic NLS, which can be written into the following general form
\begin{align}\label{eq:1}
iu_t+\Delta u=\alpha^4 \Po F(\Po u),
\end{align}
where $0\le \alpha\le 1$ and $\Po$ is a Mikhlin multiplier with real symmetric symbol.  All estimates in this section will be uniform in $\alpha$ and in the symbol norms associated to the multiplier underlying $\Po$.

If $\alpha=0$, \eqref{eq:1} is the linear Schr\"odinger equation.  When $\alpha=1$ and $\Po=I$, \eqref{eq:1} is the cubic NLS, which we denote by NLS.  If $\Po=I$, we call the equation $\nalf$.

Solutions to \eqref{eq:1} conserve the mass and energy
\begin{align*}
\int_{\R^2} |u(t,x)|^2 \,dx \qtq{and}
E(u(t)):=\int_{\R^2}\tfrac 12 |\nabla u(t,x)|^2 +\tfrac {\alpha^4}4 |\Po u(t,x)|^4 \,dx.
\end{align*}
Indeed, \eqref{eq:1} is the Hamiltonian evolution associated to $E(u)$ through the standard symplectic structure on $L^2(\R^2)$.

Next, we record the basic local and perturbation theories for \eqref{eq:1}; as the operator $\Po$ is $L^p$-bounded, no meaningful changes need to be made to the proofs given in \cite{Clay}, for example.

\begin{lem}[Local theory]\label{lm:loc}
Let $u_0\in L^2(\R^2)$ with $\|u_0\|_2\le M$.  There exists $\eps_0>0$ such that whenever $\eps<\eps_0$ and
\begin{align*}
\|e^{it\Delta}u_0\|_{\ftr}\le \eps
\end{align*}
for some interval $I$ containing $0$, there exists a unique solution $u:I\times\R^2\to \C$ to \eqref{eq:1} with initial data $u(0)=u_0$.  Moreover,
\begin{align*}
\|u\|_{\ftr}\le 2\eps \qtq{and} \|u\|_{S(I)}\lsm M.
\end{align*}
In particular, when $M\lesssim \eps_0$ (with the implicit constant given by the Strichartz inequality), the solution $u$ is global and satisfies
\begin{align*}
\|u\|_{L_{t,x}^4(\R\times\R^2)}\lesssim \| u_0 \|_{L^2(\R^2)}.
\end{align*}
\end{lem}

\begin{lem}[Perturbation theory]\label{lm:stab}
Let $I$ be a compact interval. Let $\tilde u:I\times\R^2\to \C$ be an approximate solution to \eqref{eq:1} in the sense that
\begin{align*}
(i\partial_t+\Delta)\tilde u=\alpha^4 \Po F(\Po \tilde u)+e
\end{align*}
for some function $e$. Assume that
\begin{align*}
\|\tilde u\|_{L^\infty_t L^2_x\ir}\le M \qtq{and} \|\tilde u\|_{\ftr}\le L
\end{align*}
for some positive constants $M,L$.  Let $u_0\in L^2$ be such that for some $t_0\in I$,
\begin{align*}
\|u_0-\tilde u(t_0)\|_2\le M'
\end{align*}
for some positive constant $M'$.  Finally, assume the smallness conditions
\begin{align*}
\|e^{i(t-t_0)\Delta}(u_0-\tilde u(t_0))\|_{\ftr} +  \Bigl\|\int_{t_0}^t e^{i(t-s)\Delta } e(s) \,ds\Bigr\|_{ S(I)}\le \eps
\end{align*}
for some $0<\eps<\eps_0(M, M',L)$. Then there exists a solution $u:I\times\R^2\to \C$ to \eqref{eq:1} with initial data $u_0$ at $t=t_0$ satisfying
\begin{gather*}
\|u-\tilde u\|_{\ftr}\le \eps C(M,M',L)\\
\|u-\tilde u\|_{S(I)}\le M'C(M, M', L) \\
\|u\|_{S(I)}\le C(M,M',L).
\end{gather*}
\end{lem}

\begin{lem}[Persistence of positive regularity]\label{lm:perpo}
Let $u:I\times\R^2\to \C$ be a finite mass solution to \eqref{eq:1} with
\begin{align*}
\|u\|_{\ftr}\le L
\end{align*}
for some positive constant $L$.  Fix $0\le s\le 1$ and assume that $u(t_0)\in  H^s$ for some $t_0\in I$. Then
\begin{align*}
\||\nabla|^su\|_{S(I)}\le C(L)\|u(t_0)\|_{\dot H^s}.
\end{align*}
\end{lem}

\begin{proof} Let $\eta$ be a small constant to be chosen later depending only on Strichartz constants. Divide $I$ into $J=O\bigl (1+\frac{L^4}{\eta^4}\bigr)$ many intervals such that
on each subinterval $I_j=[t_{j-1},t_j]$ we have
\begin{align*}
\|u\|_{L_{t,x}^4(I_j\times\R^2)}\leq \eta.
\end{align*}
On each subinterval we apply the Strichartz estimate and the fractional chain rule to obtain
\begin{align*}
\||\nabla|^su\|_{S(I_j)}&\lsm \|u(t_{j-1})\|_{\dot H^s}+\||\nabla|^s\Po F(\Po u)\|_{\atr}\\
&\lsm \|u(t_{j-1})\|_{\dot H^s}+\|u\|_{L_{t,x}^4(I_j\times\R^2)}^2 \||\nabla|^s u\|_{L_{t,x}^4(I_j\times\R^2)}\\
&\lsm \|u(t_{j-1})\|_{\dot H^s}+\eta^2\||\nabla|^su\|_{S(I_j)}.
\end{align*}
Choosing $\eta$ small enough we get
\begin{align*}
\||\nabla|^su\|_{S(I_j)}\lesssim \|u(t_{j-1})\|_{\dot H^s}.
\end{align*}
Iterating this process $J$ many times we derive the claim.
\end{proof}

As a consequence of Lemma \ref{lm:perpo}, we have

\begin{lem}[Persistence of low-frequency localization]\label{lm:per}
Given any $M>0$ and $L>0$, there is a threshold $c(M,L)>0$ so that the following holds:  If $u:I\times \R^2\to \C$ is a solution to \eqref{eq:1} satisfying
\begin{align*}
\|u\|_{L^\infty_t L^2_x\ir}\le M,\quad \|u\|_{\ftr}\le L, \qtq{and}  \|P_{>N} u(t_0)\|_2\le \eta \leq c(L,M)
\end{align*}
for some dyadic number $N$ and some $t_0\in I$, then
\begin{align*}
\|P_{>\frac N\eta} u\|_{S(I)}\le C(M, L) \eta.
\end{align*}
\end{lem}

\begin{proof} Let $v$ be the solution to \eqref{eq:1} with initial condition
\begin{align*}
v(t_0)=P_{\le N} u(t_0).
\end{align*}
Then
\begin{align*}
\|u(t_0)-v(t_0)\|_2=\|P_{>N} u(t_0)\|_2\le \eta.
\end{align*}

By Lemma \ref{lm:stab}, if $\eta$ is sufficiently small depending on $M$ and $L$, the solution $v$ is defined on the whole interval $I$ and satisfies
\begin{align}\label{eq:7.1}
\|u-v\|_{S(I)}\le C(M, L) \eta.
\end{align}
On the other hand, from Lemma \ref{lm:perpo},
\begin{align}\label{eq:7.2}
\|\nabla v\|_{S(I)}\le C(M, L) \|v(t_0)\|_{\dot H^1}\le C(M, L) N.
\end{align}

Combining \eqref{eq:7.1} and \eqref{eq:7.2}, we obtain
\begin{align*}
\|P_{>\frac N\eta} u\|_{S(I)}&\le \|P_{>\frac N\eta}(u-v)\|_{S(I)}+\|P_{>\frac N\eta} v\|_{S(I)}\\
&\lsm \|u-v\|_{S(I)}+\tfrac{\eta} N\|\nabla v\|_{S(I)}\\
&\le C(M,L) \eta.
\end{align*}

This completes the proof of the lemma.
\end{proof}

\begin{lem}[Persistence of negative regularity]\label{lm:perne}
Let $u:I\times\R^2\to \C$ be a finite mass solution to \eqref{eq:1} such that
\begin{align*}
\|u\|_2\le M \qtq{and} \|u\|_{\ftr}\le L
\end{align*}
for some positive constants $M$ and $L$.  Fix $0<s<\frac12$ and assume that there exists $t_0\in I$ such that
\begin{align*}
\| P_N u(t_0)\|_2\le AN^s \qtq{for all} N\in 2^{\Z}
\end{align*}
for some constant $A>0$. Then there exists $C_0=C_0(M, L)$  such that
\begin{align}
\|P_N u\|_{S(I)}\le C_0 AN^s.\label{eq:4}
\end{align}
\end{lem}

\begin{proof} Subdividing the time interval $I$ as in Lemma~\ref{lm:perpo}, we can assume
\begin{align*}
\|u\|_{\ftr}\sim \eta
\end{align*}
for some small $\eta$ to be chosen later depending only on the Strichartz constant and $M$.  An application of the Strichartz inequality yields
\begin{align}\label{S1023}
\|u\|_{S_*(I)}\lsm \|u(t_0)\|_2 +\|u\|_{\ftr}^3\lsm \|u(t_0)\|_2+\eta^3.
\end{align}

By a continuity argument, it suffices to prove \eqref{eq:4} under the additional hypothesis
\begin{align}\label{eq:4.1}
\|P_N u\|_{S(I)}\le 2C_0 A N^s.
\end{align}

Using the Strichartz inequality, we estimate
\begin{align*}
\|P_N u\|_{S(I)}&\lsm \|P_N u(t_0)\|_2 +\alpha^4\biggl \|P_N\int_{t_0}^ t e^{i(t-s)\Delta}\Po F(\Po u)(s) \,ds\biggr\|_{S(I)}\\
&\lsm AN^s+ \biggl \|P_N \int_{t_0}^ t e^{i(t-s)\Delta}\Po F(\Po u) \,ds\biggr\|_{S(I)}.
\end{align*}
Choosing $C_0$ sufficiently large, we see that the first summand above is acceptable.

It remains to control the second summand.  To that end, take $\|g\|_{N(I)}=1$ and consider the pairing
\begin{align}\label{eq:5.1}\
&\Bigl \langle g, P_N\int_{t_0}^ t e^{i(t-s)\Delta}\Po F(\Po u)(s) \,ds\Bigr\rangle_{L_{t,x}^2}\notag\\
&\quad=\iint_{I\times\R^2} \overline{g(t,x)} P_N\int_{t_0}^ t e^{i(t-s)\Delta}\Po F(\Po u)(s) \,ds \,dt \,dx\notag\\
&\quad=\iint_{I\times\R^2}\Po P_N \overline{G(s,x)}\ F(\Po u)(s) \,ds \,dx,
 \end{align}
where we use the notation
$$
G(s,x):=\int \chi_{\{t\in I: t>s\}} e^{i(t-s)\Delta}g(t)\,dt.
$$
Note that by Strichartz,
\begin{align}\label{eq:5.2}
\|G\|_{S_*^0(I)}\lsm 1.
\end{align}

Now define $\uh:=P_{>10 N}u$ and $\ul:=P_{\le 10 N} u$. We have
\begin{align*}
\text{RHS}\eqref{eq:5.1}&\lsm \iint_{I\times\R^2}|\Po P_N G||\Po\ul|^3 \,ds \,dx+\iint_{I\times\R^2}|\Po P_N G| |\Po \uh|^3 \,ds \,dx.
\end{align*}
Using H\"older and \eqref{eq:4.1}, we estimate the first term on the right-hand side above by
\begin{align*}
\|\Po & P_N G\|_{\ftr}\|\Po \ul\|_{\ftr}^3\\
&\lsm \|G\|_{\ftr}\|u\|_{\ftr}^2 \sum_{K\le 10 N}\|P_K u\|_{\ftr}\\
&\lsm \eta^2 \sum_{K\le 10 N} 2C_0 A K^s\\
&\lsm \eta^2C_0 AN^s.
\end{align*}
Using Lemma \ref{lm:bl2}, \eqref{S1023}, and \eqref{eq:5.2}, we estimate the second term by a constant multiple of
\begin{align*}
\sum_{N_1\ge N_2>10 N} &\|P_N \Po G \Po u_{N_1} \Po u_{N_2}\Po\uh \|_{L_{t,x}^1\ir}\\
&\lsm \|u\|_{\ftr}\sum_{N_1\ge N_2> 10 N}\|P_N \Po G \Po u_{N_1}\|_{L_{t,x}^2\ir}\|u_{N_2}\|_{\ftr}\\
&\lsm \eta \sum_{N_1\ge N_2>10 N}\bigl (\tfrac N{N_1}\bigr)^{\frac 12 } \|P_N \Po G\|_{S_*^0(I)} \|\Po u_{N_1}\|_{S_*^0(I)} (2C_0 AN_2^s)\\
&\lsm \eta \sum_{N_1\ge N_2> 10 N}\bigl (\tfrac N{N_1}\bigr)^{\frac 12} C_0 A N_2^s\\
&\lsm \eta C_0 A N^s.
\end{align*}

Combining the bounds above we get
\begin{align*}
\text{RHS}\eqref{eq:5.1}\lsm \eta C_0 AN^s,
\end{align*}
which is acceptable if $\eta$ is chosen small enough to defeat the implicit constant. This completes the proof of the lemma.
\end{proof}

As a consequence of Lemma \ref{lm:perne} we have

\begin{lem}[Persistence of high-frequency localization]\label{lm:per2}
Given any $M>0$ and $L>0$, there is a threshold $c(M,L)>0$ so that the following holds: If $u:I\times\R^2\to \C$ is a solution to \eqref{eq:1} with
\begin{align*}
\|u\|_2\le M, \quad \|u\|_{L_{t,x}^4\ir}\le L, \qtq{and} \|P_{\le N} u(t_0)\|_2\le \eta \leq c(M,L)
\end{align*}
for some $t_0\in I$ and $N\in 2^\Z$, then
\begin{align*}
\|P_{\le \eta N} u\|_{S(I)}\le C(M, L) \eta^{\frac 13}.
\end{align*}
\end{lem}

\begin{proof} Let $v$ be the solution to \eqref{eq:1} with initial data
\begin{align*}
v(t_0)=P_{>N} u(t_0).
\end{align*}
Then
\begin{align*}
\|u(t_0)-v(t_0)\|_2=\|P_{\le N}u(t_0)\|_2\le \eta.
\end{align*}

By Lemma \ref{lm:stab}, if $\eta$ is sufficiently small depending on $M$ and $L$, the solution $v$ is defined on the whole interval $I$ and satisfies
\begin{align*}
\|u-v\|_{S(I)}\le C(M, L) \eta \qtq{and} \|v\|_{S(I)}\le C(M, L).
\end{align*}
Moreover, by Lemma \ref{lm:perne},
\begin{align*}
\sup_{K\in 2^\Z}K^{-\frac 13}\|P_K v\|_{S(I)}&\le C(M, L) \sup_{K\in 2^\Z} K^{-\frac 13}\|P_K v(t_0)\|_2\le C(M, L) N^{-\frac 13}.
\end{align*}
We can then estimate
\begin{align*}
\|P_{\le\eta N} u\|_{S(I)}&\le \|P_{\le\eta N}(u-v)\|_{S(I)}+\|P_{\le\eta N} v\|_{S(I)}\\
&\le \|u-v\|_{S(I)}+\sum_{K\le\eta N} \|P_K v\|_{S(I)}\\
&\le C(M, L)\eta+\sum_{K\le\eta N} K^{\frac 13} C(M, L) N^{-\frac 13}\\
&\le C(M, L)\eta^{\frac 13}.
\end{align*}

This completes the proof of the lemma.
\end{proof}

\begin{lem}[Uniform space-time bound for $\nalf$] \label{lm:ub}
Suppose $u$ is the solution to $\text{NLS}_\alpha$ with initial data $u(0)=u_0$ and $\|u_0\|_2\le M$.  Then for any $0\le \alpha\le 1$, the solution $u$ is global and satisfies
\begin{align*}
\|u\|_{S(\R)}\le C(M)
\end{align*}
uniformly in $\alpha$.
\end{lem}

\begin{proof}
From the standard local theory, we only need to prove the claim as an \emph{a priori} estimate.  By Strichartz,
\begin{align*}
\|u\|_{S(\R)}\lsm \|u_0\|_2+\alpha^4\||u|^2 u\|_{L_{t,x}^{\frac 43}(\R\times\R^2)}
&\lsm \|u_0\|_2+\alpha^4\|u\|_{S(\R)}^3.
\end{align*}

Using a continuity argument, we see that there exists $0<\alpha_0(M)\ll 1$ such that when $\alpha\le \alpha_0$, the solution $u$ to $\text{NLS}_\alpha$ is global and satisfies
\begin{align*}
\|u\|_{S(\R)}\lsm M.
\end{align*}

It remains to discuss those $\alpha$ belonging to the compact interval $[\alpha_0, 1]$.  It is easy to check that if $u$ is a solution to $\text{NLS}_\alpha$, then $\alpha^2 u$ solves NLS with data $\alpha^2 u_0$.  Therefore,
\begin{align*}
\|u\|_{S(\R)}=\tfrac 1{\alpha^2} \|\alpha^2 u\|_{S(\R)}\le \tfrac 1{\alpha^2}\mathcal L(M),
\end{align*}
where
\begin{align*}
\mathcal L(M):=\sup \{\|u\|_{S(I)}\}
\end{align*}
and the supremum is taken over all time intervals $I$ and all solutions $u:I\times\R^2\to \C$ with $M(u)\le M$.  The claim now follows from Theorem~\ref{T:Dodson}, which implies that $\mathcal L(M)$ is finite.
\end{proof}

%%%%%%%%%%%%%%%%%%%%%%%%%%%%%%%%%%%%%%%%%%%%%%%%%%%%%%%%%%%%%%%%%%%%%%%%
\section{Global well-posedness and scattering for $\pds$}\label{S:4}
%%%%%%%%%%%%%%%%%%%%%%%%%%%%%%%%%%%%%%%%%%%%%%%%%%%%%%%%%%%%%%%%%%%%%%%%

In this section we prove global well-posedness and scattering for
\begin{align*}
(\pds): \begin{cases} iu_t+\Delta u=\pd F(\pd u),\\
u(0,x)=u_0(x)\in  L_x^2.
\end{cases}
\end{align*}
Here $\pd$ is a Fourier multiplier with symbol $m_D(\xi)$ defined as follows: Let $\varphi:\R^2\to[0,1]$ be the bump function used in the definition of Littlewood--Paley projections.  For $1\leq D\in 2^\Z$ we define
\begin{align}\label{m_D definition}
m_D(\xi) &:=\frac1{\log_2(2D)} \sum_{N\ge 1}^D \varphi(\xi/N) \\
	&= \varphi(\xi) +\sum_{N\ge 2}^D \biggl[\frac{\log_2(2D)-\log_2 (N)}{\log_2(2D)}\biggr]\bigl[\varphi(\xi/N) - \varphi(2\xi/N)\bigr]. \notag
\end{align}
It is easy to check that the symbol $m_D(\xi)$ satisfies
\begin{align}\label{mdbound1}
0\le m_D(\xi)\le 1\qtq{and}
\begin{cases}
m_D(\xi)=1, \ &\text{if } |\xi|\le \frac 12,\\
m_D(\xi)=0, \ &\text{if } |\xi|>2D.
\end{cases}
\end{align}
It is also easy to verify that $\pd$ is a Mikhlin multiplier uniformly for $D\geq 1$.    Moreover, for any number $k\ge 1$,
\begin{align}\label{mdbound2}
|m_D(\xi)-m_D(k\xi)|\lsm \frac{\log_2 (k)}{\log_2 (D)}.
\end{align}

Both the local theory and the small-data global theory for $\pds$ are contained in Lemma~\ref{lm:loc}.   Our goal for this section is to prove the following large-data global result:

\begin{thm}[Scattering for $\pds$]\label{thm:scapds} Given $M>0$, there are constants $C(M)$ and $D_0(M)$ so that the following holds:  For any $u_0\in L^2(\R^2)$ with $\|u_0\|_2\le M$ and any $D\ge D_0(M)$, there exists a unique global solution $u$ to $\pds$; moreover,
\begin{align}\label{E:thm:scapds}
\|u\|_{S(\R)}\le C(M).
\end{align}
In particular, there exist $u_\pm \in L^2(\R^2)$ such that
\begin{align*}
\lim_{t\to \pm \infty}\|u(t)-e^{it\Delta}u_{\pm}\|_2=0.
\end{align*}
\end{thm}

Inspired by Bourgain's induction on energy method \cite{Bourg:JAMS}, we will prove this theorem by inducting on the mass.  For well-posedness problems, the original induction argument has been mostly supplanted by the application of concentration compactness techniques in the style of \cite{KenigMerle:H1}; however, that type of argument is ill-suited to the current problem.  While we will develop a nonlinear profile decomposition in the next section, this will be for $D$ fixed.  A concentration-compactness-style proof of Theorem~\ref{thm:scapds} would require such a decomposition for a general sequence $D_n\to\infty$, leading to a significant proliferation of cases.  Moreover, one would still be left to prove space-time bounds for the individual profiles, which is not significantly simpler than just proving space-time bounds for general solutions as we do in this section.  Note that the principal advantage of solutions in the form of a single profile appears when using (localized) multiplier identities, such as monotonicity formulae; no such identities will be used in proving Theorem~\ref{thm:scapds}.  We contend that the arguments that follow are very robust and provide a general method for transferring well-posedness from a dispersive equation to Fourier truncated variants.

The proof of Theorem \ref{thm:scapds} is built on Lemma \ref{lm:loc} together with the following two propositions:

\begin{prop}[Frequency-localized data obeys bounds]\label{prop:onebubble}
Given $M>0$ there exists $\eta_0(M)>0$ so that for any $\eta_1>0$ there exists $D_1=D_1(M,\eta_1)$ so that the following holds:  Suppose $u_0\in L^2$ satisfies $\|u_0\|_2\le M$ as well as
\begin{align*}
\|P_{\le \frac{N_0}2} u_0\|_2\le 2\eta_0 \qtq{and} \|P_{>\frac {N_0}{\eta_1}} u_0\|_2\le 2\eta_0 \quad \text{for some $N_0\in 2^\Z$}.
\end{align*}
Then for any $D\ge D_1$, there is a global solution $u$ to $\pds$.  Moreover,
\begin{align*}
\|u\|_{S(\R)}\le C(M),
\end{align*}
with the right-hand side independent of $N_0$, $u_0$, $D$, and $\eta_1$.
\end{prop}

\begin{prop}[Frequency-delocalized data inherits bounds]\label{prop:twobubble} Let $M>0$ and $\eta_0>0$ be given and suppose Theorem~\ref{thm:scapds} holds up to mass $M-\eta_0$, that is, there are constants $D_2$ and $B$ so that all solutions $u$ to $\pds$ with  $D\geq D_2$ and  $\|u\|_{2}\leq M-\eta_0$ can be extended globally in time and obey $ \|u\|_{S(\R)} \leq B$.
Then, there exists a constant $\eta_1=\eta_1(M, B)$ so that the following holds: If $u_0\in L^2$ satisfies $\|u_0\|_2\le M$ and
\begin{align*}
\|P_{\le N_0}u_0\|_2\ge \eta_0 \qtq{and} \|P_{>\frac{N_0}{\eta_1}}u_0\|_2\ge \eta_0 \quad \text{for some $N_0\in 2^\Z$},
\end{align*}
then, for any $D\geq D_2$ there exists a unique global solution $u$ to $\pds$ with initial data $u_0$; moreover,
\begin{align*}
\|u\|_{S(\R)}\le C(M,B).
\end{align*}
\end{prop}

Proving these two propositions will occupy the better part of this section.  Before embarking on this task, we would first like to show how Theorem~\ref{thm:scapds} follows from these two propositions.

\begin{proof}[Proof of Theorem \ref{thm:scapds}]
It suffices to prove the result for an arbitrary (but fixed) mass threshold $M$, say $M=M_*$.  We will do this inductively.

By Lemma~\ref{lm:loc} there is a constant $\eps>0$ so that \eqref{E:thm:scapds} holds for all initial data with mass $\|u_0\|_{2}\leq \eps$.  This is the base step of the induction.

Choose $\eta_0=\eta_0(M_*)$ small enough so that Proposition~\ref{prop:onebubble} may be applied to all initial data with mass not exceeding $M_*$.  We reduce $\eta_0$ further, if necessary, so that  $2\eta_0 < \eps$.  This $\eta_0$ represents the step-size in the induction.

Now we come to the inductive step.  Let us suppose Theorem~\ref{thm:scapds} is valid at the mass threshold $M-\eta_0$ in the sense explained in Proposition~\ref{prop:twobubble}.  This introduces two constants $B$ and $D_2$, as there.  We now choose $\eta_1=\eta_1(M,B)$ as in that proposition and then set $D_0= D_2 + D_1(M,\eta_1)$ where $D_1$ is as in Proposition~\ref{prop:onebubble}.  To complete the inductive step, we must obtain a uniform bound $C$ so that for any initial data $u_0\in L^2(\R^2)$ with $\|u_0\|_{2}\leq M\wedge M_*$ and any parameter $D\geq D_0$ the corresponding solution $u$ to $\pds$ obeys $\| u\|_{S(\R)}\leq C$.  We divide into two cases:

\medskip

\textbf{Case 1:} $\qquad\displaystyle \inf_{N\in 2^\Z}\|P_{\le N/2}u_0\|_2+\|P_{>\frac N{\eta_1}} u_0\|_2 <2 \eta_0.$\\[1ex]
In this case, we can find $N_0\in 2^\Z$ so that
\begin{align*}
\|P_{\le N_0/2} u_0\|\le 2\eta_0 \qtq{and} \|P_{>\frac {N_0}{\eta_1}}\|_2 \leq 2\eta_0.
\end{align*}
In this case, the requisite bounds follow from Proposition \ref{prop:onebubble}.

\medskip

\textbf{Case 2:} $\qquad\displaystyle \|P_{\le N/2}u_0\|_2+\|P_{>\frac 1{\eta_1}N}u_0\|_2 \geq 2\eta_0 \quad\text{for any $N\in 2^\Z$}.$\\[1ex]
If $\|u_0\|_{2}<2\eta_0$, the space-time bounds follow from Lemma~\ref{lm:loc}; we exclude this from further consideration. By the dominated convergence theorem,
\begin{align*}
\|P_{\leq N} u_0\|_2\to 0 \text{ as } N\to 0 \qtq{and} \|P_{\leq N} u_0\|_2\to \|u_0\|_2\geq 2\eta_0  \text{ as } N\to \infty.
\end{align*}
Thus, one can find $N_0\in 2^\Z$ such that
\begin{align*}
\|P_{\le N_0/2}u_0\|_2\le \eta_0 \qtq{and} \|P_{\le N_0}u_0\|_2\ge \eta_0.
\end{align*}
Together with the defining property of Case~2, this then implies
\begin{align*}
\|P_{>\frac {N_0}{\eta_1}}u_0\|_2\ge \eta_0.
\end{align*}
Thus Proposition \ref{prop:twobubble} applies yielding the requisite space-time bound.

This completes the proof of Theorem \ref{thm:scapds}.
\end{proof}

It remains to prove Propositions~\ref{prop:onebubble} and \ref{prop:twobubble}.  We start with

\begin{proof}[Proof of Proposition \ref{prop:onebubble}]

Define
\begin{align*}
\pl=P_{< N_0\eta_0}, \quad \ph=P_{>\frac{N_0}{\eta_0\eta_1}}, \qtq{and} \pmd=P_{N_0\eta_0\le\cdot\leq\frac{N_0}{\eta_0\eta_1}}.
\end{align*}
As $\eta_0$ is small, the notions of high and low frequencies appearing here lie far beyond those appearing in the hypotheses of the proposition.

Let $\alpha:=m_D(N_0\eta_0)$.  We discuss two cases:  When $0\le \alpha\le \eta_0^{\frac 16}$, we approximate the solution to $\pds$ by the free Schr\"odinger evolution. In the case when $\eta_0^{\frac 16}< \alpha\le 1$, we use the solution to $\nalf$ as an approximate solution.  In either case, we then deduce the requisite bounds from Lemma~\ref{lm:stab}.

First we assume $0\le \alpha\le \eta_0^{\frac 16}$.  Define $v:=e^{it\Delta}\pmd u_0$; then $v$ solves
\begin{align*}
\begin{cases}
i\partial_t v+\Delta v=\pd F(\pd v)+e,\\
v(0)=\pmd u_0,
\end{cases}
\end{align*}
with $e=-\pd F(\pd v)$.  By Strichartz,
\begin{align}
\|u(0)-v(0)\|_2\le4\eta_0,\label{pp1:smalldata} \\
\|\pl v\|_{S(\R)}+\|\ph v\|_{S(\R)}\lsm \eta_0.\label{pp1:tail}
\end{align}
To apply Lemma \ref{lm:stab}, we need to show $e$ is small in suitable spaces. Using H\"older and Strichartz, we estimate
\begin{align*}
\|e\|_{N(\R)}&\le \|\pd F(\pd v)\|_{N(\R)}\\
&\le \|\pd v\|_{L_t^3L_x^6}^2\|\pd v\|_{L_t^\infty L_x^2}\\
&\lsm M^2 (\|\pl v\|_{L_t^\infty L_x^2}+\|\ph v\|_{L_t^\infty L_x^2}+\|\pd \pmd v\|_{L_t^\infty L_x^2}).
\end{align*}
Using the fact that $m_D(|\xi|)$ is decreasing, we get
\begin{align*}
\|\pd\pmd\|_{L^2\to L^2}\le m_D(N_0\eta_0)=\alpha\le \eta_0^{\frac 16}.
\end{align*}
Using also \eqref{pp1:tail} we obtain
\begin{align*}
\|e\|_{N(\R)}\lsm M^2(\eta_0+M\eta_0^{\frac 16}).
\end{align*}
Combining this with \eqref{pp1:smalldata} and invoking Lemma \ref{lm:stab}, we conclude that if $\eta_0$ is sufficiently small depending only on $M$, the solution to $\pds$ is global and satisfies
\begin{align*}
\|u\|_{S(\R)}\le C(M).
\end{align*}

Assume now that $ \eta_0^{\frac 16}< \alpha\le 1$ and let $v$ be the solution to
\begin{align*}
\begin{cases}
iv_t+\Delta v=\alpha^4 F(v),\\
v(0)=P_{N_0\leq\cdot\leq N_0/\eta_1} u_0.
\end{cases}
\end{align*}
By Lemma \ref{lm:ub}, such a solution exists, is global, and satisfies
\begin{align}\label{E:vbndd}
\|v\|_{S(\R)}\le C(M).
\end{align}

By construction,
\begin{align}\label{pp1:smd}
\|v(0)-u(0)\|_2\le 4\eta_0 \qtq{and} \|\pl v(0)\|_2+\|\ph v(0)\|_2=0.
\end{align}
This together with \eqref{E:vbndd} and Lemmas \ref{lm:per} and \ref{lm:per2} give
\begin{align*}
\|\pl v\|_{S(\R)}\le C(M) \eta_0^{\frac 13} \qtq{and} \|\ph v\|_{S(\R)}\le C(M)\eta_0.
\end{align*}
In particular, from Duhamel's formula, writing
$$
G(t,x):=\alpha^4\int_0^t e^{i(t-s)\Delta}F(v) \,ds,
$$
we have
\begin{align}\label{G}
\begin{cases}
\|\pl G\|_{S(\R)}\le \|\pl v\|_{S(\R)}+\|\pl v(0)\|_2\lsm C(M) \eta_0^{\frac 13} + \eta_0,\\
\|\ph G\|_{S(\R)}\le \|\ph v\|_{S(\R)}+\|\ph v(0)\|_2\lsm C(M) \eta_0 + \eta_0.
\end{cases}
\end{align}

Our goal is to show that for $\eta_0$ sufficiently small depending on $M$, $v$ is an approximate solution to $\pds$.  Indeed, $v$ solves
\begin{align*}
(i\partial_t+\Delta) v=\pd F(\pd v)+e
\end{align*}
with
\begin{align*}
e=\alpha^4 |v|^2 v-\pd F(\pd v)&=(\alpha-\pd)\alpha^3 |v|^2 v+\pd(\alpha^3|v|^2 v-|\pd v|^2 \pd v)\\
&=:e_1+e_2.
\end{align*}
We will prove that $e_1$ and $e_2$ are small in suitable spaces. For $e_1$, we directly estimate
\begin{align*}
\biggl\|\int_0^t e^{i(t-s)\Delta}e_1(s) \,ds\biggr\|_{S(\R)}&=\bigl\|\tfrac 1\alpha(\alpha-\pd)G\bigr\|_{S(\R)}\\
&\le \bigl\|\tfrac 1\alpha (\alpha-\pd)\ph G\bigr\|_{S(\R)}+\bigl\|\tfrac 1\alpha(\alpha-\pd)\pl G\bigr\|_{S(\R)}\\
&\quad+\bigl\|\tfrac 1\alpha(\alpha-\pd)\pmd G\bigr\|_{S(\R)}.
\end{align*}
As $(\alpha-\pd)$ is a Mikhlin multiplier and $\alpha>\eta_0^{1/6}$, taking $\eta_0$ sufficiently small and using \eqref{G}, we obtain
\begin{align*}
 \bigl\|\tfrac 1\alpha (\alpha-\pd)\ph G\bigr\|_{S(\R)}+\bigl\|\tfrac 1\alpha(\alpha-\pd)\pl G\bigr\|_{S(\R)}\le C(M)\eta_0^{\frac 13} \eta_0^{-\frac 16}\le C(M)\eta_0^{\frac 16}.
\end{align*}
To estimate the last term, we note that
\begin{align*}
\|(\alpha-\pd)\pmd\|_{L^2\to L^2}\lsm \frac{\log_2(\frac 1{\eta_1})}{\log_2(D_1)}\le \eta_0^2,
\end{align*}
by taking $D_1$ sufficiently large depending on $\eta_0$ and $\eta_1$.  Thus, by Strichartz and \eqref{E:vbndd},
\begin{align*}
\bigl\|\tfrac 1\alpha(\alpha-\pd)\pmd G\bigr\|_{S(\R)}&\le \|\alpha^3(\alpha-\pd)\pmd(|v|^2 v)\|_{L_t^1L_x^2}\\
&\lsm \eta_0^2 \||v|^2 v\|_{L_t^1L_x^2}\\
&\lsm \eta_0^2 \|v\|_{L_t^3L_x^6}^3\\
&\le C(M) \eta_0^2.
\end{align*}
Putting the three pieces together, we obtain
\begin{align*}
\bigl \|\int_0^t e^{i(t-s)\Delta} e_1(s) ds\bigr\|_{S(\R)}\le C(M)\eta_0^{\frac 16}.
\end{align*}

To estimate $e_2$, we use Strichartz and \eqref{E:vbndd}, as follows:
\begin{align*}
\|e_2\|_{N(\R)}&\le \|\alpha^3 |v|^2 v-|\pd v|^2\pd v\|_{L_t^{\frac 32}L_x^{\frac 65}}\\
&\le \|\alpha v-\pd v\|_{L_t^\infty L_x^2}(\|v\|_{L_t^3L_x^6}^2+\|\pd v\|_{L_t^3 L_x^6}^2)\\
&\le C(M)\bigl[\|(\alpha-\pd)\pmd v\|_{L_t^\infty L_x^2}+\|\pl v\|_{L_t^\infty L_x^2}+\|\ph v\|_{L_t^\infty L_x^2}\bigr]\\
&\le C(M)(\eta_0^2+\eta_0^{\frac 13})\\
&\le C(M)\eta_0^{\frac 13}.
\end{align*}

Invoking Lemma \ref{lm:stab}, we conclude that by taking $\eta_0$ sufficiently small depending only on $M$, there exists a unique global solution to $\pds$ such that
\begin{align*}
\|u\|_{S(\R)}\le C(M).
\end{align*}
This completes the proof of Proposition \ref{prop:onebubble}.
\end{proof}

Finally, we conclude this section with the following

\begin{proof}[Proof of Proposition \ref{prop:twobubble}]
Our basic small parameter is $\eps\ll 1$, which will be chosen later. If $\eta_1$ is sufficiently small (we need $\frac 1{\eta_1}>\eps^{-2}+\eps^{-4\eps^{-2}}$), we can find $\eps^{-2}$ disjoint intervals
\begin{align*}
[\eps^2 N_j, \eps^{-2} N_j]\subset [N_0, \tfrac 1{\eta_1} N_0],
\end{align*}
where $N_j=\eps^{-4j}N_0$ for $j\geq 1$.  Consequently, there is an $N_j$ such that
\begin{align}\label{sm}
\|P_{\eps^2 N_j\le\cdot<\eps^{-2}N_j}u_0\|_2\lsm \eps.
\end{align}

Define
\begin{align*}
\pl :=P_{\le \eps N_j} \qtq{and} \ph:=P_{>\eps^{-1}N_j}.
\end{align*}
Let $\uh$ be the solution  to $\pds $ with initial data $\uh(0)=\ph u_0$.  Let $\ul$ be the solution to $\pds$ with initial data $\ul(0)=\pl u_0$.  We have
\begin{align*}
\|\uh(0)\|_2\le M-\|P_{\le N_0}u_0\|_2\le M-\eta_0, \\
\|\ul(0)\|_2\le M-\|P_{> \frac {N_0}{\eta_1}} u_0\|_2\le M-\eta_0.
\end{align*}
By the inductive hypothesis, both $\uh$ and $\ul$ are global solutions and satisfy
\begin{align}
\|\ul\|_{S(\R)}\leq B \qtq{and} \|\uh\|_{S(\R)}\le B.\label{p2:0}
\end{align}

The implicit constants appearing in the remainder of the proof will all depend on $M$ and $B$ (but not on $\eps$ or $u_0$).  For ease of reading, we omit this from the notation.

By Bernstein, for any $s>0$,
\begin{align*}
\||\nabla|^s \ul(0)\|_2\lesssim (\eps N_j)^s \qtq{and} \||\nabla|^{-s}\uh(0)\|_2\lesssim (\eps^{-1}N_j)^{-s}.
\end{align*}
Therefore, by Lemmas \ref{lm:perpo} and \ref{lm:perne} we have
\begin{align}
\||\nabla|^s\ul\|_{S(\R)}\lesssim (\eps N_j)^s, \label{p2:l}\\
\|P_N \uh\|_{L_{t,x}^4}\lesssim N^{s} (\eps^{-1}N_j)^{-s}, \label{p2:h}
\end{align}
for any $0<s<\frac 12$.

Define $\tu:=\uh+\ul$.  We will show that $\tu$ is an approximate solution to $\pds$.  Note that $\tu$ satisfies
\begin{align*}
\begin{cases}
i\tu_t+\Delta \tu=\pd F(\pd \tu)+e,\\
\tu(0)=(\pl+\ph)u(0),
\end{cases}
\end{align*}
with $e=\pd F(\pd\uh)+\pd F(\pd\ul)-\pd F(\pd\tu)$.  By \eqref{sm} and \eqref{p2:0},
\begin{align*}
\|\tu(0)-u(0)\|_2\lsm \eps \qtq{and} \|\tu \|_{S(\R)}\lesssim 1.
\end{align*}
In order to apply Lemma \ref{lm:stab}, it suffices to show smallness of the error $e$.

Using \eqref{p2:0} and a dyadic decomposition, we estimate
\begin{align}
\|e\|_{L_{t,x}^{\frac 43}}&\lsm \bigl\| |\pd \uh|^2|\pd\ul|+|\pd \ul|^2|\pd \uh| \bigr\|_{L_{t,x}^{\frac 43}}\notag\\
&\lesssim \sum_{N\le \eps^{-\frac 12}N_j}\| (P_N \pd \uh) \pd \ul\|_{L_{t,x}^2}\label{p2:1}\\
\qquad&\qquad+\sum_{N>\eps^{-\frac 12}N_j}\|(P_N\pd \uh) P_{>N/8}\pd \ul\|_{L_{t,x}^2}\label{p2:2}\\
\qquad&\qquad+\sum_{N>\eps^{-\frac 12}N_j}\|(P_N\pd\uh) P_{\le N/8}\pd\ul\|_{L_{t,x}^2} .\label{p2:3}
\end{align}
To estimate \eqref{p2:1}, we use \eqref{p2:0} and \eqref{p2:h} as follows:
\begin{align*}
\eqref{p2:1}&\le \sum_{N\le \eps^{-\frac 12}N_j}\|P_N \pd \uh\|_{L_{t,x}^4}\|\pd \ul\|_{L_{t,x}^4} \lesssim (\eps^{-1}N_j)^{-s}\sum_{N\le \eps^{-\frac 12}N_j}N^s\lesssim \eps^{\frac s2}.
\end{align*}
To estimate \eqref{p2:2}, we use \eqref{p2:0} and \eqref{p2:l}:
\begin{align*}
\eqref{p2:2}&\le \sum_{N>\eps^{-\frac 12}N_j}\|\uh\|_{L_{t,x}^4}N^{-s}\||\nabla|^s\ul\|_{L_{t,x}^4} \lesssim \sum_{N>\eps^{-\frac 12}N_j}N^{-s}(\eps N_j)^s\lesssim \eps^{\frac 32 s}.
\end{align*}
To estimate the last term, we use the bilinear restriction estimate:  Choosing $p=\frac{11}6$ in Lemma~\ref{lm:bl1} and using Sobolev embedding, we get
\begin{align*}
\eqref{p2:3}&\lsm \sum_{N>\eps^{-\frac 12}N_j}\|(P_N \pd \uh) \ P_{\le N/8}\pd\ul\|_{L_{t,x}^{\frac{11}6}}^{\frac 12}\|P_{\le N/8}\ul\|_{L_{t,x}^{\frac{44}9}}^{\frac 12}\|P_N\uh\|_{L_{t,x}^4}^{\frac 12}\\
&\lesssim \sum_{N>\eps^{-\frac 12}N_j}N^{-\frac 2{11}\cdot\frac 12}\| |\nabla|^{\frac 2{11}}P_{\le N/8}\ul\|_{L_t^{\frac{44}9}L_x^{\frac{44}{13}}}^{\frac 12}\\
&\lesssim \sum_{N>\eps^{-\frac 12}N_j}N^{-\frac 1{11}}(\eps N_j)^{\frac 1{11}}\\
&\lesssim \eps^{\frac 3{22}}.
\end{align*}

Choosing $\eps$ sufficiently small depending on the implicit constants (that depend only on $M$ and $B$) we may then applying Lemma \ref{lm:stab} to conclude that the solution $u$ with initial data $u_0$ exists, is global, and obeys
\begin{align*}
\|u\|_{S(\R)}\leq C(M,B).
\end{align*}
This completes the proof of Proposition \ref{prop:twobubble}.
\end{proof}

\section{Nonlinear profile decomposition for $\pds$}\label{S:5}

The principal goal of this section is to prove the following theorem.  We remind the reader that the unitary transformations $G_n^j$ were defined in \eqref{G defn}.

\begin{thm}[Nonlinear profile decomposition for $\pds$]\label{T:npd}
Fix $M>0$ and $D\geq D_0(M)$ conforming to the requirements of Theorem~\ref{thm:scapds}.  Let $u_{0,n}$ be a sequence in $L^2(\R^2)$ with $\|u_{0,n}\|_2\leq M$ and let $u_n$ denote the sequence of solutions to $\pds$ with initial data $u_{n,0}$.  Suppose the linear profile decomposition of $u_{0,n}$ given by Theorem~\ref{thm:lpd} takes the form
\begin{align*}
u_{0,n}=\sum_{j=1}^ J G_n^j\phi^j +r_n^J
\end{align*}
and let $u_n^j$ denote the solution to $\pds$ with initial data $G_n^j\phi^j$.  Then
\begin{align*}
\lim_{J\to\infty} \limsup_{n\to\infty} \biggl\| u_n-  \biggl(\sum_{j=1}^ J u_n^j + e^{it\Delta} r_n^J \biggr) \biggr\|_{S(\R)} =0.
\end{align*}
\end{thm}

This theorem provides an asymptotic principle of superposition for our nonlinear equation.  To prove such an assertion, one needs to show that the sum of nonlinear solutions is almost a solution, which amounts to showing that cross terms in the nonlinearity (i.e., those coming from different profiles) are negligible in some way.  This in turn requires non-trivial structural information on the nonlinear solutions $u_n^j$, not merely uniform bounds.  Traditionally, the profiles for differing $n$ are related by a symmetry of the equation; the Fourier truncation breaks the scaling symmetry and so any possibility of such an exact relation when $N_n\not\equiv 1$.  The next two lemmas provide the requisite description of the nonlinear solutions associated to linear profiles conforming to Cases~I and~II of Theorem~\ref{thm:lpd}.

\begin{lem}[Case~I nonlinear profiles] \label{lm:npd1} Fix $M>0$ and let $D\ge D_0(M)$ as above and suppose $\phi\in L^2$ satisfies $\|\phi\|_2\le M$.  Let us further suppose that the parameters $(N_n,\xi_n, x_n, t_n)\in \R^+\times\R^2\times\R^2\times\R$ obey the following: {\upshape(i)} $t_n\equiv 0$ or $t_n\to \pm\infty$ and {\upshape(ii)} $N_n\to \infty$; or $N_n\to 0$ and $|\xi_n|\to \infty$; or $N_n\equiv1$ and $|\xi_n|\to \infty$.  If $u_n$ denotes the solution to $\pds$ with initial data $\phi_n:=G_n \phi$, then
\begin{align*}
\lim_{n\to \infty}\|u_n- v_n\|_{S(\R)}=0 \qtq{where} v_n := T_n[ e^{it\Delta} \phi ]
\end{align*}
and $T_n$ is as in \eqref{T defn}.  Moreover,
\begin{align}\label{lpr:1}
\lim_{n\to \infty}\|\pd u_n\|_{S(\R)}= 0.
\end{align}
\end{lem}

\begin{proof}
Evidently,
\begin{align*}
v_n(0)=\phi_n \qtq{and}  \|v_n\|_{S(\R)}\lsm \|\phi\|_2\lsm M.
\end{align*}
Moreover, $v_n$ solves $\pds$ up to an error
\begin{align*}
e_n=-\pd F(\pd v_n).
\end{align*}
Performing a change of variables, we first estimate
\begin{align*}
\|P_{\le 2D} v_n\|_{S(\R)}&=N_n\bigl \|P_{|\xi+\xi_n|\le 2D}\bigl\{\bigl[ e^{i(t_n+N_n^2 t)\Delta}\phi\bigl](N_n(y-x_n-2\xi_n t))\bigr\}\bigr\|_{S(\R)}\\
&=N_n \bigl\|\bigl\{e^{i(t_n+N_n^2 t)\Delta}\bigl[P_{|N_n\xi+\xi_n|\le 2D}\phi\bigr]\bigr\}(N_n x)\bigr\|_{S(\R)}\\
&\lsm \|P_{|\xi+\xi_n/N_n|\le \frac{2D}{N_n}}\phi\|_2.
\end{align*}
One can easily check that by dominating convergence,
\begin{align*}
\|P_{|\xi+\frac{\xi_n}{N_n}|\le \frac{2D}{N_n}}\phi\|_2\to 0 \qtq{as} n\to \infty,
\end{align*}
and so,
\begin{align}\label{lnp:2}
\lim_{n\to \infty}\|P_{\le 2D}v_n\|_{S(\R)}=0.
\end{align}
Using this, we estimate
\begin{align*}
\|e_n\|_{N(\R)}\le \|\pd v_n\|_{L_{t,x}^4}^3\le \|P_{\le 2D} v_n\|_{S(\R)}^3\to 0 \qtq{as} n\to \infty.
\end{align*}
An application of Lemma \ref{lm:stab} then yields
\begin{align*}
\lim_{n\to \infty}\|u_n-v_n\|_{S(\R)}=0.
\end{align*}
This completes the proof.
\end{proof}

\begin{lem}[Case~II nonlinear profiles] \label{lm:npd2}
Fix $M>0$ and let $D\ge D_0(M)$ as above and suppose $\phi\in L^2$ satisfies $\|\phi\|_2\le M$.  Let us further suppose that the parameters $(N_n,\xi_n, x_n, t_n)\in \R^+\times\R^2\times\R^2\times\R$ obey the following: {\upshape(i)} $t_n\equiv 0$ or $t_n\to \pm\infty$ and {\upshape(ii)} $N_n\to 0$ and $\xi_n\to \xi_\infty\in \R^2$.  If $u_n$ denotes the solution to $\pds$ with initial data $\phi_n:=G_n \phi$, then there exists $v\in S(\R)$ with $\|v\|_{S(\R)}\leq C(M)$ so that setting $v_n:=T_n v$ we have
\begin{align}\label{lpr:2}
\lim_{n\to \infty}\|u_n- v_n\|_{S(\R)}=0 \qtq{and also}  \lim_{n\to \infty}\|[\pd-m_D(\xi_\infty)] u_n\|_{S(\R)}= 0.
\end{align}
Moreover, if $t_n\equiv 0$ and $\xi_n\equiv 0$, then one may choose $v$ to be the solution to NLS with initial data $\phi$.
\end{lem}

\begin{proof}
Set $\alpha=m_D(\xi_\infty)$.  When $t_n\equiv0$, let $v$ be the global solution to
$\nalf$ with initial data $v(0)=\phi$.  (Note that this choice of $v$ is consistent with the final claim in the lemma.) When $t_n\to \pm\infty$, let $v$ be the global solution to $\nalf$ that scatters to $e^{it\Delta}\phi$ as $t\to \pm \infty$.  By Lemma~\ref{lm:ub} and construction,
\begin{align*}
\|v_n\|_{S(\R)}\lsm 1 \qtq{and} \lim_{n\to\infty}\|v_n(0)-\phi_n\|_2=0;
\end{align*}
here and below, the implicit constant depends only $M$.  Moreover, $v_n$ solves $\pds$ with error
\begin{align*}
e_n&=\alpha^4 F(v_n)-\pd F(\pd v_n)\\
&=(\alpha-\pd)\alpha^3|v_n|^2 v_n+\pd\bigl[\alpha^3 |v_n|^2 v_n-|\pd v_n|^2 \pd v_n].
\end{align*}
We estimate each summand in $e_n$ separately, starting with the first one.  Commuting the Galilei boost outermost, we have
\begin{align}
\|(\alpha-\pd) &\alpha^3 |v_n|^2 v_n\|_{L_t^1L_x^2}\notag\\
&= \alpha^3N_n^3 \bigl\| [\alpha-m_D(-i\nabla +\xi_n)]\bigl\{\bigl(|v|^2 v\bigr)(t_n+tN_n^2, N_n x)\bigr\}\bigr\|_{L_t^1L_x^2}\notag\\
&\le N_n^3 \bigl\|[\alpha-m_D(-i {\nabla}+\xi_n)]P_{\le \eps}\bigl\{\bigl( |v|^2 v\bigr)(t_n+t N_n^2, N_n x)\bigr\}\|_{L_t^1 L_x^2}\label{lnp:3}\\
&\quad +N_n^3\bigl\|[\alpha-m_D(-i{\nabla}+\xi_n)]P_{>\eps}\bigl\{\bigl( |v|^2 v\bigr)v(t_n+t N_n^2, N_n x)\bigr\}\|_{L_t^1 L_x^2}.\label{lnp:4}
\end{align}
Using the trivial Lipschitz bound on $m_D$, we have
\begin{align*}
\|(\alpha-m_D(-i\nabla+\xi_n))P_{\le \eps}\|_{L^2\to L^2}&\lesssim \tfrac1{\log_2(2D)} \|\xi+\xi_n-\xi_\infty\|_{L^\infty_\xi(\{|\xi|\leq2\eps\})}
	\lesssim \tfrac{\eps+|\xi_n-\xi_\infty|}{\log_2(2D)},
\end{align*}
and so may estimate
\begin{align*}
 \eqref{lnp:3}&\le \tfrac{\eps+|\xi_n-\xi_\infty|}{\log_2(2D)} N_n^3 \|(|v|^2 v)(t N_n^2+t_n, N_n x)\|_{L_t^3 L_x^6}^3\\
 &\lsm \tfrac{\eps+|\xi_n-\xi_\infty|}{\log_2(2D)}\|v\|_{L_t^3L_x^6}^3\\
 &\lsm \tfrac{\eps+|\xi_n-\xi_\infty|}{\log_2(2D)}.
 \end{align*}
 Similarly, we estimate
 \begin{align*}
 \eqref{lnp:4}&\lsm N_n^3\bigl\|P_{>\eps}\bigl[(|v|^2 v)(t_n+tN_n^2, N_nx)\bigr]\bigr\|_{L_t^1L_x^2}\\
 &\lsm \|P_{>\frac \eps{N_n}} (|v|^2 v)\|_{L_t^1L_x^2}\\
 &\lsm \|P_{>\frac{\eps}{8N_n}} v\|_{L_t^3L_x^6}\|v\|_{L_t^3L_x^6}^2\\
 &\lsm \|P_{>\frac {\eps}{8N_n}}v\|_{L_t^3L_x^6}.
 \end{align*}
 Putting these pieces together and using $D\geq 1$, we derive
 \begin{align}\label{lpr:5}
 \|(\alpha-\pd) \alpha^3|v_n|^2 v_n\|_{N(\R)}\lsm \eps+|\xi_n-\xi_\infty|+\|P_{>\frac{\eps}{8N_n}}v\|_{L_t^3L_x^6} .
 \end{align}

Next we estimate the second summand in $e_n$. We have
\begin{align*}
\|\pd [\alpha^3 |v_n|^2 v_n -|\pd v_n|^2\pd v_n]\|_{N(\R)}&\lsm \|\alpha^3 |v_n|^2 v_n-|\pd v_n|^2\pd v_n\|_{L_t^{\frac 32}L_x^{\frac 65}}\\
&\lsm \|(\alpha-\pd)v_n\|_{L_t^\infty L_x^2}\|v_n\|_{L_t^3L_x^6}^2\\
&\lsm \|(\alpha-\pd) v_n\|_{L_t^\infty L_x^2}.
\end{align*}
Arguing as above, we obtain
\begin{align*}
\|(\alpha-\pd) v_n\|_{L_t^\infty L_x^2}\le \eps+|\xi_n-\xi_\infty|+\|P_{>\frac \eps{N_n}}v\|_{L_t^\infty L_x^2},
\end{align*}
and so,
\begin{align*}
\|\pd[\alpha^3 |v_n|^2 v_n -|\pd v_n|^2\pd v_n]\|_{N(\R)}\lsm \eps+ |\xi_n-\xi_\infty|+\|P_{>\frac{\eps}{N_n}} v\|_{S(\R)} .
\end{align*}

Putting everything together, we get
\begin{align*}
\|e_n\|_{N(\R)}\lsm \eps+|\xi_n-\xi_\infty| +\|P_{>\frac{\eps}{8N_n}} v\|_{S(\R)}.
\end{align*}
Taking $n\to \infty$ and $\eps\to 0$ we obtain
\begin{align*}
\lim_{n\to \infty} \|e_n\|_{S(\R)}=0.
\end{align*}
Consequently, Lemma~\ref{lm:stab} yields
\begin{align*}
\lim_{n\to \infty} \|u_n-v_n\|_{S(\R)}=0.
\end{align*}

Finally we remark that using
\begin{align*}
\lim_{n\to \infty}\|(\alpha-\pd)\phi_n\|_2=0
\end{align*}
together with \eqref{lpr:5} and the Strichartz inequality, we derive
\begin{align*}
\lim_{n\to \infty}\|(\alpha-\pd)v_n\|_{S(\R)}=0,
\end{align*}
from which \eqref{lpr:2} follows.  This completes the proof of the lemma.
\end{proof}

The structure of nonlinear profiles associated to linear profiles of type III is relatively trivial; however, we state it as a lemma for easier referencing when we use this information in the next section.

\begin{lem}[Case~III nonlinear profiles] \label{lm:npd3}
Fix $M>0$ and let $D\ge D_0(M)$ as above and suppose $\phi\in L^2$ satisfies $\|\phi\|_2\le M$.  Assume also that $N_n\equiv 1$, $\xi_n\equiv 0$, $x_n\in\R^2$, and $t_n\in\R$ with $t_n\equiv 0$ or $t_n\to \pm\infty$.  If $u_n$ denotes the solution to $\pds$ with initial data $\phi_n:=G_n \phi$, then there exists $v\in S(\R)$ with $\|v\|_{S(\R)}\leq C(M)$ so that setting $v_n:=T_n v$ we have
\begin{align}\label{lpr:33}
\lim_{n\to \infty}\|u_n- v_n\|_{S(\R)}=0 .
\end{align}
\end{lem}

\begin{proof}
If $t_n\equiv0$, let $v$ be the global solution to $\pds$ with initial data $\phi$.  If $t_n\to \pm \infty$, let $v$ be the global solution to $\pds$ which scatters to $e^{it\Delta} \phi$ as $t\to \pm \infty$.  Note that both $T_n v$ and $u_n$ obey the same equation.  When $t_n\equiv 0$, they have the same initial data and so \eqref{lpr:33} holds trivially.  When $t_n\to\pm\infty$, we have instead
$$
\| u_n(0) - [T_n v](0)\|_{L^2(\R^2)} = \| e^{it_n\Delta} \phi - v(t_n) \|_{L^2(\R^2)} \to 0 \qtq{as} n\to\infty
$$
by the construction of $v$.  Thus, \eqref{lpr:33} also holds in that case in view of Lemma~\ref{lm:stab}.
\end{proof}

\begin{proof}[Proof of Theorem~\ref{T:npd}]
Let $u_{0,n}$ be a sequence of $L^2$ functions such that
\begin{align*}
\sup_n\|u_{0,n}\|_2\le M.
\end{align*}
Fix $D\geq D_0(M)$ and let $u_n$ be the global solutions to $\pds$ with initial data $u_n(0)=u_{0,n}$.  By Theorem~\ref{thm:scapds}, $u_n$ satisfies
\begin{align*}
\sup\|u_n\|_{S(\R)}\le C(M).
\end{align*}

Applying the linear profile decomposition Theorem~\ref{thm:lpd} and passing to a subsequence, we write
\begin{align}\label{nfp:00}
u_{0,n}&=\sum_{j=1}^J G_n^j \phi_j+ r_n^J=:\sum_{j=1}\phi_n^j +r_n^J.
\end{align}
with the properties stated in that lemma.  In particular, we have the mass decoupling:
\begin{align}\label{nfp:1}
\lim_{n\to \infty}\|u_{0,n}\|_2^2-\sum_{j=1}^ J\|\phi^j\|_2^2 -\|r_n^J\|_2^2=0 \qtq{for any} 1\le J\le J^*.
\end{align}

We now discuss the nonlinear profiles associated to each linear profile.  As in the statement of the theorem, we write $u_n^j$ for the global solution to $\pds$ with initial data $u_n^j(0)=\phi_n^j$.  The existence of such global solutions is guaranteed by Theorem~\ref{thm:scapds}.

If $j$ conforms to Case~I, then by Lemma~\ref{lm:npd1},
\begin{align}\label{nfp:5}
  \lim_{n\to \infty}\|u_n^j -T_n^j v^j\|_{S(\R)}+\|\pd T_n^j v^j\|_{S(\R)}+\|\pd u_n^j\|_{S(\R)}=0,
\end{align}
where $v^j\in S(\R)$ is independent of $n$.

If $j$ conforms to Case~II, Lemma~\ref{lm:npd2} guarantees that
\begin{align}\label{nfp:6}
 \lim_{n\to \infty}\|u_n^j-T_n^j v^j\|_{S(\R)}+\|(\pd-\alpha)T_n^j v^j\|_{S(\R)}+ \|(\pd-\alpha)u_n^j\|_{S(\R)}=0
\end{align}
for some $v^j\in S(\R)$ independent of $n$.

If $j$ conforms to Case~III, Lemma~\ref{lm:npd3} guarantees that
\begin{equation}\label{nfp:?}
\lim_{n\to \infty}\|u_n^j-T_n^j v^j\|_{S(\R)} =0
\end{equation}
for some $v^j\in S(\R)$ independent of $n$.

The asymptotic orthogonality of parameters implies asymptotic orthogonality of the nonlinear profiles. Indeed, we have

\begin{lem}[Decoupling of nonlinear profiles]\label{lm:dcp} For any $j\neq k$
\begin{gather}
\lim_{n\to \infty}\|T_n^j v^j T_n^k v^k\|_{L_{t,x}^2}=0,\label{dcp:1}\\
\lim_{n\to \infty}\|u_n^j u_n^k\|_{L_{t,x}^2}=0,\label{dcp:2}\\
\lim_{n\to \infty}\|\pd(u_n^j)\pd(u_n^k)\|_{L_{t,x}^2}=0. \label{dcp:3}
\end{gather}
\end{lem}

\begin{proof} The proof of \eqref{dcp:1} can be effected by elementary manipulations; see, for example, \cite[Lemma~3.42]{MerleVega}. The argument relies only on the fact that $v^j,v^k\in L^4(\R\times\R^2)$; we will exploit this greater generality below.

Claim \eqref{dcp:2} follows from \eqref{dcp:1} and the asymptotic agreement between $u_n^j $ and $T_n^j v^j$ provided by \eqref{nfp:5}, \eqref{nfp:6}, and \eqref{nfp:?}.

It remains to prove \eqref{dcp:3}.  Exploiting the $j\leftrightarrow k$ symmetry, we see that it suffices to consider four cases:

Case 1: $j$ conforms to Case I.  In this case, using \eqref{nfp:5} we get
\begin{align*}
\|(\pd u_n^j)(\pd u_n^k)\|_{L_{t,x}^2} &\le \|\pd u_n^j\|_{L_{t,x}^4}\|u_n^k\|_{L_{t,x}^4} \leq C(M) \|\pd u_n^j\|_{L_{t,x}^4}\to 0
\end{align*}
as $n\to \infty$.

Case 2: Both $j$ and $k$ conform to Case II.  By the triangle inequality,
\begin{align*}
&\|(\pd u_n^j)(\pd u_n^k)\|_{L_{t,x}^2}\\
&\quad\le \|[\pd-m_D(\xi^j)]u_n^j (\pd u_n^k)\|_{L_{t,x}^2}\\
&\qquad + m_D(\xi^j)\|u_n^j [\pd-m_D(\xi^k)]u_n^k\|_{L_{t,x}^4}+m_D(\xi^j)m_D(\xi^k)\|u_n^j u_n^k\|_{L_{t,x}^2}\\
&\quad\le C(M) \|[\pd-m_D(\xi^j)]u_n^j\|_{S(\R)} + C(M) \|[\pd-m_D(\xi^k)]u_n^k\|_{S(\R)} + \|u_n^j u_n^k\|_{L_{t,x}^2},
\end{align*}
which converges to zero as $n\to \infty$ by \eqref{nfp:6} and \eqref{dcp:2}.

Case 3:  $j$ conforms to Case II and $k$ conforms to Case III.  In this case we have
$$
\lim_{n\to \infty}\|\pd u_n^k-T_n^k(\pd v^k)\|_{L_{t,x}^4}=0
$$
and so,
\begin{align*}
\|\pd u_n^j \pd u_n^k\|_{L_{t,x}^2}
&\le \|[\pd-m_D(\xi^j)]u_n^j \pd u_n^k\|_{L_{t,x}^2}+m_D(\xi^j)\|(u_n^j-T_n^j v^j)\pd u_n^k\|_{L_{t,x}^2}\\
&\quad+m_D(\xi^j)\|T_n^j v^j (\pd u_n^k-T_n^k (\pd v^k))\|_{L_{t,x}^2}\\
&\quad +m_D(\xi^j)\|T_n^j v^j T_n^k (\pd v^k)\|_{L_{t,x}^2}\\
&\le C(M) \|[\pd-m_D(\xi^j)]u_n^j\|_{L_{t,x}^4} + C(M)  m_D(\xi^j)\|u_n^j-T_n^j v^j\|_{L_{t,x}^4} \\
&\quad +C(M)m_D(\xi^j)\|\pd u_n^k-T_n^k(\pd v^k)\|_{L_{t,x}^4}\\
&\quad+m_D(\xi^j)\|T_n^j v^j T_n^k (\pd v^k)\|_{L_{t,x}^2},
\end{align*}
which converges to zero as $n\to \infty$ by \eqref{nfp:6} and the extended version of \eqref{dcp:1}.

Case 4: Both $j$ and $k$ conform to Case III.  In this case we have
$$
\lim_{n\to \infty}\|\pd u_n^j-T_n^j(\pd v^j)\|_{L_{t,x}^4}+\lim_{n\to \infty}\|\pd u_n^k-T_n^k(\pd v^k)\|_{L_{t,x}^4}=0.
$$
The claim follows again from the extended version of \eqref{dcp:1}.

This completes the proof of Lemma~\ref{lm:dcp}.
\end{proof}

To continue with the proof of Theorem~\ref{T:npd}, we define
\begin{align}\label{nfp:80}
u_n^J:=\sum_{j=1}^J u_n^j +e^{it\Delta} r_n^J.
\end{align}
In view of Lemma~\ref{lm:stab}, the theorem will follow if we can verify the following three claims about $u^J_n$:

Claim 1.  $\|u_n^J(0)-u_n(0)\|_2\to 0$ as $n\to \infty$, for any $J$.

Claim 2. $\limsup _{n\to \infty}\|u_n^J\|_{L_{t,x}^4}\lsm_M 1$ uniformly in $J$.

Claim 3. $\lim_{J\to \infty}\limsup_{n\to\infty}\|(i\partial_t+\Delta) u_n^J-\pd F(\pd u_n^J)\|_{N(\R)}=0. $

The first claim is trivial: by construction $u_n(0)=u_n^J(0)$.

To prove Claim 2, we first show that
\begin{align}\label{nfp:8}
\sum_{j=1}^J \|u_n^j\|_{L_{t,x}^4}^4\lsm_M 1\qtq{uniformly in} J.
\end{align}
Indeed, from the mass decoupling \eqref{nfp:1}, we know that
\begin{align}\label{nfp:9}
\sum_{j=1}^{J^*} \mathcal M(u_n^j)=\sum_{j=1}^{J^*}\mathcal M(\phi^j) < \infty.
\end{align}
Therefore, there exists $J_0$ such that for $j>J_0$ we have $\mathcal M(\phi^j)< \eps_0$, where $\eps_0$ is the constant in the small data theory presented in Lemma~\ref{lm:loc}.  This lemma then guarantees that
\begin{align*}
\|u_n^j\|_{L_{t,x}^4}^4\lsm \mathcal M(\phi^j)^2 \qtq{for all} j\ge J_0.
\end{align*}
As $J_0$ is finite, we can find a uniform constant depending only on $M$ such that
\begin{align}\label{nfp:10}
\|u_n^j\|_{L_{t,x}^4}^4\lsm_M \mathcal M(\phi^j) \qtq{for all} j\ge 1.
\end{align}
Combining this with \eqref{nfp:9} and \eqref{nfp:10}, we derive \eqref{nfp:8}.  Consequently,
\begin{align*}
\|u_n^J\|_{L_{t,x}^4}^4&\lsm \biggl\|\sum_{j=1}^J u_n^j\biggr\|_{L_{t,x}^4}^4+\|e^{it\Delta}r_n^J\|_{L_{t,x}^4}^4\\
&\lsm \sum_{j=1}^J \|u_n^j\|_{L_{t,x}^4}^4+C(J)\sum_{j\neq k}\iint_{\R^2} |u_n^j||u_n^k|^3+M^4\\
&\lsm_M 1+C(J)\sum_{j\neq k}\|u_n^j u_n^k\|_{L_{t,x}^2}\\
&\lsm_M 1+C(J)o(1) \qtq{as} n\to \infty.
\end{align*}
This completes the proof of Claim 2.

Next we verify Claim 3. A direct computation yields
\begin{align*}
(i\partial_t+\Delta) u_n^J -\pd F(\pd u_n^J)
&=\sum_{j=1}^J \pd F(\pd u_n^j)-\pd F\biggl(\pd\sum_{j=1}^J u_n^j\biggr)\\
&\quad +\pd F(\pd(u_n^J-e^{it\Delta}r_n^J))-\pd F(\pd u_n^J).
\end{align*}
We estimate the two summands separately.  By Lemma \ref{lm:dcp},
\begin{align*}
\biggl\|\sum_{j=1}^J\pd F(\pd u_n^j)-\pd F\biggl(\pd \sum_{j=1}^J u_n^j\biggr)\biggr\|_{N(\R)}
&\lsm_J \sum_{j\neq k}\|\pd u_n^j |\pd u_n^k|^2\|_{L_{t,x}^{\frac 43}}\\
&\lsm_J \sum_{j\neq k}\|\pd u_n^j \pd u_n^k\|_{L_{t,x}^2}\|\pd u_n^k\|_{L_{t,x}^4}\\
&\lsm_{J,M} o(1) \qtq{as} n\to \infty.
\end{align*}
Thus for any $J$,
\begin{align*}
\lim_{n\to\infty}\biggl\|\sum_{j=1}^J\pd F(\pd u_n^j)-\pd F\biggl(\pd\sum_{j=1}^J u_n^j\biggr)\biggr\|_{N(\R)}=0.
\end{align*}
We now turn to estimating the second term in the error. We have
\begin{align*}
\|\pd F(\pd(u_n^J&-e^{it\Delta}r_n^J))-\pd F(\pd u_n^J)\|_{N(\R)}\\
&\le \|\pd e^{it\Delta}r_n^J\|_{L_{t,x}^4}\bigl (\|u_n^J\|_{L_{t,x}^4}^2+\|e^{it\Delta}r_n^J\|_{L_{t,x}^4}^2\bigr)\\
&\lsm_M \|e^{it\Delta} r_n^J\|_{L_{t,x}^4},
\end{align*}
which converges to zero as $n\to \infty$ and $J\to \infty$.  This completes the proof of Claim~3 and so also the proof of Theorem~\ref{T:npd}.
\end{proof}

\section{Approximation in the weak $L^2$ topology}\label{S:6}

In \cite[\S5]{BG99}, Bahouri and G\'erard show how a nonlinear profile decomposition can be used to prove well-posedness in the weak topology (in the setting of the energy-critical wave equation).  The analogue of their result for (NLS) would be the following:

\begin{thm}[Well-posedness in the weak topology]\label{BG-like}
Suppose we have (finite-mass) solutions $u_n$ and $u_\infty$ to NLS that satisfy
\begin{align*}
u_n(0) \rightharpoonup u_{\infty}(0) \text{ weakly in $L^2(\R^2)$}.
\end{align*}
Then
\begin{align*}
u_n(t)\rightharpoonup u_\infty(t) \quad\text{weakly in $L^2(\R^2)$}
\end{align*}
for all $t\in\R$.
\end{thm}

For our purposes, we need a variant of this type of statement, where the underlying equation changes as well.  We omit the proof of Theorem~\ref{BG-like} since it is easily reconstructed from the (more complicated) proof of this variant:

\begin{thm}[Approximation in the weak topology]\label{thm:wc}
Fix $M>0$, $D\ge D_0(M)$, and a sequence $M_n\to \infty$.  Assume that $\{u_{0,n}\}\subset L^2$ satisfy
\begin{align*}
\|u_{0,n}\|_2\le M \qtq{and}  u_{0,n}\rightharpoonup u_{0,\infty} \text{ weakly in $L^2$}.
\end{align*}
Let $u_n$ be the solutions to
\begin{align*}(\pnns)\quad
\begin{cases}
(i\partial_t+\Delta) u_n=\pmn F(\pmn u_n)\\
u_n(0,x)=u_{0,n}(x),
\end{cases}
\end{align*}
where $\pmn$ denotes the Fourier multiplier with symbol $m_D(\xi/M_n)$ and let $u_\infty$ be the global solution to NLS with initial data $u_{0,\infty}$. Then $u_n$ are defined globally in time and for any $t\in \R$,
\begin{align*}
u_n(t)\rightharpoonup u_\infty(t) \quad\text{weakly in } L^2(\R^2).
\end{align*}
\end{thm}

\begin{proof}
Recurrent objects in this argument will be the parameters
$$
N_n^0:=M_n^{-1},\quad \xi_n^0\equiv 0,\quad x_n^0\equiv 0,\qtq{and} t_n^0\equiv 0,
$$
together with the associated operators
$$
G_n^0 f := M_n^{-1} f(M_n^{-1} x) \qtq{and} [T_n^0 v](t,x) :=  M_n^{-1} v( M_n^{-2} t , M_n^{-1} x)
$$
defined as in \eqref{G defn} and \eqref{T defn}.

Let $w_n$ denote the solution to
\begin{align*}
 \begin{cases}
 (i\partial_t+\Delta )w_n=\pd F(\pd w_n),\\
 w_n(0)=G_n^0 u_{0,n}.
 \end{cases}
\end{align*}
By Theorem~\ref{thm:scapds} this solution is global and so
$$
u_n := (T_n^0)^{-1} w_n
$$
is the solution to the initial value problem $(\pnns)$.  This proves that $u_n$ is global, as stated in the theorem.

Let $f_n:=G_n^0 (u_{0,n}-u_{0,\infty})=w_n(0) - G_n^0 u_{0,\infty}$, which by assumption satisfies
\begin{align}\label{wc:5}
(G_n^0)^{-1} f_n\rightharpoonup 0 \quad\text{weakly in } L^2.
\end{align}
Applying the  linear profile decomposition to $\{f_n\}$, we write
\begin{align}\label{wc:5.1}
f_n=\sum_{j=1}^J G_n^j \phi^j +r_n^J
\end{align}
with the properties stated in Theorem~\ref{thm:lpd}. In  particular, for each $\phi^j$ we have
\begin{align}\label{wc:5.2}
\phi^j=\wlim_{n\to \infty}{}(G_n^j)^{-1} r_n^{j-1},
\end{align}
with the convention $r_n^0=f_n$.  Further, let $\phi^0:= u_{0,\infty}$.  We now claim that
\begin{align}\label{wc:7.0}
 w_n(0)=\sum_{j=0}^ J G_n^j\phi^j +r_n^J
\end{align}
provides a linear profile decomposition of $w_n(0)$, with all the attributes given in Theorem~\ref{thm:lpd} with just one exception: in this setting, it is possible that $\phi^0\equiv 0$.  The justification of this claim is elementary once one affirms asymptotic orthogonality of parameters, which in turn means showing
\begin{align}\label{wc:6}
 (N_n^j, \xi_n^j,x_n^j, t_n^j)\perp (N_n^0, \xi_n^0, x_n^0, t_n^0) \qtq{for each} j\geq 1.
\end{align}

Let us pause to verify \eqref{wc:6}.  Supposing that it failed, we choose $k\geq 1$ to be the \emph{first} witness.  Then by \eqref{wc:5.2},
\begin{align*}
0\neq \phi^{k}&=\wlim_{n\to\infty} (G_n^{k})^{-1} r_n^{k-1}\\
  &=\wlim_{n\to \infty}(G_n^{k})^{-1}\biggl[f_n-\sum_{j=1}^{k-1}G_n^j\phi^j\biggr]\\
  &=\wlim_{n\to \infty} (G_n^{k})^{-1}G_n^0 \bigl[(G_n^0)^{-1} f_n\bigr] - \sum_{j=1}^{k-1} \wlim_{n\to \infty}(G_n^{k})^{-1}G_n^j\phi^j\\
  &=0.
\end{align*}
To deduce the last step we used the following: 1) the operators $(G_n^{k})^{-1}G_n^0$ converge strong-$*$ as $n\to\infty$ by assumption and  \eqref{wc:5} and 2) the operators $(G_n^{k})^{-1}G_n^j$ converge weakly to zero (as $n\to\infty$) due to orthogonality of the parameters $1\leq j < k$.

Continuing from the linear profile decomposition \eqref{wc:7.0}, we obtain an associated nonlinear profile decomposition of $w_n$ via Theorem~\ref{T:npd}:
$$
w_n(t,x) = \sum_{j=0}^J w_n^j + e^{it\Delta} r_n^J + o_{S(\R)}(1) \quad \text{as $n\to \infty$ and $J\to \infty$}.
$$
Next we employ Lemmas~\ref{lm:npd1}, \ref{lm:npd2}, and~\ref{lm:npd3} to replace $w_n^j$ by $n$-dependent transformations applied to fixed functions $v^j$, up to a negligible error.  Note that the zeroth profile conforms to Case~II and Lemma~\ref{lm:npd2} shows that one may take $v^0=u_\infty$.  After these reductions, we deduce that $u_n=(T_n^0)^{-1}  w_n$ obeys
$$
\lim_{J\to\infty} \lim_{n\to\infty} \biggl\| u_n(t,x) - \biggl[ u_\infty(t,x) + \sum_{j=1}^J [(T_n^0)^{-1} T_n^j v^j](t,x) + [(T_n^0)^{-1} e^{it\Delta} r_n^J](x) \biggr\|_{S(\R)} = 0.
$$

By virtue of \eqref{wc:6} we have that
$$
\wlim_{n\to\infty} [(T_n^0)^{-1} T_n^j v^j](t) = 0
$$
in $L^2(\R^2)$ for each $t\in\R$.  Thus, to complete the proof of Theorem~\ref{thm:wc} we need only show that
$$
[(T_n^0)^{-1} e^{it\Delta} r_n^J](x) = e^{it\Delta} (G_n^0)^{-1} r_n^J = e^{it\Delta} \biggl[ (G_n^0)^{-1} f_n - \sum_{j=1}^J (G_n^0)^{-1} G_n^j \phi^j \biggr]
$$
converges weakly to zero as $n\to\infty$ for each fixed $J\geq 1$ and $t\in\R$.  The last equality written here is an application of \eqref{wc:5.1} and provides the key to verifying this assertion.  Indeed, we simply apply \eqref{wc:5} and use \eqref{wc:6} to see that $(G_n^0)^{-1} G_n^j\to 0$ in the weak operator topology.
\end{proof}

 \section{Local-in-time dispersive estimates on the torus}\label{S:7}

In this section we establish local-in-time dispersive and Strichartz estimates for the linear Schr\"odinger flow on the torus. We will show that if the size of the torus is sufficiently large depending on the frequency localization and time for which we seek the estimates, we have the full range of Strichartz estimates.

We will use the following lemma.

\begin{lem}[Midpoint rule error]\label{lm:diffcell}
Given $\xi_0\in\R^d$ and $L>0$, let $Q$ denote the cube $Q=\xi_0+[-\tfrac 1{2L},\tfrac1{2L})^d$ of side-length $L^{-1}$ centered at $\xi_0$. Then
\begin{align}\label{diffcell:1}
\biggl|\int_{Q} h(\xi) \,d\xi -\tfrac 1{L^d} h(\xi_0)\biggr| \lsm \tfrac{1}{L^{d+2}} \|\partial^2 h\|_{L^\infty(Q)}.
\end{align}
\end{lem}

\begin{proof} Performing a Taylor expansion, we get
\begin{align*}
h(\xi)=h(\xi_0)+\nabla h(\xi_0)\cdot (\xi-\xi_0)+r(\xi) \qtq{with} |r(\xi)|\lsm \|\partial^2 h\|_\infty |\xi-\xi_0|^2.
\end{align*}
Thus,
\begin{align*}
\biggl|\int_Q h(\xi) \,d\xi &-\tfrac 1{L^d}h(\xi_0)\biggr| =\biggl|\int_Q r(\xi)\,d\xi \biggr|\lsm \tfrac{1}{L^{d+2}} \|\partial^2 h\|_{L^\infty(Q)}.
\end{align*}
This completes the proof of the lemma.
\end{proof}

\begin{prop}[Local-in-time dispersive and Strichartz estimates]\label{pp:stri}
Given $T>0$ and $1\leq N\in 2^\Z$, there exists $L_0=L_0(T,N)\geq 1$ sufficiently large so that for $L\ge L_0$,
\begin{align}
&\|e^{it\Delta}P^L_{\le N}f\|_{L_x^{\infty}(\T_L)}\lsm |t|^{-1}\|f\|_{L_x^1(\T_L)} \quad\text{uniformly for $t\in [-T,T]\setminus\{0\}$},\label{str:01}\\
&\|e^{it\Delta }P_{\le N}^L f\|_{L_t^q L_x^r([-T,T]\times\T_L)}\lsm_{p,q} \|f\|_{L_x^2(\T_L)}. \label{str:02}
\end{align}
Here, $\T_L=\R^2/L\Z^2$, $(q,r)$ is a Schr\"odinger admissible pair, in the sense that
\begin{align*}
\tfrac 2q+\tfrac 2r=1 \qtq{with} 2< q \leq \infty,
\end{align*}
and $P_{\le N}^L$ denotes the Fourier multiplier $P_{\leq N}$ on $\T_L$.
\end{prop}

\begin{proof}
We will present the details for the dispersive estimate \eqref{str:01}.  The Strichartz estimates \eqref{str:02} follow from this via the usual $TT^*$ argument.

The convolution kernel associated to the operator  $e^{it\Delta}P_{\le N}^L$ is
\begin{align*}
k(x):=\tfrac 1{L^2}\sum_{n\in \Z^2} e^{i\Phi(n)} \phi \bigl(\tfrac n{NL}\bigr) \qtq{with} \Phi(n):= 2\pi x\cdot \tfrac nL  - 4\pi^2 t |\tfrac nL|^2 .
\end{align*}
We need to show
\begin{align}\label{str:2}
|k(x)| \lesssim |t|^{-1}
\end{align}
uniformly for $0<|t|\leq T$ and $x\in [-L/2,L/2]^2$.  To do this, we divide into two cases: $|x|\leq R$ and $|x|>R$ with $R=1 + NT$.
In the former case, we approximate $k$ by the Euclidean kernel (defined via an integral):
\begin{align*}
k(x)&= \sum_{n\in \Z^2} \biggl(\tfrac 1{L^2}  e^{2\pi i x\cdot\frac nL -4\pi^2i t(\frac nL)^2}\phi\bigl(\tfrac n{NL}\bigr)
	- \int_{Q_n} e^{2\pi i x\cdot\xi -4\pi^2i t|\xi|^2} \phi(\xi/N)\, d\xi\biggr)\\
&\quad + \int_{\R^2} e^{2\pi i x \cdot \xi -4\pi^2i t|\xi|^2}\phi(\xi/N)\,d\xi.
\end{align*}
where $Q_n$ is the cube centered at $n/L$ of side-length $1/L$.

The second summand is $O(|t|^{-1})$ by the usual Euclidean argument.  To estimate the first, we apply Lemma~\ref{lm:diffcell} to obtain the bound
$$
L^{-4} \sum_{|n|\lesssim NL} \bigl( R^2 + N^2 T^2 + T + N^{-2} \bigr) \lesssim L^{-4} R^2 (NL)^2\lsm T^{-1},
$$
provided we choose $L \gg (N^2 T^{\frac 32} + N\sqrt T)$.

It remains to consider those $|x|>R$ and $x\in [-\frac L2, \frac L2]^2$. By symmetry, we can also assume the $x_1$-variable
dominates so that $|x_1|>\frac{\sqrt 2} 2 R$.

Let $\vec e_1=\bigl(\begin{smallmatrix}1\\0\end{smallmatrix}\bigr)$. Using the identity
\begin{align*}
e^{i\Phi(n)}&=\frac{e^{i\Phi(n+\vec e_1)}-e^{i\Phi(n)}}{e^{i[\Phi(n+\vec e_1)-\Phi(n)]}-1} =: \frac{e^{i\Phi(n+\vec e_1)}-e^{i\Phi(n)}}{\Psi(n)}
\end{align*}
and applying a change of variables, we can write
\begin{align}
k(x)&=\tfrac 1{L^2} \sum_{n\in \Z^2} e^{i\Phi(n)}\biggl[\frac{\phi(\frac{n-\vec e_1}{NL})}{\Psi(n-\vec e_1)}-
\frac{\phi(\frac n{NL})}{\Psi(n)}\biggr]\notag\\
&=\tfrac 1{L^2} \sum_{n\in\Z^2} e^{i\Phi(n)}\tfrac 1{\Psi(n-\vec e_1)}\bigl[\phi\bigl(\tfrac{n-\vec e_1}{NL}\bigr)-\phi\bigl(\tfrac n{NL}\bigr)\bigr]\label{str:4}\\
&\quad+\tfrac 1{L^2} \sum_{n\in \Z^2}e^{i\Phi(n)}\phi\bigl(\tfrac n{NL}\bigr)\bigl[\tfrac 1{\Psi(n-\vec e_1)}-\tfrac 1{\Psi(n)}\bigr]\label{str:5}.
\end{align}

From
\begin{align*}
\Phi(n+\vec e_1)-\Phi(n)=2\pi \tfrac{x_1}L -8\pi^2 \tfrac t{L^2} n\cdot \vec e_1-4\pi^2 \tfrac t{L^2},
\end{align*}
and the definition of $R$ we have
\begin{align*}
|\Phi(n+\vec e_1)-\Phi(n)|&\ge \sqrt 2\pi \tfrac RL - O\bigl(\tfrac{NT}{L^2} + \tfrac T{L^2}\bigr) \ge \tfrac RL
\end{align*}
provided $L\gg 1$.
This lower bound also implies
\begin{align}\label{b2.5}
|\Psi(n-\vec e_1)|^{-1} + |\Psi(n)|^{-1}\lesssim \tfrac LR.
\end{align}
From the Fundamental Theorem of Calculus we have
\begin{align*}
\bigl|\phi\bigl(\tfrac{n-\vec e_1}{NL}\bigr)-\phi\bigl(\tfrac n{NL}\bigr)\bigr|&\lsm \tfrac 1{NL}\\
\bigl |\tfrac 1{\Psi(n-\vec e_1)}-\tfrac 1{\Psi(n)}\bigr|&\le \tfrac{|\Phi(n+e_1)-2\Phi(n)+\Phi(n-\vec e_1)|}{|\Psi(n)\Psi(n-\vec e_1)|}\lsm \tfrac T{R^2}.
\end{align*}
Inserting these estimates into \eqref{str:4} and \eqref{str:5}, we obtain
\begin{align*}
|k(x)|&\lsm \tfrac 1{L^2}(NL)^2\bigl[\tfrac L R\tfrac 1{NL}+\tfrac T{R^2}\bigr]\lsm \tfrac NR +\tfrac{N^2}{R^2} T\lsm \tfrac 1T.
\end{align*}
This completes the proof of the proposition.
\end{proof}

As a direct consequence of these local-in-time Strichartz estimates and the arguments used in the Euclidean case (cf. \cite{Clay}), we obtain the following result regarding perturbations of the frequency-localized NLS on the torus:
\begin{align}\label{stabt:1}
\begin{cases}
(i\partial_t+\Delta )u =P_{\le N}^ L \Po^L F(\Po^L u),\\
u(0)=P_{\le N}^L u_0.
\end{cases}
\end{align}
Here $\Po^L$ is a Mikhlin multiplier on the torus $\R^2/L\Z^2$.

\begin{prop}[Perturbation theory for NLS on the torus]\label{P:stab} Given $T>0$ and $1\leq N\in 2^\Z$, let $L_0$ be as in Proposition~\ref{pp:stri}. Fix $L\geq L_0$ and let $\tilde u$ be an approximate solution to \eqref{stabt:1} on $[-T,T]$ in the sense that
\begin{align*}
\begin{cases}
(i\partial_t+\Delta) \tilde u=P_{\le N}^L \Po^L F(\Po^L \tilde u) +e, \\
\tilde u(0)=P_{\le N}^L \tilde u_0
\end{cases}
\end{align*}
for some function $e$ and $\tilde u_0\in L^2(\T)$. Assume
\begin{align*}
\|\tilde u\|_{S([-T,T])}\le A
\end{align*}
and the smallness conditions
\begin{align*}
\|u_0-\tilde u_0\|_2\le \eps \qtq{and} \|e\|_{N([-T,T])}\le \eps.
\end{align*}
Then if $\eps\le \eps_0(A)$, there exists a unique solution $u$ to \eqref{stabt:1} such that
\begin{align*}
\|u-\tilde u\|_{S([-T,T])}\le C(A) \eps.
\end{align*}
\end{prop}

\section{Torus problem approximation}\label{S:8}

Fix $M>0$, $T>0$ and $D\geq D_0(M)$; recall that $D_0$ was introduced in Theorem~\ref{thm:scapds}.  Let $M_n\to \infty$ and $\eps_n\to 0$.  Let $L_n\geq L_{*}(M,D,M_n, \eps_n,T)$, where $L_*$ is a large constant determined through the arguments in this section.  Lastly, we define $\btn:=\R^2/L_n \Z^2$ and let $\pml$ denote the Fourier multiplier on $\btn$ with symbol $m_D(\cdot/M_n)$.

We consider the following sequence of finite-dimensional Hamiltonian systems
\begin{align}\label{per1}
\begin{cases}
(i\partial_t+\Delta)u_n =\pml F(\pml u_n), \quad (t,x)\in \R\times\btn\\
u_n(0)=u_{0,n}
\end{cases}
\end{align}
with initial data
\begin{align}\label{8.0}
u_{0, n}\in \mathcal H_n:=\{f\in L^2(\btn):\,P_{>2DM_n}^{L_n} f=0 \} \qtq{with} \|u_{0,n}\|_2\le M.
\end{align}
We will show that for $n$ sufficiently large, solutions to \eqref{per1} can be well approximated by solutions to the corresponding problem in $\R^2$ on the fixed time interval $[-T,T]$.  Note that as a finite-dimensional system with coercive Hamiltonian, \eqref{per1} automatically has global solutions.

\subsection{Choice of cutoffs}\label{SS:cutoffs}
We define one-dimensional cutoffs. The two-dimensional cutoff functions will be tensor products of these one-dimensional cutoff functions.

Let $\eta_n:=\eps_n^2$.  Subdivide the interval $[\frac{L_n}4,\frac{L_n}2]$ into at least $16 M^2/\eta_n$ many subintervals of length $20\frac 1{\eta_n} DM_n T$; this can be achieved since we may assume that
\begin{align}\label{cf:1}
L_n\gg \tfrac{M^2}{\eta_n} \cdot \tfrac 1{\eta_n} DM_nT.
\end{align}
By the pigeonhole principle, there exists a subinterval, which we denote by
\begin{align*}
I_n^1:=[c_n^1-\tfrac {10}{\eta_n} DM_n T, c_n^1+\tfrac {10}{\eta_n}DM_nT],
\end{align*}
that satisfies
\begin{align*}
\|u_{0,n}(x+L_n\Z^2)\chi_{I_n^1}(x_1) \chi_{[-L_n/2, L_n/2]}(x_2)\|_{L^2(\R^2)}\le \tfrac 14 \eps_n.
\end{align*}
For $0\leq j\leq 4$, let $\chi_{n,1}^j:\R\to [0,1]$ be smooth cutoff functions adapted to $I_n^1$ in the following sense
\begin{align*}
\chi_{n,1}^j(x)=\begin{cases}1, & x\in [c_n^1-L_n+\tfrac {10-2j}{\eta_n}DM_nT, c_n^1-\tfrac {10-2j}{\eta_n}DM_nT],\\
0, & x\in (-\infty, c_n^1-L_n+\tfrac {10-2j-1}{\eta_n}DM_nT)\cup(c_n^1-\tfrac {10-2j-1}{\eta_n}DM_nT,\infty).
\end{cases}
\end{align*}

An analogous argument allows us to find
\begin{align*}
I_n^2:=[c_n^2-\tfrac {10}{\eta_n} DM_n T, c_n^2+\tfrac {10}{\eta_n}DM_nT]\subset [\tfrac{L_n}4,\tfrac{L_n}2],
\end{align*}
with the property that
\begin{align*}
\|u_{0,n}(x+L_n\Z^2) \chi_{[-L_n/2, L_n/2]}(x_1) \chi_{I_n^2}(x_2)\|_{L^2(\R^2)}\le \tfrac 14 \eps_n.
\end{align*}
For $0\leq j\leq 4$, we define cutoff functions $\chi_{n,2}^j:\R\to [0,1]$ adapted to $I_n^2$, as above.

We then define $\chi_n^j:\R^2\to [0,1]$ via
\begin{align}
\chi_n^j(x):=\chi_{n,1}^j(x_1)\chi_{n,2}^j(x_2).
\end{align}
We list below some of the properties of $\chi_n^j$ that we will rely on in this section:
\begin{align}\label{cf:1.0}
\begin{cases}
&\|\nabla \chi_n^j \|_\infty \le \frac{\eta_n}{DM_nT} \\[0.5ex]
&\chi_n^i\chi_n^j =\chi_n^i, \mbox{ for } j>i \\[0.5ex]
&\dist (\supp\chi_n^i, \supp(1-\chi_n^j))\ge \frac 1{\eta_n}DM_nT, \mbox{ for }j>i\\[0.5ex]
&\|(1-\chi_n^j)u_{0,n}\|_{L^2(\R^2)}\le \eps_n, \mbox{ for } 0\le j\le 4.
\end{cases}
\end{align}

\subsection{Control on the solution to $\pnns$}\label{SS:8.2}

Let $\tun:\R\times\R^2\to \C$ be solutions to
\begin{align}\label{lm:1}
\begin{cases}
(i\partial_t+\Delta)\tun=\pmn F(\pmn\tun),\\
\tun(0,x)=\cha(x) u_{0,n}(x+L_n\Z^2),
\end{cases}
\end{align}
where $u_{0,n}$ are as in \eqref{8.0}.  The existence of such solutions is guaranteed by Theorem~\ref{thm:scapds}; indeed, this equation is simply a rescaling of $(\pds)$.
Rescaling \eqref{E:thm:scapds} shows that
\begin{align}\label{4.1tilde}
\| \tun \|_{S(\R)} \lsm_M  1.
\end{align}
More generally, we have the following:

\begin{lem}[Control of $\tun$]\label{lm:bdtun}
For $n$ sufficiently large, the global solution $\tun$ to \eqref{lm:1} satisfies
\begin{align*}
\||\nabla|^s \tun \|_{S(\R)}\lsm_M (DM_n )^s \qtq{for all} \ 0\le s\le 1.
\end{align*}
\end{lem}

\begin{proof}
Observe first that by Bernstein's inequality and $u_{0,n}\in\mathcal H_n$,
\begin{align*}
\| \nabla \tun(0) \|_{L^2(\R^2)} \lsm \| \nabla\cha \|_{L^\infty(\R^2)} \| u_{0,n} \|_{L^2(\btn)} + \| \nabla u_{0,n} \|_{L^2(\btn)}
	\lsm (1 + DM_n) M,
\end{align*}
for $n$ large enough.  Using this, Strichartz, Bernstein, and \eqref{4.1tilde}, we deduce
\begin{align*}
\||\nabla|^s \tun \|_{S(\R)} &\lsm \| |\nabla|^s \tun(0) \|_{L^2(\R^2)}  + \bigl\| |\nabla|^s \pmn F(\pmn\tun)\bigr\|_{L^1_t L^2_x} \\
&\lsm \| \tun(0) \|_{L^2(\R^2)}^{1-s} \| \nabla \tun(0) \|_{L^2(\R^2)}^{s} + (DM_n)^s \| \tun \|_{L^3_t L^6_x}^3 \\
&\lsm_M (DM_n)^s,
\end{align*}
thus proving the lemma.
\end{proof}

\begin{lem}[Mismatch estimate in $\R^2$]\label{lm:mr}
Fix $1<p<\infty$ and let $E, F$ be two sets in $\R^2$ such that $\dist(E,F)\ge A\ge 1$. Then
\begin{align*}
\|\chi_E \pmn \chi_F\|_{L^p(\R^2)\to L^p(\R^2)}\lsm \tfrac 1{M_n A},
\end{align*}
uniformly for $D\geq 1$.
\end{lem}

\begin{proof} Note that the kernel of $\chi_E \pmn \chi_F$ is given by
\begin{align*}
(\chi_E\pmn \chi_F)(x,y)=\chi_E(x)\chi_F(y) M_n^2\check{m}_D(M_n(x-y)).
\end{align*}
Using the rapid decay of $\check m_D$, we obtain
\begin{align*}
|(\chi_E\pmn \chi_F)(x,y)|\lsm  M_n^{-1} |x-y|^{-3} \chi_E (x)\chi_F(y)
\end{align*}
(independent of $D\geq1$), and so
\begin{align*}
\sup_y\int_{\R^2}|(\chi_E\pmn \chi_F)(x,y)|\,dx&+\sup_x\int_{\R^2}|(\chi_E\pmn \chi_F)(x,y)|\,dy\\
&\lsm \int_{|x-y|\ge A} M_n^{-1} |x-y|^{-3}\, dx\lsm M_n^{-1}A^{-1}.
\end{align*}
An application of Schur's test yields the claim.
\end{proof}

\begin{lem}[Commutator estimates on $\R^2$]\label{lm:cmr}  For $0\leq j\leq 4$, we have the following commutator estimates, which hold uniformly for $D\geq 1:$
\begin{align}
\|[\chi_n^j, \pmn]\|_{L^2\to L^2}\lsm \tfrac{\eta_n}{M_n^2DT}, \label{cmr:1}\\
\|[(1-\chi_n^j)^2, \pmn]\|_{L^2\to L^2}\lsm \tfrac{\eta_n}{M_n^2DT}. \label{cmr:2}
\end{align}
\end{lem}

\begin{proof} The proof of \eqref{cmr:2} is similar to that of \eqref{cmr:1}, so we only present the details for \eqref{cmr:1}. Let
\begin{align*}
k(x,y)&:=[\chi_n^j,\pmn](x,y)=(\chi_n^j(x)-\chi_n^j(y))M_n^2 \check m_D (M_n(x-y)).
\end{align*}
We compute
\begin{align*}
\sup_y\|k(x,y)\|_{L^1(dx)}&\le \sup_y\|\nabla \chi_n^j\|_\infty M_n^2\int_{\R^2}|x-y||\check m_D(M_n(x-y))| \,dx\\
&\le \tfrac{\eta_n}{DM_n T} M_n^2 M_n^{-3}\int_{\R^2} |x| |\check m_D(x)| \,dx\lsm \tfrac{\eta_n}{DT M_n^2}.
\end{align*}
By symmetry, we also have
\begin{align*}
\sup_x\|k(x,y)\|_{L^1(dy)}\lsm \frac{\eta_n}{DTM_n^2}.
\end{align*}
An application of Schur's test yields \eqref{cmr:1}.
\end{proof}

\begin{lem}[Motion of mass for $\tun$]\label{lm:sm}
Let $\tun$ be the solution to \eqref{lm:1} and let $\chi_n^j$ be as above. Then for $1\leq j\leq 4$ we have
\begin{align*}
\|(1-\chi_n^j)\tun\|_{L_t^\infty L_x^2([-T,T]\times\R^2)}\lsm_{D,M} \eps_n.
\end{align*}
\end{lem}

\begin{proof} We need only present the proof for $j=1$.  To this end, we define
\begin{align*}
\mathcal M(t):=\int_{\R^2}|1-\chb(x)|^2 |\tun(t,x)|^2 \,dx.
\end{align*}
By construction, $\mathcal M(0)\le \eps_n^2$.  Moreover,
\begin{align}
\frac d{dt} \mathcal M(t)&=2 \Re\int_{\R^2}|1-\chb|^2  \overline{\tun} \partial _t \tun \,dx\notag\\
&=-2\Im \int_{\R^2}|1-\chb|^2 \overline{ \tun} \Delta \tun \,dx\label{sm:1}\\
&\quad+2 \Im \int_{\R^2}|1-\chb|^2 \overline{\tun} \pmn F(\pmn \tun) \,dx. \label{sm:2}
\end{align}
Integrating by parts and using Lemma~\ref{lm:bdtun}, we estimate
\begin{align*}
|\eqref{sm:1}|&=\Bigl|4\Im\int_{\R^2}(1-\chb)\overline{\tun} \nabla\chb\cdot \nabla\tun \,dx\Bigr|\\
&\lsm \|(1-\chb)\tun\|_2\cdot \tfrac{\eta_n}{DM_nT} \cdot C(M)DM_n\\
&\le C(M) \tfrac {\eta_n}T \mathcal M^{\frac 12}(t).
\end{align*}

Next we estimate \eqref{sm:2}. We write
\begin{align*}
\eqref{sm:2}&=2 \Im \int_{\R^2}F(\pmn \tun)\pmn \bigl(|1-\chb|^2 \overline{\tun} \bigr)\, dx\\
&=2\Im \int_{\R^2}F(\pmn \tun)[\pmn, (1-\chb)^2]\overline{\tun}\,dx\\
&\qquad\qquad+2\Im \int_{\R^2}F(\pmn \tun)(1-\chb)^2 \pmn \overline{\tun}\,dx\\
&=2\Im\int_{\R^2}F(\pmn \tun)[\pmn,(1-\chb)^2]\overline{\tun} \,dx.
\end{align*}
Using H\"older and Lemma \ref{lm:cmr}, we estimate
\begin{align*}
|\eqref{sm:2}| &\lsm \|[\pmn,(1-\chb)^2]\tun\|_2 \|F(\pmn\tun)\|_2\\
&\le \|[\pmn, (1-\chb)^2]\|_{L^2\to L^2}\|\tun\|_2\|\pmn \tun\|_6^3\\
&\lsm \tfrac{\eta_n}{M_n^2DT}\|\tun\|_2\|\pmn u_n\|_{\dot H^{\frac 23}}^3\\
&\le C(M)\tfrac{\eta_n}{M_n^2DT} (M_nD)^2\\
&\le C(M)\tfrac{D\eta_n} T.
\end{align*}

Putting things together we get
\begin{align*}
\tfrac d{dt}\mathcal M(t)\le C(M)\tfrac{\eta_n}T \mathcal M^{\frac 12}(t) +C(M) \tfrac{\eta_n D}T.
\end{align*}
Recalling that $\eta_n=\eps_n^2$ and $\mathcal M(0)\leq \eps_n^2$, this differential inequality then implies
\begin{align*}
\mathcal M(t)\lsm_{D,M} \eps_n^2 \qtq{for all} t\in [-T,T].
\end{align*}
This completes the proof of the lemma.
\end{proof}

\subsection{Estimates used for approximating solutions to \eqref{per1} by those to \eqref{lm:1}}\label{SS:8.3}

Throughout this subsection, we will use $K^{L_n}$ to denote the kernel of the Fourier multiplier $\pml$:
\begin{align*}
K^{L_n}(x,y) = \frac 1{L_n^2}\sum_{j\in \Z^2} e^{2\pi i(x-y)\cdot j/L_n}m_D\bigl(\tfrac j{M_n L_n}\bigr).
\end{align*}
When $D=1$, this gives the kernel of the standard Littlewood--Paley projection $P_{\leq M_n}$.  The estimates proved in this section will be uniform in $D\geq 1$, allowing them to be applied to these operators as well.

Our preferred notion of distance on $\btn=\R^2/L_n\Z^2$ is
\begin{align*}
\dist(x,y)=\dist(x_1-y_1,L_n\Z)\vee \dist(x_2-y_2,L_n\Z)
\end{align*}
The following estimate will be used repeatedly:

\begin{lem}[Basic kernel estimate]\label{lm:kernel}
Given $A\geq 1$,
\begin{align*}
\int_{\dist(x,y)\geq A}|K^{L_n}(x,y)| \,dx \lsm \tfrac 1{D M_n A}.
\end{align*}
\end{lem}

\begin{proof} By symmetry, we may restrict attention to the region where $\dist(x_1,y_1)\ge \dist(x_2,y_2)$.

Let $\Phi(j):=\frac{2\pi(x-y) \cdot j}{L_n}$. Using the identity
\begin{align}\label{ker:1}
e^{i\Phi(j)}=\frac{e^{i\Phi(j+\vec e_1)}-e^{i\Phi(j)}}{e^{i[\Phi(j+\vec e_1)-\Phi(j)]}-1}=\frac{e^{i\Phi(j+\vec e_1)}-e^{i\Phi(j)}}{e^{2\pi i(x_1-y_1)/L_n}-1},
\end{align}
we write
\begin{align*}
K^{L_n}(x,y)=\frac 1{L_n^2} \frac 1{e^{2\pi i(x_1-y_1)/L_n}-1}\sum_{j\in \Z^2}e^{i\Phi(j)}\Bigl[m_D\bigl(\tfrac{j-\vec e_1}{M_n L_n}\bigr)-m_D\bigl(\tfrac j{M_n L_n}\bigr)\Bigr].
\end{align*}
Performing two more summation by parts using \eqref{ker:1}, we obtain
\begin{align*}
K^{L_n}(x,y)&=\frac 1{L_n^2}\frac 1{(e^{2\pi i(x_1-y_1)/L_n}-1)^3}\\
&\cdot\sum_{j\in \Z^2} e^{i\Phi(j)}\Bigl [m_D\bigl(\tfrac{j-3\vec e_1}{M_n L_n}\bigr)-3 m_D\bigl(\tfrac{j-2\vec e_1}{M_n L_n}\bigr)+3 m_D\bigl(\tfrac{j-\vec e_1}{M_n L_n}\bigr)-m_D\bigl(\tfrac j{M_n L_n}\bigr)\Bigr].
\end{align*}

To continue, we note that
\begin{align*}
|e^{2\pi i(x_1-y_1)/L_n}-1|\sim\dist \bigl(\tfrac{x_1-y_1}{L_n},\Z\bigr) \gtrsim \tfrac{\dist(x,y)}{L_n}
\end{align*}
on the support of the integral we need to estimate.  Thus, using also the estimate
\begin{align*}
&\Bigl |m_D\bigl(\tfrac{j-3\vec e_1}{M_n L_n}\bigr)-3 m_D\bigl(\tfrac{j-2\vec e_1}{M_n L_n}\bigr)+3 m_D\bigl(\tfrac{j-\vec e_1}{M_n L_n}\bigr)-m_D\bigl(\tfrac j{M_n L_n}\bigr)\Bigr|\\
&\lsm \Bigl\|\partial_{\xi_1} ^3 \bigl[m_D\bigl(\tfrac \xi{M_n L_n}\bigr)\bigr]\Bigr\|_\infty \lsm (DM_n L_n)^{-3},
\end{align*}
we deduce that 
\begin{align*}
|K^{L_n}(x,y)|\lsm \frac 1{DM_n} \dist(x,y)^{-3}.
\end{align*}
Integrating in $x$ and performing a change of variables, we obtain
\begin{align*}
\int_{\dist(x,y)\geq A}|K^{L_n}(x,y)| \,dx& \lsm \tfrac 1{DM_n A}.
\end{align*}
This completes the proof of the lemma.
\end{proof}

As an immediate application of Lemma~\ref{lm:kernel} and Schur's test, we get

\begin{lem}[Mismatch estimate on the torus] \label{lm:mstrs}
Let $E, F $ be two subsets of $\btn$ such that
\begin{align}\label{mt:0}
\dist(E,F)=\inf_{x\in E, y\in F}\dist(x,y)\ge A
\end{align}
for some $A\geq 1$. Then for all $1\leq p\leq \infty$,
\begin{align*}
\|\chi_E\pml \chi_F\|_{L^p(\btn)\to L^p(\btn)}\lsm \tfrac 1{DM_n A}.
\end{align*}
\end{lem}

\begin{lem}[Commutator estimate on the torus]\label{lm:comers} Fix $0\leq j\leq 4$ and let $\chi_n^j$ be one of the cutoff functions introduced above.  Then for all $1\leq p\leq \infty$,
\begin{align}
\|[\chi_n^j, \pml ]\|_{L^p(\btn)\to L^p(\btn)}\lsm M_n^{-1}. \label{comtrs:2}
\end{align}
\end{lem}

\begin{proof}
By Schur's test, it suffices to prove
\begin{align}
\sup_y\int_{\btn}\bigl|[\chj, \pml](x,y)\bigr| \,dx+\sup_x\int_{\btn}\bigl|[\chj, \pml](x,y)\bigr| \,dy \lsm M_n^{-1}. \label{comtrs:1}
\end{align}
By symmetry, it suffices to consider the first term on the left-hand side of \eqref{comtrs:1}.

Let $H(x,y):=[\chj,\pml](x,y)=(\chj(x)-\chj(y))K^{L_n}(x,y)$ where $K^{L_n}(x,y)$ denotes the kernel of $\pml$, as previously.
We divide the integral into two regions:
\begin{align*}
\D^0=\{x:\,\dist(x,y)\le A_n\} \ \, \text{and}\ \,  \D^1=\{x:\,\dist(x,y)\geq A_n\} \ \, \text{with}\ \, A_n:=\frac {DM_nT}{\eta_n}.  
\end{align*}
Note that $A_n\to\infty$ as $n\to\infty$; in particular, $A_n\gtrsim 1$.

The estimate on $\D^1$ follows directly from $|H(x,y)|\le 2|K^{L_n}(x,y)|$ and Lemma~\ref{lm:kernel}.
For those $x\in \D^0$, we write 
\begin{align*}
K^{L_n}(x,y)
&=\sum_{j\in \Z^2}\Bigl[\tfrac 1{L_n^2} e^{2\pi i(x-y)\cdot j/L_n} m_D\bigl(\tfrac j{M_n L_n}\bigr)-\int_{\D_j} e^{2\pi i(x_i-y_i)\cdot\xi}m_D\bigl(\tfrac\xi{M_n}\bigr)\,d\xi\Bigr]\\
&+\int_{\R^2} e^{2\pi i(x-y)\cdot \xi}m_D\bigl(\tfrac{\xi}{M_n}\bigr) \,d\xi,
\end{align*}
where $\D_j$ denotes the square box centered at $j/L_n$ with side length $1/L_n$. Note 
\begin{align*}
\bigl\|\partial_\xi^2 (e^{2\pi i(x-y)\cdot\xi}m_D(\tfrac \xi{M_n}))\bigr\|_\infty\lsm A_n^2+\tfrac 1{M_n^2}\lsm A_n^2. 
\end{align*}
Using Lemma \ref{lm:diffcell} we estimate
\begin{align*}
|K^{L_n}(x,y)|&\lsm (DM_nL_n)^2 \tfrac 1{L_n^4} A_n^2 +M_n^2| \check  m_D(M_n(x-y))|\\
&\lsm (\tfrac{DM_nA_n}{L_n}) ^2 +M_n^2 |\check m_D(M_n(x-y))|.
\end{align*}
Hence for $x\in \D^0$,
\begin{align*}
|H(x,y)|&\lsm \bigl\{ 1\wedge  |x-y|\|\nabla \chj\|_\infty\bigr\}  |K^{L_n}(x,y)|\\
&\lsm (\tfrac{DM_nA_n}{L_n})^2+\tfrac 1{A_n} |x-y|M_n^2 |\check m_D(M_n(x-y))|.
\end{align*}
Integrating over $x$ and ensuring that $L_n$ is sufficiently large, we have
\begin{align*}
\int_{\D^0}|H(x,y)| dx&\lsm (\tfrac{DM_nA_n^2}{L_n})^2+\tfrac 1{M_n A_n}
\lsm \tfrac 1{M_n}. 
\end{align*}
 This completes the proof of Lemma \ref{lm:comers}. 
\end{proof}

Next we prove estimates for the difference between the Fourier multiplier $\pmn$ on $\R^2$ and the Fourier multiplier $\pml$ on $\T_n$. We remark that the bound we obtain here is not optimal, but it suffices for our purposes.  To compare such operators on these two different manifolds, we must first transfer functions between the two.  We do this via push-forward and pull-back through the natural covering map $p:\R^2\to\R^2/L_n\Z^2$:
\begin{equation}\label{covering shit}
[p_* f](x+L_n\Z^2) = \sum_{y\sim x} f(y) \qtq{and} [p^* g](x) = g(x+L_n\Z^2).
\end{equation} 
As our cutoffs $\chi_n^j:\R^2\to\R$ are supported in a single fundamental domain of the covering map and we will only be applying these operations in the presence of such cutoffs, the heavy burden of such notations is unwarranted in what follows.  Below, the transition between functions on the Euclidean space and the torus will be made without further explanation.

\begin{lem}[Closeness of $\pmn$ and $\pml$]\label{lm:cls} For $0\leq j\leq 4$ and $1\leq p\leq \infty$ we have
\begin{gather*}
\|\chi_n^j (\pmn-\pml) \chi_n^j \|_{L^p(\R^2)\to L^p(\R^2)}\lsm M_n^{-1}.
\end{gather*}
\end{lem}

\begin{proof}
By Schur's test and symmetry, it suffices to prove
\begin{align}\label{close}
\sup_y\int_{\R^2} \bigl| \chj(x)[K(x,y)-K^{L_n}(x,y)]\chj(y) \bigr|\,dx \lsm M_n^{-1}.
\end{align}
To this end, we write
\begin{align*}
H(x,y)&:= \chj(x)(K(x,y)-K^{L_n}(x,y))\chj(y)
\end{align*}
and note that on the support of $H(x,y)$, 
\begin{align*}
|x_1-y_1|\vee |x_2-y_2|\le L_n-A_n, \ \mbox{ with } A_n=\tfrac{DM_nT}{\eta_n}.
\end{align*}

Again we discuss the integral on two regions 
\begin{align*}
\D^0=\{x: |x_1-y_1|\vee|x_2-y_2|\le A_n\} 
\end{align*}
and
\begin{align*}
\D^1=\{x: A_n< |x_1-y_1|\vee|x_2-y_2|\le L_n-A_n\}. 
\end{align*}

On $\D^0$ we can write
\begin{align*}
H(x,y)=\chj(x)\chj(y)\sum_{j\in \Z^2}\int_{\D_j} e^{2\pi i(x-y)\cdot\xi}m_D(\tfrac \xi{M_n}) d\xi-\tfrac 1{L_n^2} e^{2\pi i(x-y)\cdot j/L_n} m_D(\tfrac{j}{M_n L_n}) , 
\end{align*}
where $\D_j$ denotes the square box centered at $j/L_n$ with side length $1/L_n$.  Applying Lemma \ref{lm:diffcell} we obtain
\begin{align*}
|H(x,y)|\lsm (\tfrac{DM_nA_n}{L_n})^2. 
\end{align*}
Hence 
\begin{align*}
\int_{\D^0}|H(x,y)| dx\lsm (\tfrac{DM_n A_n^2}{L_n})^2\lsm M_n^{-1}, 
\end{align*}
provided $L_n$ is sufficiently large.  On $\D^1$, we simply estimate each piece separately. 
\begin{align*}
\int_{\D^1}|\chj(x)K(x,y)\chj(y)| dx&\lsm \int_{|x-y|>A_n} |K(x,y)|dx\\
&\lsm \int_{|x-y|>A_n}M_n^2 |\check m_D(M_n(x-y))| dx\\
&\lsm M_n^2 M_n^{-10}\int_{|x-y|>A_n}|x-y|^{-10} dx\\
&\lsm M_n^{-8}A_n^{-8}\lsm M_n^{-1}. 
\end{align*}
The estimate on the piece involving $K^{L_n}$ follows directly from Lemma \ref{lm:kernel}: 
\begin{align*}
\int_{\D^1}|\chj(x)K^{L_n}(x,y)\chj(y)|dx \lsm \int_{\D^1} |K^{L_n}(x,y)| dx\lsm M_n^{-1}.
\end{align*}
Collecting all estimates together proves Lemma~\ref{lm:cls}. 
\end{proof}

\subsection{Approximating solutions to \eqref{per1} by solutions to \eqref{lm:1}} Given $u_{0,n}\in\mathcal H_n$, let $u_n$ be solutions to the initial value problem \eqref{per1} and let $\tilde u_n$ be those to \eqref{lm:1}.

Following the conventions outlined in the previous subsection, we implicitly push-forward and pull-back functions via the natural covering map $\R^2\to\R^2/L_n\Z^2$ when necessary.

Our goal is to prove the following result:

\begin{thm}[Approximation]\label{thm:app} Fix $M>0$, $D\ge D_0(M)$ and $T>0$. Let $M_n\to \infty$ and $\eps_n\to 0$. Let $L_n$ be sufficiently large depending on $D, M, T, M_n, \eps_n$. Assume $u_{0,n}\in \mathcal H_n$ with $\|u_{0,n}\|_{L^2(\btn)} \le M$. Let $u_n$ and $\tun$ be the solutions to \eqref{per1} and \eqref{lm:1}, respectively. Then
 \begin{align}\label{pert:2.0}
 \lim_{n\to \infty} \|P_{\le 2D M_n}^{L_n}(\chc \tun)-u_n\|_{S([-T,T]\times \btn)}=0.
 \end{align}
 \end{thm}

 \begin{proof}
To keep formulas within margins, we simply write
 \begin{align*}
 z_n:=P_{\le 2D M_n}^{L_n}(\chc \tun).
 \end{align*}
We will deduce \eqref{pert:2.0} as an application of the perturbation result Proposition~\ref{P:stab}.  Correspondingly, it suffices to verify the following:
 \begin{gather}
 \|z_n\|_{S([-T,T]\times \btn)}\le C(M) \mbox{ uniformly in } n.\label{per:1}\\
 \lim_{n\to \infty}\|z_n (0)-u_n(0)\|_{L^2(\btn)}=0. \label{per:2}\\
 \lim_{n\to \infty}\|(i\partial_t+\Delta)z_n-\pml F(\pml z_n)\|_{N\tp}=0.\label{per:3}
 \end{gather}

Claim \eqref{per:1} is an easy consequence of Theorem~\ref{thm:scapds}:
 \begin{align*}
 \|z_n\|_{S\tp}\lsm \|\chc \tun\|_{S\tp}\lsm \|\tun\|_{S\tr}\le C(M).
 \end{align*}

To prove \eqref{per:2}, we use the fact that $\pmll u_{0,n} = u_{0,n}$ together with \eqref{cf:1.0} and Lemma~\ref{lm:sm}, as follows:
 \begin{align*}
 \|z_n(0)-u_n(0)\|_{L^2(\btn)}
 &=\|\pmll (\chc \cha u_{0,n}-u_{0, n})\|_{L^2(\btn)}\\
 &\lsm \|u_{0,n}-\cha u_{0,n}\|_{L^2(\btn)}\lsm \eps_n=o(1) \qtq{as} n\to \infty.
 \end{align*}

It remains to verify \eqref{per:3}.  As $\pmll\pml=\pml$, direct computation gives
\begin{align*}
(i\partial_t+\Delta)z_n-\pml F(\pml z_n)&=\pmll\Bigl[2\nabla \chc \cdot \nabla \tun +\Delta \chc \tun\Bigr] \\
 &+\pmll \Bigl[\chc \pmn F(\pmn \tun)-\pml F(\pml(\chc \tun))\Bigr].
\end{align*}
Using the $L^p$-boundedness of $\pmll$, it suffices to show that the terms in square brackets converge to zero as $n\to \infty$ in $N\tp$.

Using Lemma \ref{lm:bdtun} and \eqref{cf:1.0}, we obtain
\begin{align*}
\|\nabla \chc \cdot \nabla \tun\|_{N\tp}&\le  \|\nabla \chc \cdot \nabla \tun\|_{L^1_tL^2_x\tr}\\
&\le  T \|\nabla \chc\|_{L^\infty_x(\R^2)} \|\nabla \tun \|_{L^2_x(\R^2)}\\
&\le C(M)T\tfrac{\eta_n}{DM_n T} DM_n= o(1) \qtq{as} n\to \infty.
\end{align*}
and
\begin{align*}
\|\Delta \chc \tun\|_{N\tp}&\le T\|\Delta \chc\|_{L^\infty_x(\R^2)}\|\tun\|_{L^2_x(\R^2)}\\
&\le C(M)\bigl(\tfrac{\eta_n}{DM_nT}\bigr)^2 T = o(1) \qtq{as} n\to \infty.
\end{align*}

To estimate the remaining term, we decompose it as follows:
\begin{align}
\chc \pmn F(\pmn \tun)&-\pml F(\pml(\chc \tun))\notag\\
&=\chc \pmn \bigl[F(\pmn \tun)-F(\pmn(\chc\tun))\bigr]\label{per:6}\\
&\quad+\chc \pmn ( 1- \chd ) F(\pmn(\chc\tun))\label{per:7}\\
&\quad+\chc\pmn\chd \bigl[F(\pmn(\chc\tun))- F(\pml(\chc \tun))\bigr]\label{per:8}\\
&\quad+\chc\bigl(\pmn-\pml\bigr)\chd F(\pml(\chc\tun)) \label{per:9}\\
&\quad+ [\chc,\pml]\chd F(\pml (\chc\tun)) \label{per:10a}\\
&\quad+ \pml(\chc-1) F(\pml (\chc\tun)). \label{per:10b}
% &\quad+\chc\pml\chd F(\pml (\chc\tun))-\pml F(\pml(\chc\tun)). \label{per:10}
\end{align}

To estimate \eqref{per:6}, we use H\"older, Theorem~\ref{thm:scapds}, and Lemma~\ref{lm:sm}:
\begin{align*}
\|\eqref{per:6}\|_{N\tp}
&\lsm \|F(\pmn \tun)-F(\pmn(\chc\tun))\|_{L^{\frac 32}_tL^{\frac 65}_x\tr}\\
&\lsm \|(1-\chc)\tun\|_{L^\infty_t L^2_x\tr}\|\tun\|_{L^3_tL^6_x\tr}^2\\
&\le C(M) \eps_n= o(1) \qtq{as} n\to \infty.
\end{align*}

We turn next to \eqref{per:7}. Using the fact that
\begin{align*}
\dist(\supp\chc, \supp(1-\chd))\ge \tfrac 1{\eta_n}DM_n T,
\end{align*}
Lemma \ref{lm:mr} and Theorem~\ref{thm:scapds} yield
\begin{align*}
\|\eqref{per:7}&\|_{N\tp}\\
&\lsm \|\chc \pmn(1-\chd)\|_{L^2(\R^2)\to L^2(\R^2)} \|F(\pmn(\chc \tun))\|_{L^1_tL^2_x\tr}\\
&\lsm \tfrac{\eta_n}{DM_n^2T}\|\tun\|_{L^3_tL^6_x\tr}^3\le C(M)\tfrac{\eta_n}{DM_n^2T} = o(1) \qtq{as} n\to \infty.
\end{align*}

Consider now \eqref{per:8}. Using \eqref{cf:1.0}, Lemma~\ref{lm:cls}, and Theorem~\ref{thm:scapds}, we estimate
\begin{align*}
\|\eqref{per:8}&\|_{N\tp}\\
&\lsm \|\chd(\pmn-\pml)\chd \|_{L^2(\R^2)\to L^2(\R^2)} \cdot\|\chc \tun\|_{L^\infty_t L^2_x \tr}\\
&\qquad\cdot \bigl(\|\pmn (\chc \tun)\|_{L^3_tL^6_x\tr}^2+\| \che \pml(\chc\tun)\|_{L^3_tL^6_x\tr}^2\bigr)\\
&\le M_n^{-1} C(M)= o(1) \qtq{as} n\to \infty.
\end{align*}

We turn now to \eqref{per:9}. Using Lemma~\ref{lm:cls} and Theorem~\ref{thm:scapds}, we get
\begin{align*}
\|\eqref{per:9}&\|_{N\tp}\\
%&\le \|\chd(\pmn-\pml)\chd \cdot \che F(\pml(\chc\tun))\|_{L^1_tL^2_x\tr}\\
&\lsm \|\chd(\pmn-\pml)\chd\|_{L^2(\R^2)\to L^2(\R^2)}\|\che F(\pml(\chc\tun))\|_{L^1_tL^2_x\tr}\\
&\lsm M_n^{-1}\|\tun\|_{L^3_tL^6_x\tr}^3\le C(M)M_n^{-1}= o(1) \qtq{as} n\to \infty.
\end{align*}

For \eqref{per:10a}, we use the commutator estimate from Lemma~\ref{lm:comers}:
\begin{align*}
\|\eqref{per:10a}&\|_{N\tp}\\
&\lsm \|[\chc,\pml] \|_{L^2(\btn)\to L^2(\btn)} \|\chi_n^4 \pml (\chc\tun)\|_{L^3_t L^6_x\tr}^3\\
&\leq M_n^{-1}C(M)= o(1) \qtq{as} n\to \infty.
\end{align*}

This leaves us to estimate only \eqref{per:10b}.  Writing $\tun=\chb\tun+(1-\chb)\tun$ and employing Lemmas~\ref{lm:mstrs} and~\ref{lm:sm} shows
\begin{align*}
\|\eqref{per:10b} &\|_{N\tp} \\
&\lsm \|(1-\chc)\pml\chb \|_{L^2(\btn)\to L^2(\btn)} \|\tun\|_{L^3_tL^6_x\tr}^3 \\
&\qquad + \|(1-\chb)\tun\|_{L^\infty_tL^2_x\tr} \|\tun\|_{L^3_tL^6_x\tr}^2 \\
&\leq C(M)\tfrac{\eta_n}{D^2M_n^2T} + C(M)\eps_n= o(1) \qtq{as} n\to \infty.
\end{align*}

This completes the proof of Theorem~\ref{thm:app}.
\end{proof}

\section{Proofs of Theorems~\ref{thm:nsqz}, \ref{T:fad}, and \ref{T:dad}.}\label{S:9}

We begin this section with a proof of Theorem~\ref{T:dad}; the proofs of the other two theorems will then be built on this foundation.

\begin{proof}[Proof of Theorem~\ref{T:dad}]
By using a diagonal argument, it suffices to prove the theorem for $t\in[-T,T]$ with $T$ arbitrary but fixed.  We choose $D\geq D_0(M)$ as dictated by Theorem~\ref{thm:scapds} and then $L_n\to\infty$ as dictated by Theorem~\ref{thm:app}.

Next, we define solutions $u_n$ to (NLS) and $\tilde u_n$ to
\begin{align}
(i\partial_t+\Delta) \tun=\pmn F(\pmn \tun) \quad \text{posed on $\R^2$}
\end{align}
with initial data
$$
u_n(0,x)=\tun(0,x)=\cha(x) v_{n}(0,x+L_n\Z) \in L^2(\R^2).
$$
Recall that $\cha$ is defined in Subsection~8.1; it depends on $v_n$, and so the map from $v_n(0)$ to $u_n(0)$ is nonlinear.  Note that existence of the solutions $u_n$ and $\tilde u_n$ is guaranteed by Theorems~\ref{T:Dodson} and~\ref{thm:scapds}, respectively.

We claim that the conclusion of Theorem~\ref{T:dad} holds for this choice of $u_n$.  To see this, fix a compactly supported $\ell\in L^2(\R^2)$. Without loss of generality, we may pass to a subsequence and assume that
$$
u_n(0)=\tun(0)\rightharpoonup u_{\infty,0} \quad\text{weakly in $L^2(\R^2)$}.
$$
Let $u_\infty$ be the solution to (NLS) with initial data $u_{\infty,0}$ at time $t=0$.

By Theorems~\ref{BG-like} and~\ref{thm:wc} we have
$$
\bigl| \langle \ell, \tilde u_n(t) \rangle - \langle \ell, u_n(t) \rangle\bigr|
\leq \bigl| \langle \ell, \tilde u_n(t) \rangle - \langle \ell, u_\infty(t) \rangle\bigr| +
\bigl| \langle \ell, u_n(t) \rangle - \langle \ell, u_\infty(t) \rangle \bigr| \to 0
$$
as $n\to\infty$.  Thus, it remains to verify that
\begin{align}\label{1.666}
\bigl| \langle p_* \ell, v_n(t) \rangle - \langle \ell, \tilde u_n(t) \rangle \bigr| = o(1) \qtq{as} n\to \infty.
\end{align}

To finish the proof, we now explain why \eqref{1.666} follows from Theorem~\ref{thm:app}, which guarantees that
\begin{align*}%\label{main:2'}
\lim_{n\to \infty}\|\pmll (\chc \tun) - v_n\|_{L^\infty_t L^2_x\tp} = 0.
\end{align*}
As $\ell$ has compact support, we have $\chc \ell=p_*\ell=\ell$ for $n$ sufficiently large.  Hence, by the triangle inequality, Lemma~\ref{lm:cls}, and the Dominated Convergence Theorem,
\begin{align}
\|\ell - \chc &\pmll \ell\|_{L^2_x} \notag \\
&\le \|\chc(1-P_{\le 2DM_n})\ell \|_{L^2_x}+\|\chc (P_{\le 2DM_n}-\pmll)\chc \ell\|_{L^2_x} \label{again41.5}\\
&\lsm \|P_{> 2DM_n}\ell\|_2+M_n^{-1}= o(1) \qtq{as} n\to \infty. \notag
\end{align}
Thus, combining all these observations, we have
\begin{align*}
\text{LHS\eqref{1.666}} &\leq \|\ell\|_{L^2}\|\pmll (\chc \tun(t)) - v_n(t)\|_{L^2_x}
+ \|\tilde u_n\|_{L^\infty_t L^2_x} \|\ell - \chc \pmll \ell\|_{L^2_x} \\
&= o(1) \qtq{as} n\to \infty.
\end{align*}
This proves \eqref{1.666} and so Theorem~\ref{T:dad}.
\end{proof}

\begin{proof}[Proof of Theorem~\ref{thm:nsqz}]
Fix parameters $z_*\in L^2(\R^2)$, $l\in L^2(\R^2)$ with $\|l\|_2=1$, $\alpha \in \C$, $0<r<R<\infty$, and $T>0$.  Let $M:=\|z_*\|_2 + R$ and choose $M_n\to \infty$.  We then choose $D$ and $L_n\to\infty$ as required by Theorem~\ref{T:dad}.  Fix $\delta>0$ so that $R-r>8\delta$.

By density, we can find $\tilde z_*, \tilde l  \in C_c^\infty(\R^2)$ such that
\begin{align}\label{zzstar.1}
\| z_*-\tilde z_*\|_{L^2} \le \delta \qtq{and} \|l-\tilde l\|_{L^2} \le \delta/M  \qtq{with} \|\tilde l\|_2=1.
\end{align}

For $n$ sufficiently large, the supports of $\tilde z_*$ and $\tilde l$ are contained inside the box $[-L_n/4, L_n/4]^2$.  This has two consequences:  First, we can view $\tilde z_*$ and $\tilde l$ as functions on $\T_n=\R^2/L_n\Z^2$.  Second, setting $\chi_n^j$ to denote cutoff functions adapted to $\tilde z_*$ as in subsection~\ref{SS:cutoffs}, we have that $\chi_n^j \tilde z_* = \tilde z_*$ for all values of $j$.

We claim that
\begin{align}\label{zzstar.3}
\|\tilde z_* - P_{\leq D M_n}^{L_n} \tilde z_*\|_{L^2(\btn)} = o(1) \qtq{as} n\to\infty;
\end{align}
indeed, this follows from
\begin{align*}
\text{LHS\eqref{zzstar.3}} &\leq \|\tilde z_* - P_{\leq D M_n} \tilde z_*\|_{L^2(\R^2)} + \| (1-\chc) P_{\leq D M_n} \chb \tilde z_*\|_{L^2(\R^2)} \\
	&\quad \ \  + \|\chc [P_{\leq D M_n} - P_{\leq D M_n}^{L_n}] \chc \tilde z_*\|_{L^2(\R^2)} + \|(1-\chc) P_{\leq D M_n}^{L_n} \chb \tilde z_*\|_{L^2(\btn)}, 
\end{align*}
by applying the Dominated Convergence Theorem (in Fourier variables), Lemma~\ref{lm:mr}, Lemma~\ref{lm:cls}, and Lemma~\ref{lm:mstrs}, respectively, to these terms.  (As noted in subsection~\ref{SS:8.3}, these lemmas apply also to the standard Littlewood--Paley operators.)

Consider now the finite-dimensional Hamiltonian system
\begin{equation*}%\label{329}
(i\partial_t+\Delta) v_n=\pml F(\pml v_n) \qtq{posed on} \mathcal H_n=\{v\in L^2(\btn): \, P_{>2D M_n}^{L_n} v=0\}.
\end{equation*}
By Gromov's symplectic non-squeezing theorem, there exist witnesses contradicting any assertion of squeezing for this system.  In particular, there are solutions $v_n$ with
\begin{align}\label{vn non}
v_{n}(0) \in B_{\mathcal H_n}(P_{\leq D M_n}^{L_n} \tilde z_*, R-4\delta) \qtq{and} |\langle \tilde l, v_n(T)\rangle_{L^2(\btn)} -\alpha| > r + 4\delta.
\end{align}
Note that by \eqref{zzstar.3},
$$
 \| v_n \|_{L^\infty_tL^2_x(\R\times\btn)} \leq M  \quad\text{for $n$ large.}
$$

We now apply Theorem~\ref{T:dad} to obtain a sequence of solutions $u_n$ to (NLS) which have the property that
\begin{align}\label{from1.6}
\bigl| \langle p_*f, v_n(t) \rangle - \langle f, u_n(t) \rangle\bigr| \longrightarrow 0 \quad\text{as $n\to\infty$}
\end{align}
for all $t\in\R$ and all $f\in L^2(\R^2)$ of compact support.  We then pass to a subsequence (in $n$) so that $u_n(0)$ converges weakly, writing $u_0$ for this limit and $u$ for the solution to (NLS) with initial data $u(0)=u_0$.  We will complete the proof of Theorem~\ref{thm:nsqz} by showing that $u$ is a witnesses to non-squeezing in the sense stated there.

First we show that $u_0\in B(z_*,R)$.   Writing $F$ for the set of compactly supported unit vectors in $L^2(\R^2)$, we argue as follows, using \eqref{zzstar.1}, \eqref{from1.6}, \eqref{vn non}, and finally \eqref{zzstar.3}:
\begin{align*}
\| u_0 - z_* \|_2 &\leq  \delta + \sup_{f\in F} \ \liminf_{n\to\infty} \ | \langle f , [u_n(0) - \tilde z_*]\rangle | \\
&\leq \delta + \sup_{f\in F} \ \liminf_{n\to\infty} \ | \langle f , [v_n(0) - \tilde z_*]\rangle | \\
&\leq \delta + (R-4\delta)  + \liminf_{n\to\infty} \ \|\tilde z_* - P_{\leq D M_n}^{L_n} \tilde z_*\|_{L^2} \\
&< R. 
\end{align*}

Second, using Theorem~\ref{BG-like}, then \eqref{zzstar.1}, \eqref{from1.6}, and \eqref{vn non},  we have
\begin{align*}
| \langle \ell , u(T) \rangle -\alpha | &=  \lim_{n\to\infty} | \langle \ell , u_n(T) \rangle -\alpha |
\geq - \delta + \liminf_{n\to\infty} | \langle \tilde \ell , v_n(T) \rangle -\alpha | > r,
\end{align*}
which completes the proof of Theorem~\ref{thm:nsqz}.
\end{proof}

\begin{proof}[Proof of Theorem~\ref{T:fad}]
Let $v_n$ be the solutions to \eqref{approx system} with initial data
\begin{align}\label{E:zero}
v_n(0) = P^{L_n}_{\leq DM_n} p_* \bigl(\chi_{[-L_n/4,L_n/4]} u_n(0)\bigr).
\end{align}
Trivially, we have
$$
\|v_n\|_{L^2} \leq \|u_n\|_{L^2} \leq M.
$$

If Theorem~\ref{T:fad} were to fail, that is,
\begin{align}\label{not1.5}
\limsup_{n\to\infty}  \bigl| \langle p_*\ell_0, v_n(t_0) \rangle_{L^2(\T_n)} - \langle \ell_0, u_n(t_0) \rangle_{L^2(\R^2)}\bigr| \neq 0
\end{align}
for some $\ell_0\in L^2(\R^2)$ of compact support and some $t_0\in\R$, then it fails along some subsequence where $u_n(0)$ converges weakly.  Correspondingly, there is no loss of generality in assuming that
\begin{align}\label{WLOGweak}
u_n(0) \rightharpoonup u_{\infty,0} \quad\text{weakly in $L^2(\R^2)$}
\end{align}
for some $u_{\infty,0}\in L^2(\R^2)$, which we do henceforth.  We also define $u_\infty$ to be the solution to (NLS) with initial data $u_\infty(0)=u_{\infty,0}$.

By Theorem~\ref{BG-like}, we have that $u_n(t_0)\rightharpoonup u_\infty(t_0)$; thus to falsify \eqref{not1.5} and so prove the theorem, we need only show that
\begin{align}\label{yes1.5}
\bigl| \langle p_*\ell_0, v_n(t_0) \rangle_{L^2(\T_n)} - \langle \ell_0, u_\infty(t_0) \rangle_{L^2(\R^2)}\bigr| \longrightarrow 0.
\end{align}

By Theorem~\ref{T:dad}, there are solutions $\tilde u_n$ to (NLS) so that
\begin{align}\label{per1.6}
\bigl| \langle p_*\ell, v_n(t) \rangle_{L^2(\T_n)} - \langle \ell, \tilde u_n(t) \rangle_{L^2(\R^2)}\bigr| \longrightarrow 0
\end{align}
for all $t\in\R$ and all $\ell\in L^2(\R^2)$ of compact support.  Moreover, from \eqref{E:zero}, \eqref{WLOGweak}, \eqref{again41.5}, and the $t=0$ case of \eqref{per1.6}, we see that $\tilde u_n(0) \rightharpoonup u_\infty(0)$ weakly in $L^2(\R^2)$.

We now apply Theorem~\ref{BG-like} to the sequence $\tilde u_n$ to deduce that
$\tilde u_n(t_0) \rightharpoonup u_\infty(t_0)$ weakly in $L^2(\R^2)$.  Combining this with \eqref{per1.6} yields \eqref{yes1.5} and so Theorem~\ref{T:fad}. 
\end{proof}


\begin{thebibliography}{11}
\newcommand{\msn}[1]{\href{http://www.ams.org/mathscinet-getitem?mr=#1}{\sc MR#1}}

\bibitem{AbbMajer} A. Abbondandolo and P. Majer,
\emph{A non-squeezing theorem for convex symplectic images of the Hilbert ball.} 
Calc. Var. Partial Differential Equations \textbf{54} (2015), no. 2, 1469--1506.
\msn{3396420}

\bibitem{BG99}  H. Bahouri, P. G\'erard, \emph{High frequency approximation of solutions to critical nonlinear wave equations.}
Amer. J. Math. {\bf 121} (1999), 131--175.
\msn{1705001}

\bibitem{Bourg:lattice}  J. Bourgain, \emph{Fourier transform restriction phenomena for certain lattice subsets and applications to nonlinear evolution equations. I. Schr\"odinger equations.} 
Geom. Funct. Anal. \textbf{3} (1993), no. 2, 107--156.
\msn{1209299}

\bibitem{Bourg:approx} J. Bourgain, \emph{Approximation of solutions of the cubic nonlinear Schr\"odinger equations by finite-dimensional equations and nonsqueezing properties.} 
Internat. Math. Res. Notices 1994, no. 2, 79--88.
\msn{1264931}

\bibitem{Bourg:aspects} J. Bourgain, \emph{Aspects of long time behaviour of solutions of nonlinear Hamiltonian evolution equations.} 
Geom. Funct. Anal. \textbf{5} (1995), no. 2, 105--140.
\msn{1334864}

\bibitem{Bourg:JAMS} J. Bourgain, \emph{Global wellposedness of defocusing critical nonlinear Schr\"odinger equation in the radial case.}
J. Amer. Math. Soc. \textbf{12} (1999), no. 1, 145--171.
\msn{1626257} 

\bibitem{borg:book}
J. Bourgain, \emph{Global solutions of nonlinear Schr\"odinger equations.}
American Mathematical Society Colloquium Publications, \textbf{46}. American Mathematical Society, Providence, RI, 1999.
\msn{1691575}

\bibitem{BrezisLieb}
H. Brezis and E. Lieb, \emph{A relation between pointwise convergence of functions and convergence of functionals.}
Proc. Amer. Math. Soc. \textbf{88} (1983), 486--490.

\bibitem{CazenaveWeissler} T. Cazenave and F. Weissler, \emph{The structure of solutions to the pseudo-conformally invariant nonlinear Schr\"odinger equation.} 
Proc. Roy. Soc. Edinburgh Sect. A \textbf{117} (1991), no. 3--4, 251--273.
\msn{1103294}

\bibitem{ChernoffMarsden} P. R. Chernoff and J. E. Marsden, \emph{Properties of infinite dimensional Hamiltonian systems.} 
Lecture Notes in Mathematics, Vol. 425. Springer-Verlag, Berlin-New York, 1974.
\msn{0650113}

\bibitem{ChCoTao} M. Christ, J. Colliander, and T. Tao,
\emph{Asymptotics, frequency modulation, and low regularity ill-posedness for canonical defocusing equations.}
Amer. J. Math. \textbf{125} (2003), no. 6, 1235--1293.
\msn{2018661}

\bibitem{CKSTT:JAMS}
J. Colliander, M. Keel, G. Staffilani, H. Takaoka, and T. Tao,
\emph{Sharp global well-posedness for KdV and modified KdV on $\R$ and $\T$.} J. Amer. Math. Soc. \textbf{16} (2003), 705--749.
\msn{1969209}

\bibitem{CKSTT:JFA04}
J. Colliander, M. Keel, G. Staffilani, H. Takaoka, and T. Tao,
\emph{Multilinear estimates for periodic KdV equations, and applications.} 
J. Funct. Anal. \textbf{211} (2004), no. 1, 173--218.
\msn{2054622}

\bibitem{CKSTT:squeeze}
J. Colliander, M. Keel, G. Staffilani, H. Takaoka, and T. Tao, \emph{Symplectic nonsqueezing of the Korteweg-de Vries flow.} 
Acta Math. \textbf{195} (2005), 197--252.
\msn{2233689}

\bibitem{CKSTT:gwp}
J. Colliander, M. Keel, G. Staffilani, H. Takaoka, and T. Tao,
\emph{Global well-posedness and scattering for the energy-critical nonlinear Schr\"odinger equation in $\R^3$.}
Ann. Math. \textbf{167} (2008), 767--865.

\bibitem{DacMoser}
B. Dacorogna and J. Moser, \emph{On a partial differential equation involving the Jacobian determinant.}
Ann. Inst. H. Poincaré Anal. Non Linéaire \textbf{7} (1990), no. 1, 1--26.
\msn{1046081}

\bibitem{Dodson} B. Dodson,
\emph{Global well-posedness and scattering for the defocusing, $L^2$-critical, nonlinear Schr\"odinger equation when $d=2$.}
Preprint \texttt{arXiv:1006.1375}.

\bibitem{Gromov} M. Gromov, \emph{Pseudoholomorphic curves in symplectic manifolds.} 
Invent. Math. \textbf{82} (1985), no. 2, 307--347.
\msn{0809718}

\bibitem{HofZehn}
H. Hofer and E. Zehnder,\emph{Symplectic invariants and Hamiltonian dynamics.}
Birkh\"auser Verlag, Basel, 1994.
\msn{1306732}

\bibitem{HongKwon} S. Hong and S. Kwon, \emph{Nonsqueezing property of the coupled KdV type system without Miura transform.}
Preprint \texttt{arXiv:1509.08114}.

\bibitem{KappTop} T. Kappeler and P. Topalov, \emph{Global wellposedness of KdV in $H^{-1}(\T,\R)$.} 
Duke Math. J. \textbf{135} (2006), no. 2, 327--360.
\msn{2267286}

%\bibitem{Kato} T. Kato, \emph{On nonlinear Schr\"odinger equations. II. $H^s$-solutions and unconditional well-posedness.} 
%J. Anal. Math. \textbf{67} (1995), 281--306.
%\msn{1383498}

\bibitem{KenigMerle:H1} C. Kenig and F. Merle, \emph{Global well-posedness, scattering and blow-up for the energy-critical, focusing, non-linear Schr\"odinger equation in the radial case.}
Invent. Math. \textbf{166} (2006), no. 3, 645--675.
\msn{2257393}

\bibitem{Keraani} S. Keraani, \emph{On the blow up phenomenon of the critical nonlinear Schr\"odinger equation.} 
J. Funct. Anal. \textbf{235} (2006), no. 1, 171--192.
\msn{2216444}

\bibitem{KTV} R. Killip, T. Tao, and M. Visan, \emph{The cubic nonlinear Schr\"odinger equation in two dimensions with radial data.} 
J. Eur. Math. Soc. (JEMS) \textbf{11} (2009), no. 6, 1203--1258.
\msn{2557134}

\bibitem{Kishimoto} N. Kishimoto, \emph{Remark on the periodic mass critical nonlinear Schr\"odinger equation.} 
Proc. Amer. Math. Soc. \textbf{142} (2014), no. 8, 2649--2660.
\msn{3209321}

%\bibitem{KuksinBook} S. B. Kuksin, \emph{Nearly integrable infinite-dimensional Hamiltonian systems.} 
%Lecture Notes in Mathematics, \textbf{1556}. Springer-Verlag, Berlin, 1993.
%\msn{1290785}

\bibitem{KuksinBook} S. B. Kuksin, \emph{Analysis of Hamiltonian PDEs.} 
Oxford Lecture Series in Mathematics and its Applications, \textbf{19}. Oxford University Press, Oxford, 2000. 
\msn{1857574}

\bibitem{Kuksin} S. B. Kuksin, \emph{Infinite-dimensional symplectic capacities and a squeezing theorem for Hamiltonian PDEs.} 
Comm. Math. Phys. \textbf{167} (1995), no. 3, 531--552.
\msn{1316759}

\bibitem{Kuksin2} S. B. Kuksin, \emph{On squeezing and flow of energy for nonlinear wave equations.}
Geom. Funct. Anal. \textbf{5} (1995), no. 4, 668--701.
\msn{1345018}

\bibitem{Clay} R. Killip and M. Visan, \emph{Nonlinear Schr\"odinger equations at critical regularity.}
In ``Evolution equations'', 325--437, Clay Math. Proc. \textbf{17}, Amer. Math. Soc., Providence, RI, 2013.
\msn{3098643}

\bibitem{LaxPhillips} P. Lax and R. Phillips, \emph{Scattering theory.} 
Second edition. With appendices by Cathleen S. Morawetz and Georg Schmidt. Pure and Applied Mathematics, 26. Academic Press, Inc., Boston, MA, 1989.
\msn{1037774}

\bibitem{Mendelson} D. Mendelson, \emph{Symplectic non-squeezing for the cubic nonlinear Klein-Gordon equation on $\mathbb{T}^3$}.
Preprint \texttt{arXiv:1411.3659}.

\bibitem{Mendelson:talk} D. Mendelson, \emph{Symplectic non-squeezing for the cubic nonlinear Klein-Gordon equation on $\mathbb{T}^3$}.  One hour lecture delivered August 20, 2015, as part of the MSRI workshop ``Connections for Women: Dispersive and Stochastic PDE.''

\bibitem{McDuffSalomon} D. McDuff and D. Salamon, \emph{Introduction to symplectic topology.} 
Second edition. Oxford Mathematical Monographs. The Clarendon Press, Oxford University Press, New York, 1998.
\msn{1698616}

\bibitem{MerleVega}
F. Merle and L. Vega, \emph{Compactness at blow-up time for $L^2$ solutions of the critical nonlinear Schr\"odinger equation in 2D.}
Int. Math. Res. Not. \textbf{8} (1998), 399--425.
\msn{1628235}

\bibitem{Moser} J. Moser, \emph{On the volume elements on a manifold.} 
Trans. Amer. Math. Soc. \textbf{120} 1965 286--294.
\msn{0182927}

\bibitem{Roumegoux} D. Roum\'egoux, \emph{A symplectic non-squeezing theorem for BBM equation.} 
Dyn. Partial Differ. Equ. \textbf{7} (2010), no. 4, 289--305.
\msn{2780246}

\bibitem{Tao:bi}  T. Tao, \emph{A sharp bilinear restrictions estimate for paraboloids.} Geom. Funct. Anal. \textbf{13} (2003), no. 6, 1359--1384.
\msn{2033842}

\bibitem{Matador} T. Tao, M. Visan, and X. Zhang, \emph{The nonlinear Schr\"odinger equation with combined power-type nonlinearities.} 
Comm. Partial Differential Equations \textbf{32} (2007), no. 7--9, 1281--1343.
\msn{2354495}

\bibitem{Visan:Duke} M. Visan, \emph{The defocusing energy-critical nonlinear Schr\"odinger equation
in higher dimensions.} Duke Math. J. \textbf{138} (2007), 281--374.
\msn{2318286}

\end{thebibliography}
\end{document}